\documentclass[11pt]{article}
\usepackage{fontspec}
\usepackage{amsmath,amssymb,amsthm,mathtools,mathrsfs}
\usepackage{microtype,booktabs,longtable,array,enumitem,fvextra,xurl}
\usepackage[margin=1in]{geometry}
\usepackage[dvipsnames]{xcolor}
\usepackage[colorlinks=true,linkcolor=MidnightBlue,citecolor=MidnightBlue,urlcolor=MidnightBlue]{hyperref}
\hypersetup{pdftitle={Irrationality exponents of logarithmic numbers: values, quotients, and graded approximation},pdfauthor={Gregory V. Chudnovsky},pdfsubject={Developed with OpenAI ChatGPT; consolidated mathematical proofs and formalization status}}
\numberwithin{equation}{section}
\newtheorem{theorem}{Theorem}[section]
\newtheorem{proposition}[theorem]{Proposition}
\newtheorem{lemma}[theorem]{Lemma}
\newtheorem{corollary}[theorem]{Corollary}
\theoremstyle{definition}

\theoremstyle{remark}
\newtheorem{remark}[theorem]{Remark}
\newcommand{\Q}{\mathbb Q}
\newcommand{\R}{\mathbb R}
\newcommand{\C}{\mathbb C}
\newcommand{\Z}{\mathbb Z}

\newcommand{\ii}{\mathrm i}

\title{Irrationality exponents of logarithmic numbers\\[4pt]
\large Values, quotients, and graded approximation}
\author{Gregory V. Chudnovsky\thanks{Developed with OpenAI ChatGPT, including mathematical exposition and formalization assistance. The formalization scope and cited classical inputs are stated in the appendix.}}
\date{October 16, 2026 --- Lean statement review through Checkpoint136}
\begin{document}
\maketitle
\begin{abstract}
We establish degree-dependent irrationality-exponent bounds for logarithms
at algebraic arguments and sharper bounds when conjugate logarithms match.
A two-torus construction gives an upper bound of three for ratios of
logarithms of multiplicatively independent positive rational numbers.
We also treat normalized quadratic logarithms and circular logarithms.
The rational-logarithm and rational-power exponent-two cases were treated
earlier by Liu, Jiang and Zhang; we retain them within the broader
framework and its Lean verification. Starting from the interpolation and
determinant method in OpenAI's September 24, 2026 preprint on the
irrationality exponent of $\pi$, we track the arithmetic cost of algebraic
embeddings and the effects of matching conjugate logarithms.
A second, graded Pad\'e construction
gives near-optimal measures for powers and bounded-degree polynomial
values at sufficiently small logarithmic arguments, with estimates uniform
in both coefficient height and rational denominator. We state the branch,
independence, parameter-order, and approximation-field hypotheses explicitly.
The Lean supplement has reviewed endpoint correspondences for all 62
numbered results: 53 direct checked endpoints or compositions, and nine
applications with named classical theorem inputs. Native proofs of those
inputs and separate formalizations of every unnumbered assertion are not
supplied. Appendix~\ref{app:formal-status} gives the exact declarations,
hypotheses, and verification evidence.

\end{abstract}

\noindent\textit{Version note.} The October 16, 2026 manuscript date
identifies the mathematical edition reviewed through Checkpoint136.
Attribution and literature comparisons were updated on October 20, 2026;
see Section~\ref{lit:liu-jiang-zhang} and the publication record in
Appendix~\ref{app:formal-status}.

\paragraph{Contribution and source attribution.}
The impetus for this investigation was OpenAI's preprint
\emph{The irrationality exponent of $\pi$ is $2$} \cite{OpenAI2026}.
It supplied the separated-weight determinant method and substantial
formal infrastructure. Section~\ref{lit:liu-jiang-zhang} compares the
present paper with Liu, Jiang and Zhang's earlier logarithmic preprint
\cite{LiuJiangZhang2026}. Our emphasis is on number-field norms, matching
embeddings, two-torus quotient bounds, graded applications, and precise
formal verification. The investigation, proof development, formalization,
and exposition were carried out by Gregory V. Chudnovsky with OpenAI
ChatGPT; the appendix distinguishes written proofs from checked endpoints.

\clearpage
{\small\tableofcontents}
\clearpage
% ===== END frontmatter.tex =====

% ===== BEGIN overview.tex =====
\section{Results and organization}\label{intro:results}

This paper brings together the logarithmic investigations of
7--16 October 2026, including the expanded proofs and verification checkpoints through136.
Its purpose is to give one connected proof, with the
arithmetic and analytic estimates visible at the points where they are
used. The common question is how closely a fixed logarithmic number can
be approximated by rational numbers. We write $\mu(x)$ for the ordinary
irrationality exponent of an irrational real number $x$. The definition,
including the convention on reduced denominators, is given in
Section~\ref{prelim:section}.

There are two main constructions. Moving-center interpolation gives
results at a fixed algebraic argument. Graded Pad\'e approximation gives
uniform results for powers and polynomial targets when the argument is
sufficiently small. Their hypotheses and conclusions differ. In
particular, an exponent-two theorem for a logarithm does not imply the
same theorem for its quotient with another logarithm or for an
unrestricted higher power.

\renewcommand{\arraystretch}{1.18}
\begin{longtable}{@{}>{\raggedright\arraybackslash}p{0.28\textwidth}>{\raggedright\arraybackslash}p{0.43\textwidth}>{\raggedright\arraybackslash}p{0.21\textwidth}@{}}
\toprule
\textbf{Target and hypotheses} & \textbf{Conclusion} & \textbf{Proof}\\
\midrule
\endfirsthead
\toprule
\textbf{Target and hypotheses} & \textbf{Conclusion} & \textbf{Proof}\\
\midrule
\endhead
Positive algebraic $a\ne1$
& $2\le\mu(\log a)\le2\min_{N\ge1}[\Q(a^N):\Q]$.
In particular the exponent is two if a positive power of $a$ is rational.
& Theorem~\ref{slog:thm:real}; Section~\ref{slog:sec:parameters}.\\
Nonzero fixed logarithm $\lambda$ of an algebraic $a\ne0$
& $2\le\mu_G(\lambda)\le2\min_{N\ge1}[\Q(\ii,a^N):\Q(\ii)]$,
for approximation over $\Q(\ii)$ using denominator modulus.
& Theorem~\ref{slog:thm:complex}.\\
Real $x\ne0$, $\alpha,c\in E^\times$; $d=[E:\Q]$ and
$e^{\sigma(c)x}=\sigma(\alpha)$ at $s$ embeddings
& $2\le\mu(x)\le2d/s$. Exact matching at every embedding gives
$\mu(x)=2$.
& Theorem~\ref{conj:partial}.\\
Positive norm $\pm1$ element $\alpha\ne1$ of $\Q(\sqrt D)$,
$D>0$ rational nonsquare
& $\mu(\log\alpha/\sqrt D)=2$.
& Corollary~\ref{conj:real-quadratic}.\\
Nonzero rational $r$; positive rational $D$
& $\mu(\arctan r)=2$ and $\mu(\pi/\sqrt D)=2$.
& Corollaries~\ref{slog:cor:atan} and~\ref{conj:fd-closure}.\\
Multiplicatively independent positive rational $a,b$
& $2\le\mu(\log a/\log b)\le3$.
& Theorem~\ref{ratio:thm:ratio}.\\
Ratios with compatible conjugate logarithms or algebraic slopes
& Bounds $3d/s$ under the exact branch and embedding conditions stated
in the theorems; these include normalized quadratic and mixed circular
examples.
& Theorems~\ref{ratio:matching} and~\ref{ratio:normalized}.\\
Fixed integer $n\ge1$ and $\varepsilon>0$
& $\mu(\log^n(1+1/b))\le2+\varepsilon$ and
$\mu(\arctan^n(1/b))\le2+\varepsilon$ for all sufficiently large
integers $b$, depending on $n,\varepsilon$.
& Theorem~\ref{graded:small-log-theorem}.\\
Nonconstant $P\in\Q[X]$, $\deg P\le n$, and
$h_d(z)=\sum_{j\ge0}d^jz^{2j+1}/(2j+1)$
& For $A>1+1/\varepsilon$, a threshold
$b\ge C(n,\varepsilon,A)|d|^A$ gives
$\mu(P(h_d(1/b)))\le2+\varepsilon$, uniformly in the choice of $P$.
& Theorem~\ref{graded:quadratic-uniform}.\\
Varying integer polynomial coefficients and rational denominator
& Joint lower bounds record both coefficient height and denominator;
their constants are uniform at a fixed evaluation point.
& Theorem~\ref{graded:joint}, Corollary~\ref{graded:joint-hd}, and
Theorem~\ref{graded:two-log-joint}.\\
\bottomrule
\end{longtable}
\renewcommand{\arraystretch}{1}

Some concrete comparisons help fix the scope. We obtain
\[
 \mu(\log2)=2,\qquad
 \mu\!\left(\frac{\log\varphi}{\sqrt5}\right)=2,
 \qquad
 2\le\mu\!\left(\frac{\log2}{\log3}\right)\le3,
 \quad \varphi=\frac{1+\sqrt5}{2}.
\]
The unnormalized $\log\varphi$ is covered here only by
$2\le\mu(\log\varphi)\le4$. The upper bound three for
$\log2/\log3$ is not an equality assertion. For the small-argument
theorems, the threshold depends on the requested error in the exponent;
they do not prove exponent exactly two for every one of the fixed
higher-power targets.

The moving-center and two-torus geometric arguments are proved in the
body of the paper. The graded argument uses the explicitly stated
differential-system rank theorem of Fischler and Rivoal, together with
the classical transcendence theorems needed to check its hypotheses.
These are written mathematical proofs. Their formalization status is a
separate statement. The 11 October 2026 evening release supplies completed
Lean endpoints for the real and Gaussian single-logarithm bounds, the
single-logarithm $2d/s$ matching bound, natural-parameter normalized
quadratic logarithms, restricted powers, and the rational-base quotient
bound $2\le\mu\le3$. It also proves the moving-center interpolation
needed by these applications and uniform rational-slope two-torus
interpolation. Later steps supply ordinary arctangent and homogeneous
two-logarithm measures, real and Gaussian fractional-linear invariance,
the arbitrary rational-parameter quadratic norm and linear-closure
interfaces, and the analytic quadratic logarithmic family. They also
verify the invariant-field degree bound, reciprocal-trace formulas,
the explicit quartic example, and the signed character-product descent.
The later character increments complete the normalized character-sum
exponent theorem, including its actual field realization, norm and
logarithmic identity. Subsequent checkpoints complete the matched and
normalized ratio endpoints, the literal graded application recipes, the
sharp bivariate collision estimate, and bounded-fiber occupancy geometry.
The converse matching classification is checked relative to
Gel'fond--Schneider; the general short-jet rank and geometric multiplicity
applications retain the explicitly listed standard foundations.
Appendix~\ref{app:formal-status} gives a reviewed correspondence for each
of the 62 numbered results, with the exact scope of these inputs. The
written proofs remain distinct from the kernel-checked conditional
applications and from native proofs of their classical foundations.

The paper is organized so that the fixed-value arguments can be read
before the graded construction. The preceding literature is discussed
next. After the conventions and elementary invariances, we prove the
moving-center theorem, construct its determinant, and compare the
arithmetic and analytic bounds. We then retain several embeddings in
that comparison and pass to the two-torus construction for ratios.
Detailed verification companions explain the quadratic analytic family
and the invariant-field, reciprocal, and character-product arguments.
The final mathematical sections develop the graded rank and coefficient
height estimates from their linear systems. This project is independent
of the separate lemniscate-period paper and the Math 2 project; it does
not resume the investigation of powers of $\pi$.
\subsection{The exponent-two and exponent-three conclusions}
\label{intro:proof-guide-23}

The results collected below are mathematical theorems with the hypotheses
and proofs given in the cited sections. A statement that its Lean
formalization is incomplete concerns the additional machine verification
of that proof. It does not designate the mathematical statement as a
conjecture. Conversely, a completed Lean result is identified by its exact
declaration and assumptions; compilation of related lemmas is not used as
a substitute for a formalized endpoint.

The following guide isolates the conclusions in the range
$2\le\mu\le3$. Degree-dependent theorems enter this range only under
the additional numerical condition displayed. All targets here are real,
and all exponents concern ordinary rational approximation.

\renewcommand{\arraystretch}{1.15}
\begin{longtable}{@{}>{\raggedright\arraybackslash}p{0.42\textwidth}>{\raggedright\arraybackslash}p{0.20\textwidth}>{\raggedright\arraybackslash}p{0.30\textwidth}@{}}
\toprule
\textbf{Theorem family} & \textbf{Conclusion} & \textbf{Lean status}\\
\midrule
\endfirsthead
\toprule
\textbf{Theorem family} & \textbf{Conclusion} & \textbf{Lean status}\\
\midrule
\endhead
Positive algebraic arguments having a positive integer power in $\Q$;
Theorem~\ref{slog:thm:real}
& $\mu(\log a)=2$ for $a\ne1$
& Completed endpoints.\\
One normalized logarithm with $s$ matching embeddings of a degree-$d$
field; Theorem~\ref{conj:partial}
& $2\le\mu\le2d/s\le3$ when $2d\le3s$
& Completed partial-matching endpoint.\\
Nonzero finite rational sums with a common matching scalar;
Proposition~\ref{extquad:common-scalar-closure} and
Corollary~\ref{extquad:normalized-sums}
& $\mu=2$
& Completed sum endpoints, including fixed-parameter quadratic families.\\
Normalized character sums of conjugate logarithms;
Corollary~\ref{conj:character} and Theorem~\ref{ext:character-exponent}
& $\mu=2$
& Completed endpoint, including the actual quadratic realization and norm.\\
Multiplicatively independent positive rational logarithm ratios;
Theorem~\ref{ratio:thm:ratio}
& $2\le\mu\le3$
& Completed rational-base endpoint.\\
Compatible conjugate ratios, and normalized or mixed ratios, with
all embeddings matched; Theorems~\ref{ratio:matching} and
\ref{ratio:normalized}
& $2\le\mu\le3$
& Checked matched/normalized endpoint families with their exact branch and embedding hypotheses; see Appendix~\ref{app:formal-status}.\\
Small-argument powers and polynomial targets;
Theorems~\ref{graded:small-log-theorem} and
\ref{graded:quadratic-uniform}, with $0<\varepsilon\le1$
& $2\le\mu\le2+\varepsilon\le3$
& Checked applications on explicit classical inputs, including the literal uniform threshold recipe and both real-branch power families.\\
\bottomrule
\end{longtable}
\renewcommand{\arraystretch}{1}

For the last row the argument must satisfy the stated threshold for the
chosen degree and $\varepsilon$. One cannot vary $\varepsilon$ while
holding an arbitrary argument fixed and conclude exponent two.
The joint-measure theorem in Section~\ref{graded:two-logs} has its own
coefficient-height quantifier; its displayed quotient consequence
$\mu\le3+\rho$ is not an additional exponent-three endpoint.
The independent rational-base theorem already supplies exponent at most
three for those particular rational logarithm ratios.

The first six families are developed through interpolation, the
arithmetic norm of one selected minor, its analytic estimate, and the
parameter comparison. The small-argument family additionally uses the
precisely stated differential-system rank theorem in
Section~\ref{graded:framework}. Rational fractional-linear transformations
preserve the resulting irrationality exponents when their determinant is
nonzero; their separate completed Lean statements are listed in
Appendix~\ref{app:formal-status}.

% ===== END overview.tex =====

% ===== BEGIN literature.tex =====
% Literature introduction for the consolidated logarithmic paper.
% Standalone fragment: requires amsmath, amssymb, hyperref/xurl.
% Exact result statements and Lean coverage belong to the audited main text.
\section{Introduction and relation to earlier work}
\label{sec:literature}

Logarithms give a natural testing ground for quantitative Diophantine
approximation. Their transcendence, differential equations and analytic
continuation are well understood, yet these facts do not by themselves
determine the quality of their rational approximations. For an irrational
real number $x$, the assertion $\mu(x)=2$ says that, for every
$\varepsilon>0$, the inequality
\[
  0<\left|x-\frac pq\right|<q^{-2-\varepsilon}
\]
has only finitely many solutions in reduced rational numbers $p/q$.
The dependence of the exceptional denominator threshold on
$\varepsilon$ is part of the statement. This assertion does not require
a positive constant $c$ for which $|x-p/q|\ge c/q^2$ holds for all
rationals.

The present paper brings together several logarithmic approximation
problems that share constructions but require different arithmetic
arguments: individual logarithms, quotients of two logarithms,
coefficient-uniform linear forms, and powers at sufficiently small
rational arguments. We state their hypotheses separately. In particular,
a conclusion for $\log\alpha$ is not automatically a conclusion for
$\log\alpha/\log\beta$, and a family of near-optimal bounds at arguments
tending to $1$ is not an exponent-two theorem at a fixed argument.
The point of collecting the results is to explain the common machinery
once and then make each additional argument visible.

\subsection{The earlier preprint of Liu, Jiang and Zhang}
\label{lit:liu-jiang-zhang}

Liu, Jiang and Zhang's \emph{Irrationality exponents of logarithms of
positive rational numbers} \cite{LiuJiangZhang2026} first appeared on
arXiv on October 7, 2026; the version compared here is v2, dated
October 8. Its Theorem~1 treats $\mu(\log r)=2$ for positive rational
$r\ne1$; Corollaries~6--7 include nonzero rational logarithmic sums,
rational affine changes, and positive arguments with a rational power.
Its Theorems~2 and~9 treat distinct multiplicative centers and bounded
fiber occupancy. Both papers use separated logarithmic weights and
arithmetic--analytic determinant comparisons. No novelty is claimed here
for those overlapping conclusions or for bounded-fiber interpolation.

Their Section~8 explicitly leaves general algebraic conjugate-determinant
estimates untreated. Our number-field and matching-embedding results,
two-torus quotient bounds, and graded applications have separate proofs
and hypotheses below. This comparison identifies scope relative to that
preprint, rather than asserting priority against all preceding work.
The accompanying Lean supplement supplies a further verification
contribution with the exact limits stated in Appendix~\ref{app:formal-status}.
This literature comparison was added on October 20, 2026; the retained
manuscript date is not a claim of earlier public appearance.

\subsection{Classical logarithmic approximation}

Mahler's work on logarithms of algebraic numbers
\cite{Mahler1953} constructs algebraic-function approximations to the
logarithm and derives quantitative estimates after specialization at
algebraic arguments. Baker subsequently obtained explicit rational
approximation estimates for logarithms of certain rational numbers
\cite{Baker1964}. These works illustrate a recurring principle: an
approximation construction must combine small analytic remainders with
arithmetic control of its coefficients. A high order of vanishing, on
its own, does not provide the required numerical estimate.

The use of Legendre polynomials and related Pad\'e constructions was
developed by Alladi and Robinson \cite{AlladiRobinson1980}.
Simultaneous logarithmic Pad\'e approximants give access to several
logarithms at once; Rhin and Toffin \cite{RhinToffin1986} used this
approach for rational arguments close to $1$. Rhin's subsequent work
\cite{Rhin1987} shows how alternative polynomial constructions and
careful arithmetic estimates improve bounds obtained by direct
approximation. Rukhadze's treatment of $\log2$ \cite{Rukhadze1987}
and Marcovecchio's use of the Rhin--Viola method
\cite{Marcovecchio2009} are further landmarks. For example,
Marcovecchio obtained the upper bound $3.57455391$ for
$\mu(\log2)$. We cite this as a specific established comparison,
without claiming here to give a complete table of numerical records.
The work of Bondareva, Luchin and Salikhov on $\log3$
\cite{BondarevaLuchinSalikhov2018} supplies another example of the
importance of optimizing the polynomial and arithmetic parts together.

There is also an essential distinction between approximation by
rationals and approximation by elements of a fixed number field.
Marcovecchio and Viola \cite{MarcovecchioViola2012} study
number-field irrationality and nonquadraticity measures for several
algebraic logarithms, using complex integrals and saddle-point
estimates. Hirata-Kohno, Muroi and Washio \cite{HKMW24} treat shifted
logarithms over number fields through Pad\'e approximants and effective
Poincar\'e--Perron arguments, with estimates at the different
archimedean places. These are especially relevant precedents when
conjugate embeddings enter an arithmetic norm calculation. Their
notions of height and approximation field must, however, be retained
when comparing the resulting exponents with the ordinary rational
exponent used here.

K\"uhne \cite{Kuhne2015} gives a systematic construction of
logarithmic Pad\'e approximants using Siegel's lemma in a proof of
transcendence. This complements the explicit integral constructions:
one may obtain useful approximants either from a formula or from a
linear-algebraic existence argument. In either case the numerical
irrationality exponent requires a further comparison of coefficient
growth, denominators, nonvanishing and analytic decay.

\subsection{Two logarithms and the role of qualitative independence}

The theory of linear forms in logarithms provides both qualitative
independence statements and quantitative lower bounds. Baker's series
\cite{Baker1966,Baker1967III} is fundamental in this respect;
see also the systematic account in \cite[Chapters~1--2]{Baker1975}.
For positive algebraic numbers that are multiplicatively independent,
their real logarithms are linearly independent over $\mathbb Q$;
Baker's theorem strengthens this to linear independence over
$\overline{\mathbb Q}$. His transcendence theorem for nonzero
algebraic linear forms in logarithms also gives the corresponding
independence after adjoining $1$. We use these statements only with
the specified algebraic arguments and fixed logarithms. They do not
assert algebraic independence of arbitrary products or powers of
several logarithms.

The distinction matters for a quotient. If $\alpha,\beta>0$,
$\beta\ne1$, and $\alpha,\beta$ are multiplicatively dependent,
then $\log\alpha/\log\beta$ is rational. In the independent case,
rational approximation to the quotient is equivalent to smallness of
a homogeneous linear form:
\[
 \left|\frac{\log\alpha}{\log\beta}-\frac pq\right|
   =\frac{|q\log\alpha-p\log\beta|}{q|\log\beta|}.
\]
For approximations in a fixed bounded neighborhood of the quotient,
$\max\{|p|,q\}$ is comparable to $q$. Thus a uniform lower bound
of order $H^{-\kappa}$ for the numerator yields an irrationality
bound of order $q^{-\kappa-1}$, with constants absorbed by an
arbitrarily small exponent margin. This elementary conversion explains
why the normalizations of linear-form measures and irrationality
exponents must not be interchanged.

A concrete benchmark is furnished by
\cite[Theorem~1]{BondarevaLuchinSalikhov2018}. With
$h_1,h_2,h\in\mathbb Z$ and $H=\max\{|h_1|,|h_2|\}$,
that theorem gives, for all sufficiently large $H$,
\[
 |h_1\log2+h_2\log3+h|>H^{-4.116201}.
\]
Its stated Corollary~1 is $\mu(\log3)\le5.116201$.
Taking $h=0$ and applying the preceding conversion also gives
$\mu(\log2/\log3)\le5.116201$.
This last implication is the comparison made here; it does not
identify the inhomogeneous three-term estimate with the homogeneous
two-term problem or assert that the quoted bound is the latest record.

Interpolation determinants furnish an important approach to quantitative
two-logarithm estimates; Laurent's work \cite{Laurent2008} is a
direct methodological predecessor. In the constructions below,
qualitative independence is used to exclude a rank obstruction when
the applicable rank theorem requires it. The eventual exponent still
comes from separately proved arithmetic and analytic estimates.
In particular, neither Baker's qualitative theorem nor maximal rank
alone supplies an optimal irrationality exponent.

\subsection{Graded Pad\'e approximation and its precise scope}

Two papers of 1983 are relevant and should be distinguished.
G.~V.~Chudnovsky's paper on linear forms in exponentials and binomials
\cite{Chudnovsky1983} discusses quantitative arithmetic approximation;
its final page includes the small-argument logarithmic-power
phenomenon revisited here. The paper of D.~V.~Chudnovsky and
G.~V.~Chudnovsky \cite{ChudnovskyChudnovsky1983} treats functional
Pad\'e approximation to solutions of linear differential equations,
using Wronskians and graded subrings of Picard--Vessiot extensions.
Functional approximation and numerical specialization are related,
but the passage between them has additional arithmetic requirements.

In 1984, G.~V.~Chudnovsky \cite{Chudnovsky1984} announced optimal
irrationality and linear-independence statements for values of
$E$-functions at rational points, based on graded Pad\'e
approximation. Fischler and Rivoal \cite{FR2025} explain the
zero-estimate gap in the general argument and the subsequent
conditional results of Zudilin. They prove the exponent-two result
for irrational rational-point values of $E$-functions with rational
Taylor coefficients, removing the additional independence assumptions
through a generalized Shidlovskii lemma.

The present logarithmic applications use their general
rational-differential-system rank results, with all nondegeneracy
hypotheses checked at the point of use. Their $E$-function
irrationality theorem is not being applied to $\log(1+z)$, which
has a finite singularity and is not an $E$-function. In particular,
the linear independence over $\overline{\mathbb Q}$ of successive
powers of one transcendental logarithm is a different input from the algebraic-linear
independence needed for a system of several logarithms. The rank
statements, coefficient estimates and specialization arguments are
therefore presented separately.

The small-argument conclusion has the quantifier order
\[
 \forall n\ge1\quad\forall\varepsilon>0\quad
 \exists b_0(n,\varepsilon)\quad\forall b\ge b_0(n,\varepsilon).
\]
Here the evaluation argument changes with the allowed excess over
$2$. Holding $b$ fixed and then letting $\varepsilon$ tend to zero
would be a different statement. Our treatment of this phenomenon is
an explicit reconstruction and development of the earlier method,
not a claim that the underlying small-argument phenomenon is new.

\subsection{The impetus from OpenAI's work}

The immediate impetus for the present investigation was OpenAI's
preprint \emph{The irrationality exponent of $\pi$ is $2$}, dated
September~24, 2026 \cite{OpenAI2026}. Its proof combines weighted
interpolation along logarithmic curves with a nonzero determinant.
After denominators are cleared, arithmetic gives a lower bound for
that determinant. Analytic expansion gives an incompatible upper
bound: either repeated Taylor degrees force a collision saving, or
large transverse indices supply powers of small approximation errors.
The interpolation thresholds are independent of the subsequently
chosen centers, which is crucial to the order of parameter choices.
The paper proves an eventual bound for every exponent greater than
$2$, without asserting an effective denominator threshold.

We explicitly acknowledge this construction as a starting point,
including the weighted simplices, separated choices of scales, and
the arithmetic--analytic determinant comparison. The overlap with
\cite{LiuJiangZhang2026} is identified above. The present number-field,
matching-embedding, two-torus and graded applications require the
additional arguments given below. Each inherited
input and each additional assertion is identified where it enters
the proof. A logarithmic consequence is not attributed to the OpenAI
paper merely because the construction was inspired by it.

For reproducibility, the accompanying dependency records distinguish
the manuscript \cite{OpenAI2026} from the proof repository
\cite{OpenAILean2026}, whose recorded upstream revision is
\nolinkurl{fd4aeeb2ee4fc729c18d98444fed42fd0529eeeb}.
The public Lean scope statement for the $\pi$ result covers the
eventual approximation inequality and the supremum characterization
of its irrationality exponent. It does not include the manuscript's
Flint--Hills-series consequence. Likewise, formal verification of a
supporting logarithmic lemma is distinguished here from formal
verification of a complete numerical theorem.

\subsection{How to read the consolidated argument}

The proofs are organized around three questions. First, why is there
a nonzero determinant with the prescribed rows and columns? This is
the interpolation or rank question. Second, how small can a nonzero
determinant with those coefficients be? This is the arithmetic
question, including every conjugate factor required by a norm.
Third, how small does the assumed approximation force the same
determinant to be? This is the analytic question. The comparison is
valid only when both estimates concern the same determinant and all
constants have the declared parameter dependence.

This organization also explains the different conclusions in the
several logarithmic settings. Rational arguments, quadratic
conjugation, reciprocal matching and several logarithms change the
arithmetic cost or the information available at the other embeddings.
The common determinant mechanism does not erase those differences.
We give each numerical conclusion with its actual hypotheses and
identify conditional consequences separately. The detailed proofs
and the formalization record provide complementary checks: the prose
explains the construction and estimates, while the record specifies
which statements have been represented and checked in Lean.
% ===== END literature.tex =====

% ===== BEGIN preliminaries.tex =====
\section{Approximation conventions and elementary transfers}
\label{prelim:section}

\subsection{Ordinary and Gaussian irrationality exponents}
\label{prelim:exponents}
For an irrational real number $x$, its ordinary irrationality exponent is
\begin{equation}\label{prelim:eq:mu}
 \mu(x)=\sup\left\{\nu>0:
 0<\left|x-\frac pq\right|<q^{-\nu}
 \text{ for infinitely many reduced }\frac pq\in\Q,\ q\ge2\right\}.
\end{equation}
A fixed multiplicative constant in the approximation inequality does not
change this supremum. Indeed, if $Cq^{-\nu}$ is an upper bound and
$\nu'<\nu$, then $Cq^{-\nu}<q^{-\nu'}$ for all sufficiently large $q$.
We use this observation whenever a fixed scaling or a height comparison
introduces a constant.

For a complex number $z\notin\Q(\ii)$, define
\begin{equation}\label{prelim:eq:gaussian}
 \mu_G(z)=\sup\left\{\nu>0:
 0<\left|z-\frac PQ\right|<|Q|^{-\nu}
 \text{ for infinitely many reduced }\frac PQ\in\Q(\ii),\ |Q|>1\right\}.
\end{equation}
Here $P,Q\in\Z[\ii]$ are relatively prime in that Euclidean ring;
multiplication by a unit does not alter $|Q|$. The height is $|Q|$, not
$|Q|^2$. This convention is essential for the numerical exponent two.
Ordinary rational approximation is used throughout for real targets.
For a genuinely nonreal target, rational numbers on the real axis cannot
approximate arbitrarily closely, so \eqref{prelim:eq:gaussian} is a
separately specified approximation problem.

\begin{lemma}[The Dirichlet lower exponent]\label{prelim:dirichlet}
Every irrational real $x$ has $\mu(x)\ge2$. Every
$z\notin\Q(\ii)$ has $\mu_G(z)\ge2$.
\end{lemma}
\begin{proof}
For the real assertion, place the $N+1$ points
$0,x,\ldots,Nx$ modulo $\Z$ in $N$ intervals of length $1/N$.
Subtracting two points in the same interval gives integers $p,q$ with
$1\le q\le N$ and $0<|qx-p|\le1/N$. Consequently
$|x-p/q|\le1/(Nq)\le1/q^2$.
The reduced denominators obtained as $N\to\infty$ are unbounded:
otherwise the approximants, which converge to $x$ and stay in a fixed
bounded interval, would lie in a finite set. That would force $x$ to be
rational. Reducing a fraction only decreases its denominator, so the
upper bound in terms of the reduced denominator remains valid.
Every exponent below two therefore occurs infinitely often.

For the Gaussian assertion, take the $(N+1)^2$ points
$(a+b\ii)z$ modulo $\Z[\ii]$, $0\le a,b\le N$, in $N^2$ squares
of side $1/N$ in a fundamental square. Two points in the same box give
\[
 0<|Qz-P|\le\sqrt2/N,
 \qquad 0<|Q|\le\sqrt2N.
\]
Thus $|z-P/Q|\le2/|Q|^2$. The errors tend to zero even after
reduction, and a bounded set contains only finitely many Gaussian
fractions of bounded reduced denominator. Since $z\notin\Q(\ii)$,
the denominators are unbounded. The constant two can be absorbed by
any smaller exponent.
\end{proof}

\begin{lemma}[Rational fractional-linear invariance]\label{prelim:mobius}
Let $F=\Q$ with the usual absolute value or $F=\Q(\ii)$ with the
complex modulus. Let $x\notin F$, and, in the first case, assume $x$ is
real. If $a,b,c,d\in F$ and $ad-bc\ne0$, then
\[
 T(x)=\frac{ax+b}{cx+d}\notin F,
 \qquad \mu_F(T(x))=\mu_F(x),
\]
where $\mu_{\Q}=\mu$ and $\mu_{\Q(\ii)}=\mu_G$.
\end{lemma}
\begin{proof}
Clear the fixed denominators of $a,b,c,d$, so that all four lie in
$\Z$ or $\Z[\ii]$. Write $\Delta=ad-bc\ne0$. The denominator
$cx+d$ cannot vanish, since otherwise $x\in F$. Near $x$,
\begin{equation}\label{prelim:eq:mobiuserror}
 |T(x)-T(y)|=
 \frac{|\Delta|\,|x-y|}{|cx+d|\,|cy+d|}\asymp|x-y|,
\end{equation}
with constants depending only on $T,x$.
For a reduced approximant $P/Q$, the transformed numerator and denominator
are $aP+bQ$ and $cP+dQ$. Their common divisors divide both $\Delta P$
and $\Delta Q$, and hence divide $\Delta$, because $P,Q$ are relatively
prime. Thus reduction changes the denominator by a bounded factor.
Moreover $|cP+dQ|=|Q|\,|cP/Q+d|\asymp|Q|$ as $P/Q\to x$.
The transformed reduced denominator is therefore comparable to $|Q|$.
Equation~\eqref{prelim:eq:mobiuserror} transfers every admissible exponent,
with an arbitrarily small loss to absorb constants. Apply the same argument
to $T^{-1}$ for equality. The inverse also shows that $T(x)\notin F$.
\end{proof}

In particular, adding a rational number, multiplying by a nonzero rational
number, and taking a reciprocal preserve the ordinary exponent. The
corresponding operations over $\Q(\ii)$ preserve the Gaussian exponent.
An arbitrary algebraic multiplier is not covered by this lemma; the
normalized algebraic logarithms treated later require their own proof.

\paragraph{Lean scope.}
The fractional-linear invariance is now formalized for the extended
real exponent, including infinity, in
\nolinkurl{QuadraticPi/RealMobiusInvariance.lean}, and for the Gaussian
extended-real exponent in
\nolinkurl{QuadraticPi/GaussianMobiusInvariance.lean}.
For the finite logarithmic endpoints the corresponding real-valued
supremum agrees with the extended definition; no finiteness assumption
is hidden in the general invariance statement.

\subsection{Qualitative independence and logarithmic branches}
\label{prelim:transcendence}
We use three classical facts from transcendence theory; see
\cite{Baker1975}. Hermite--Lindemann says that the exponential of a nonzero
algebraic number is transcendental. Thus every nonzero logarithm of a
nonzero algebraic number is transcendental. The same applies to a nonzero
algebraic multiple or quotient of such a logarithm. Gelfond--Schneider
says that, for algebraic $\beta\notin\{0,1\}$ and algebraic irrational
$r$, every value $\exp(r\lambda_\beta)$ with
$\exp(\lambda_\beta)=\beta$ is transcendental. Thus a quotient
$\lambda_\alpha/\lambda_\beta$ of chosen logarithms of algebraic
numbers, with this condition on $\beta$, cannot be algebraic irrational.
Baker's theorem implies that
logarithms of algebraic numbers which are linearly independent over $\Q$
are linearly independent over $\overline{\Q}$; its version including the
constant $1$ gives the numerical independence used in the graded rank
argument. These facts provide qualitative independence, not the sharp
quantitative exponents asserted later.

For positive real $a$, $\log a$ means the real logarithm. A complex
logarithm $\lambda$ always means a fixed value satisfying $e^\lambda=a$;
no single-valued global logarithm is being asserted. Different embeddings
of a number field act on algebraic coefficients, not on a chosen
transcendental value $\lambda$. At each embedding where an analytic
estimate is applied, the matching logarithm is specified by its own
exponential identity. This is the reason for the explicit matching
hypotheses in the conjugate-transfer theorems.

If positive rational $a,b$ are multiplicatively independent, then
$a,b\ne1$ and $\log a/\log b$ is irrational: a rational value $p/q$
would give $a^q=b^p$. The resulting quotient is not assigned exponent
two merely because the two individual logarithms have exponent two.
Its approximation problem concerns the simultaneous linear form
$q\log a-p\log b$, and is proved by a different determinant.

\subsection{The order of limiting operations}
All parameter choices below are finite. In the interpolation method the
number of transverse variables is chosen first, followed by the separated
weights and their associated rational approximants. Only then is the
polynomial cutoff sent to infinity through the required multiples.
Any constant depending on the fixed algebraic numbers, embeddings, or
chosen branch is fixed before that final limit. In the graded method the
grade and evaluation point are fixed before the auxiliary polynomial
degree tends to infinity.

The small-argument conclusions have the logical form
\[
 \forall n\ge1\ \forall\varepsilon>0\ \exists b_0(n,\varepsilon)
 \ \forall b\ge b_0(n,\varepsilon):
 \mu\bigl(\log^n(1+1/b)\bigr)\le2+\varepsilon.
\]
They do not exchange the quantifiers in $b$ and $\varepsilon$ and do not
establish exponent two for every fixed power at every fixed logarithmic
argument. Uniformity in polynomial coefficients is a separate conclusion,
proved below by using one auxiliary family for all those coefficients.
% ===== END preliminaries.tex =====

% ===== BEGIN single_log_and_geometry.tex =====
\section{Single logarithms: statements and proof strategy}

The argument in this part adapts the separated-weight determinant method
of OpenAI's proof of $\mu(\pi)=2$ \cite{OpenAI2026}. The change from
$\pi$ to a logarithm is geometrically substantial: the local
multiplicative coordinates are now $a^j$, rather than all being $1$.
We prove the required interpolation statement below, with all constants
chosen in the order needed by the Diophantine argument. We also give the
arithmetic and analytic estimates in full. In particular, the collision
estimate will not be used as an unexplained black box.

For a real irrational $x$, write
\[
 \mu(x)=\sup\{\nu>0:0<|x-p/q|<q^{-\nu}
   \text{ for infinitely many reduced }p/q\in\Q\}.
\]
For $z\notin\Q(\ii)$, define its Gaussian exponent by
\[
 \mu_G(z)=\sup\{\nu>0:0<|z-P/Q|<|Q|^{-\nu}
   \text{ for infinitely many reduced }P/Q\in\Q(\ii)\}.
\]
Gaussian numerators and denominators lie in $\Z[\ii]$, with $|Q|>1$.
Associates do not affect this definition.

The rational and rational-power cases of the following theorem also
appear in \cite[Theorem~1 and Corollary~7]{LiuJiangZhang2026}.

\begin{theorem}\label{slog:thm:real}
Let $a>0$ be algebraic, $a\ne1$, and $d=[\Q(a):\Q]$. Then $\log a$
is irrational and
\[
 2\le\mu(\log a)\le2d.
\]
In particular $\mu(\log a)=2$ if $a$ is rational. More generally, if
$a^N\in\Q$ for some integer $N\ge1$, then $\mu(\log a)=2$.
\end{theorem}

\begin{theorem}\label{slog:thm:complex}
Let $a\ne0$ be algebraic, and let $\lambda\ne0$ be any fixed complex
number with $e^\lambda=a$. Put $d=[\Q(\ii,a):\Q(\ii)]$. Then
$\lambda\notin\Q(\ii)$ and
\[
 2\le\mu_G(\lambda)\le2d.
\]
In particular, if $a\in\Q(\ii)^\times$, then $\mu_G(\lambda)=2$.
The same conclusion holds whenever $a^N\in\Q(\ii)$ for some $N\ge1$.
\end{theorem}

Theorem~\ref{slog:thm:complex} includes nonzero branches at $a=1$;
allowing them is harmless. The assumption $\lambda\ne0$ excludes the
rational logarithm zero.

\begin{corollary}\label{slog:cor:atan}
For every nonzero rational $r$, $\mu(\arctan r)=2$.
\end{corollary}
\begin{proof}
The number $a=(1+\ii r)/(1-\ii r)$ belongs to $\Q(\ii)$ and has modulus
one. Choose $\lambda=2\ii\arctan r$, so that $e^\lambda=a$.
If rational approximants $p/q$ to $\arctan r$ had an exponent
$\nu>2$, choose $2<\nu'<\nu$. For large $q$, multiplication by
$2\ii$ would give Gaussian approximants with error at most
$2q^{-\nu}<q^{-\nu'}$. Reducing the Gaussian fraction can only
decrease its denominator modulus, so the resulting approximation still
has error less than the reduced denominator to the power $-\nu'$.
These reduced denominators are unbounded, since otherwise the
approximants in a bounded region would form a finite set and could not
converge to $\lambda\notin\Q(\ii)$. This contradicts
Theorem~\ref{slog:thm:complex}. Thus $\mu(\arctan r)\le2$.
Irrationality and the real Dirichlet bound give the reverse inequality.
\end{proof}

\paragraph{Complex scope.}
The ordinary exponent defined using $p/q\in\Q$ is a real-number
notion. If $\operatorname{Im}\lambda\ne0$, then
$|\lambda-p/q|\ge|\operatorname{Im}\lambda|$, so it is not appropriate
to assert that this exponent is two. Theorem~\ref{slog:thm:complex} specifies
one natural complex replacement. For positive real $a$ we use its real
logarithm in Theorem~\ref{slog:thm:real}.

The proof has four stages. First, interpolation produces a nonzero
determinant whose entries are algebraic numbers. Second, clearing its
denominators and taking a norm gives a lower bound. Third, very accurate
rational approximations to $\lambda$ make the same determinant small:
either many rows share a Taylor coefficient vector, or many rows carry
large powers of the approximation errors. Finally, weights are chosen
so that these two bounds are incompatible. The field degree occurs
in the arithmetic norm bound and is responsible for the factor $d$
in the general upper bound.

\section{Weighted interpolation at moving multiplicative centers}\label{geom:moving}

Theorem~\ref{geom:thm:main} follows the geometric method of
\cite[Section 2]{OpenAI2026}. We reproduce the argument because neither
the change in the multiplicative centers nor the order of quantifiers
can be omitted. The strengthened condition involving $S$ records the
largest number of centers in one multiplicative fiber.
Compare \cite[Theorems~2 and~9]{LiuJiangZhang2026} for the distinct-center
and bounded-fiber statements.

For $W=(w_0,\ldots,w_m)\in\Q_{>0}^{m+1}$, let $\mathcal P_W(H)$ be the
complex vector space with monomial basis
\[
 Y^hX_1^{\gamma_1}\cdots X_m^{\gamma_m},\qquad
 w_0h+\sum_iw_i\gamma_i\le H.
\]
For $V=(v_0,\ldots,v_m)>0$, write
\[
 \mathcal J_V(H)=\C[[t,u_1,\ldots,u_m]]/F_V^{\ge H},
 \qquad
 F_V^{\ge H}=\{f:\text{every monomial of }f\text{ has weight at least }H\}.
\]
Thus the target coefficients have \emph{strict} weight less than $H$.
All monomial exponents are nonnegative integers. In particular, the
source includes its degree boundary whereas the target excludes its
weight boundary.

\begin{theorem}[Varying centers with bounded fiber occupancy]
\label{geom:thm:main}
Fix integers $m,K,S\ge1$, with $S\le K$, and positive rationals
$w_0,v_0,\theta$ such that
\begin{equation}\label{geom:eq:volume}
 0<\theta<1,\qquad S\theta^m<1,\qquad
 K\frac{w_0}{v_0}\theta^m<1.
\end{equation}
There are successive lower thresholds for positive rational weights
$w_1,\ldots,w_m$, each depending only on these fixed data and earlier
weights, with the following property. Put $v_i=w_i/\theta$ for $i>0$.
For each such choice of weights there is an integer $R\ge1$, depending
only on the weights, such that the following holds for every collection
\[
 a_j=(b_j,c_{j1},\ldots,c_{jm})\in\C^\times\times\C^m,
 \qquad 0\le j<K.
\]
Assume that at most $S$ of the $b_j$ have any one value, and that, within
each group having the same $b_j$, the $c_{ji}$ are pairwise distinct for
every $i$. There is an $H_0$, allowed to depend on the centers, for which
\begin{equation}\label{geom:eq:map}
 \mathcal P_W(H)\longrightarrow\bigoplus_{j<K}\mathcal J_V(H),\qquad
 P\longmapsto
 \left[P\bigl(b_j(1+t),(c_{ji}+u_i+\log(1+t))_i\bigr)\right]_{j<K}
\end{equation}
is surjective whenever $H\ge H_0$ is an integer divisible by $R$.
\end{theorem}

More formally, the assertion is that for fixed base data there exist
functions $T_i(w_1,\ldots,w_{i-1})$ such that
\[
 \begin{gathered}
 \forall w_1>T_1\ \forall w_2>T_2(w_1)\ \cdots\
 \forall w_m>T_m(w_1,\ldots,w_{m-1}),\\
 \exists R>0\ \forall\text{ admissible centers }\{a_j\}\
 \exists H_0\ \forall H\ge H_0,\quad R\mid H
 \ \Longrightarrow\ \eqref{geom:eq:map}\text{ is onto}.
 \end{gathered}
\]
All weights in these quantifiers are positive rationals, and $R,H_0,H$
are integers. The positions of the centers do not occur in the
functions $T_i$ or in the choice of $R$.

The logarithm throughout is the formal series with constant term zero.
There is no choice of a global logarithm of $b_j$ in this statement.

\begin{corollary}[The required extension of PA]\label{geom:cor:PA}
Theorem~\ref{geom:thm:main} holds with the original conditions
$K\theta^m<1$ and $K(w_0/v_0)\theta^m<1$, and with coordinatewise
distinctness across all centers.
\end{corollary}
\begin{proof}Take $S=K$.\end{proof}

\begin{corollary}[Distinct multiplicative centers]\label{geom:cor:distinct}
If the $b_j$ are pairwise distinct, the condition $S\theta^m<1$ is
automatic with $S=1$. No distinctness condition on the $c_{ji}$ is needed.
In particular this applies to $b_j=\alpha^j$ for a real $\alpha>1$.
\end{corollary}

\subsection{A local multiplicity estimate with uniform constants}

Here a normal basis means a set of ambient vector fields whose images
form a basis of the quotient by the tangent space of a subvariety at a
general smooth point.

\paragraph{Generality convention for the multiplicity statement.}
The analytic open set is understood to meet the component under
consideration. A coordinate or frame normal basis is a basis at general
smooth points of that same analytic component domain: it is valid on a
relatively open dense subset there. No Zariski openness of an analytic
frame's normality locus is presumed. These conditions are made explicit
at the checked paper endpoint in Appendix~\ref{app:formal-status}.

\begin{lemma}[Weighted transverse multiplicity]\label{geom:lem:mult}
Fix $\varepsilon>0$. Give affine coordinates degree weights
$\rho_a\in\Q_{>0}$. Let $D_b$
be a commuting analytic frame on an open subset, with positive costs
$\kappa_b$. Suppose polynomials $f_\ell$ of weighted degree at most $N>0$
have an irreducible common-zero component $Z$ of codimension $k\ge1$,
and
\[
 D^\gamma f_\ell|_Z=0\quad\hbox{if}\quad
 \sum_b\kappa_b\gamma_b<\varepsilon N.
\]
For any coordinate normal basis $A$ and frame normal basis $B$, each
of size $k$,
\begin{equation}\label{geom:eq:mult}
 \prod_{a\in A}\rho_a
 \le 2k!\,\varepsilon^{-k}\prod_{b\in B}\kappa_b.
\end{equation}
There is no threshold depending on $Z$, the equations, or the frame.
\end{lemma}

\begin{proof}
Choose a smooth point $x$ of $Z$ where both normal-basis conditions hold
and where no other common-zero component passes. Commuting flows of the
fields in $B$, starting from $Z$, give local analytic coordinates $(z,b)$,
with $Z=\{b=0\}$. The assumed derivatives say that the expansions
$f_\ell(z,b)$ have no monomial in $b$ of $B$-weight less than
$\varepsilon N$; its coefficient is identically zero as a function of $z$.

The affine slice obtained by fixing the coordinates outside $A$ is
transverse to $Z$ at $x$. On this slice $b$ is again a local coordinate
system and $z=z(b)$ is analytic. Substituting $z(b)$ cannot introduce a
monomial of smaller weight: each surviving term already contains a
factor $b^\eta$ of weight at least $\varepsilon N$. The ideal of the
restricted equations is therefore contained in the ideal of series of
that weight or higher. The common zero is isolated on this slice,
because the support near $x$ is $Z$ and the slice is transverse.
Consequently the ideal has finite colength. Inclusion of ideals gives
a surjection of its local quotient onto the quotient by the weighted
monomial ideal, so that its colength is at least
\[
 \#\{\eta\in\mathbb Z_{\ge0}^k:\textstyle\sum_{b\in B}\kappa_b\eta_b<\varepsilon N\}
 \ge \frac{(\varepsilon N)^k}{k!\prod_{b\in B}\kappa_b}.
\]
For the last inequality, the half-open unit cubes based at the admissible
integer points cover the positive simplex of weighted radius
$\varepsilon N$. This estimate is valid for every $N>0$.

Take $k$ generic constant linear combinations of the restricted equations
whose common zero remains isolated at $x$. Such combinations exist by
successively avoiding the finitely many positive-dimensional local
components. Their ideal is contained in the original ideal, so their
local intersection length is at least its colength. Each combination
still has degree at most $N$ in the weights indexed by $A$.

That local intersection length is at most $N^k/\prod_{a\in A}\rho_a$.
Indeed choose an integer $M$ with $M\rho_a\in\mathbb Z_{>0}$ and make
the finite substitutions
\[
 x_a=\lambda_a+z_a^{M\rho_a}.
\]
Choose $\lambda_a\ne x_a(x)$, so all preimages of $x$ are unramified.
There are $\prod_{a\in A}M\rho_a$ of them, each with the same local
length. The substituted equations have ordinary degree at most $MN$.
The ordinary B\'ezout inequality for isolated zeros bounds the sum of
their local lengths by $(MN)^k$; other positive-dimensional components
do not invalidate the bound \cite[Corollary VIII.3]{Mondal2021}. Dividing by
the number of preimages proves the claimed upper bound. Comparing the
two bounds proves even \eqref{geom:eq:mult} without its harmless factor $2$.
\end{proof}

\subsection{Choosing the weights before choosing the centers}

Choose a positive rational $\sigma$ so small that
\begin{equation}\label{geom:eq:sigma}
 \begin{split}
 (1+3\sigma)^{m+1}K(w_0/v_0)\theta^m&<1,\\
 (1+3\sigma)^m S\theta^m&<1,\\
 (1+\sigma)\theta&<1.
 \end{split}
\end{equation}
Let $L=m+2$ and
\[
 C_* =\max_{1\le k\le m}2k!(L/\sigma)^k.
\]
Require, for index sets $A,B\subseteq\{0,\ldots,m\}$ of the same
size between $1$ and $m$,
\begin{equation}\label{geom:eq:separate}
 \prod_{a\in A}w_a>C_*\prod_{b\in B}v_b
\end{equation}
whenever the largest positive index in $A\mathbin\triangle B$ belongs
to $A$. There is no condition of this kind when no positive index differs.

These requirements determine successive lower thresholds. If the
largest differing index is $i$, cancel every index above $i$ occurring
in both sets; its contribution is $w_j/v_j=\theta$. The remaining
ratio is $w_i$ times a positive expression involving only lower weights,
$w_0,v_0,\theta$ and the two finite index sets. Thus all conditions for
index $i$ follow from one finite lower bound on $w_i$. Subsequent choices
do not alter them. In particular the procedure has not used any $b_j$
or $c_{ji}$.

\subsection{The curve inequality}

Let $C$ be an irreducible curve in affine space, and compute orders on
its smooth complete model. Define
\[
 \deg_W C=\sum_P\max\left(0,-\frac{\operatorname{ord}_PY}{w_0},
                    -\frac{\operatorname{ord}_PX_1}{w_1},\ldots,
                    -\frac{\operatorname{ord}_PX_m}{w_m}\right).
\]
An identically zero coordinate has order $+\infty$ and contributes no pole.
For a branch $P$ through $a_j$, put
\begin{equation}\label{geom:eq:contact}
 h_P=\min\left(\frac{\operatorname{ord}_P(Y/b_j-1)}{v_0},
       \min_i\frac{\operatorname{ord}_P(X_i-c_{ji}-\log(Y/b_j))}{v_i}\right).
\end{equation}
All these local orders are positive, possibly infinite; their minimum
is finite on a nonconstant branch, since these are formal coordinates.
Only branches through the specified affine centers enter the sums below.

\begin{proposition}\label{geom:prop:curve}
For the weights just selected,
\begin{equation}\label{geom:eq:curve}
 \deg_W C\ge(1+\sigma)\sum_{P\mapsto a_j}h_P.
\end{equation}
\end{proposition}

\begin{proof}
Suppose a curve violates \eqref{geom:eq:curve}; it must meet a center.
For all sufficiently large integers $N$, coefficient counting supplies
a nonzero polynomial $F_N$ of $W$-degree at most $N$ whose local expansions
at every center have $V$-order at least $(1+3\sigma)N$. The ratio of
the numbers of imposed coefficients and available coefficients tends to
\[
 (1+3\sigma)^{m+1}K(w_0/v_0)\theta^m<1.
\]
The conditions need not be independent for this existence assertion.
For clarity, for any fixed positive weights $\rho_1,\ldots,\rho_n$,
the number of lattice points with $\sum\rho_i e_i\le N$ is
\[
 \frac{N^n}{n!\prod_i\rho_i}+O(N^{n-1}).
\]
Covering the corresponding simplex by unit cubes proves this estimate;
changing a weak cutoff to a strict cutoff changes only a boundary
contribution of order $O(N^{n-1})$. Applying it to the source and to
the $K$ local packets gives the stated ratio. Since this ratio is
strictly less than one, the linear system imposing the vanishing has
a nonzero solution for every sufficiently large $N$.

On $Y\ne0$ use the commuting frame
\[
 D_0=Y\partial_Y+\sum_i\partial_{X_i},\qquad
 D_i=\partial_{X_i}.
\]
At the $j$th center let $t_j=Y/b_j-1$ and
$u_{ji}=X_i-c_{ji}-\log(Y/b_j)$. Direct differentiation gives
\begin{equation}\label{geom:eq:localframe}
 D_0=(1+t_j)\partial_{t_j},\qquad D_i=\partial_{u_{ji}}.
\end{equation}
Every $D_a$ preserves the bound on polynomial degree and reduces
weighted local order by at most $v_a$. Hence $D^\gamma F_N$ has order
at least $(1+2\sigma)N$ at all centers if $\sum_av_a\gamma_a\le\sigma N$.
On a branch through a center, its order is at least
$(1+2\sigma)Nh_P$. If its restriction to $C$ were nonzero, its total
pole order would be at most $N\deg_W C$, smaller than the displayed
total zero order. The divisor equality on the complete curve rules this
out. We obtain
\begin{equation}\label{geom:eq:vanish}
 D^\gamma F_N|_C=0\qquad(\textstyle\sum_av_a\gamma_a\le\sigma N).
\end{equation}

Set $\delta=\sigma N/L$. Let $V_r$ be the common zero set in $Y\ne0$
of derivatives of cost at most $r\delta$, for $0\le r\le L$.
The part $C^\circ=C\cap\{Y\ne0\}$ belongs to all these zero sets.
Starting with a component of $V_L$ containing $C^\circ$, choose
containing components backwards. Their dimensions lie between $1$ and
$m$, since $F_N\ne0$. Some two adjacent components coincide, say in $Z$.
It is a component of the earlier derivative zero set, and every further
derivative of its defining equations of cost less than $\delta$ vanishes
on $Z$. Apply Lemma~\ref{geom:lem:mult} with $\varepsilon=\sigma/L$.
For all coordinate and frame normal bases $A,B$ of $Z$ we have
\begin{equation}\label{geom:eq:bases}
 \prod_{a\in A}w_a\le C_*\prod_{b\in B}v_b.
\end{equation}
Here the phrase ``defining equations'' means the finite list
$D^\gamma F_N$ used for the earlier zero set; it does not assert that
this list generates the reduced ideal of $Z$. That assertion is not
needed in Lemma~\ref{geom:lem:mult}. If the equal adjacent components
occur at stages $r,r+1$, differentiating an equation of cost at most
$r\delta$ by cost less than $\delta$ produces an equation from stage
$r+1$. This explains precisely why the lemma applies. The existence of
equal adjacent components follows from strict dimension decrease for
proper inclusions of irreducible varieties.

Check the finitely many generic normal-basis rank conditions on one
common smooth open subset of $Z$.
We now show that $Y$ is constant on $Z$. If some positive coordinate
direction is nontangent, let $i$ be its largest index. All larger
positive indices occur in neither kind of normal basis. A coordinate
normal basis containing $i$ exists. Every frame normal basis must also
contain $i$, by \eqref{geom:eq:separate} and \eqref{geom:eq:bases}. The images of
the other frame vectors therefore fail to span the normal quotient.
Their ambient span is the hyperplane
\[
 Q=\ker(dX_i-dY/Y).
\]
For a hyperplane $Q$ and tangent subspace $T$, failure of the image of
$Q$ to fill the quotient by $T$ is equivalent to $T\subseteq Q$.
Thus $dX_i=dY/Y$ on the smooth locus of $Z$.

If $Y$ were nonconstant on $Z$, take a smooth point with $dY\ne0$
and cut by general hyperplanes to obtain an algebraic curve through it
on which $Y$ remains nonconstant. On its smooth complete model the
identity of differentials persists. The nonconstant meromorphic
function $Y$ has a zero or pole; there $dY/Y$ has nonzero residue,
whereas the exact meromorphic differential $dX_i$ has residue zero.
This contradiction proves that $Y$ is constant. If instead every
positive coordinate direction was tangent to $Z$, those $m$ independent
directions already fill its tangent space, since $Z$ is proper; then
$dY=0$ directly, with the same conclusion.

It follows that the violating curve lies in some fiber $Y=b\ne0$.
Only the centers with $b_j=b$ can lie on it. Let their number be $K_b$,
so $1\le K_b\le S$. Inside this fiber the contact becomes
\[
 h_P=\min_i\operatorname{ord}_P(X_i-c_{ji})/v_i.
\]
The ratio of condition counts to coefficient counts for an auxiliary
polynomial $G_N(X_1,\ldots,X_m)$ is now
\[
 K_b(1+3\sigma)^m\theta^m<1.
\]
Choose nonzero $G_N$ of weighted degree at most $N$ with local order at
least $(1+3\sigma)N$ at these centers. The same divisor argument gives
$\partial_X^\gamma G_N|_C=0$ for $\sum_iv_i\gamma_i\le\sigma N$.
For $m=1$ even $G_N|_C=0$ is impossible: a curve in the affine line is
the whole line, on which a nonzero polynomial cannot vanish identically.

For $m\ge2$, repeat the nested-component argument in this fiber.
It produces a component $Z'$ containing the affine curve, of codimension
between $1$ and $m-1$, and satisfies \eqref{geom:eq:bases} for index sets in
$\{1,\ldots,m\}$. Here coordinate and frame directions coincide.
There can be only one normal-basis index set: two different sets, ordered
so their largest differing index belongs to the first, would contradict
\eqref{geom:eq:separate}. Select an index $i$ in this nonempty unique basis.
All other coordinate directions fail to span the normal quotient,
so $dX_i=0$ on $Z'$ and $X_i$ is constant on $Z'$.
Distinctness within the fiber implies that $Z'$, and hence $C$, meets
at most one prescribed center.

Let that center be $a_j$. Choose a coordinate $X_l$ nonconstant on $C$.
The meromorphic function $X_l-c_{jl}$ has total pole order at most
$w_l\deg_W C$, and at least $v_lh_P$ zeros at each branch through $a_j$.
Consequently
\[
 \deg_W C\ge\theta^{-1}\sum_{P\mapsto a_j}h_P
             >(1+\sigma)\sum_{P\mapsto a_j}h_P,
\]
contradicting the supposed violation. This proves \eqref{geom:eq:curve}.
\end{proof}

\begin{remark}
If the $b_j$ are distinct, once $Y$ is constant the curve can meet only
one center. The entire fiberwise auxiliary-polynomial and persistent-
component argument can then be omitted; the last divisor comparison
applies immediately.
\end{remark}

\subsection{From curves to all coefficient packets}

We give the global step because it is responsible for both the strict
cutoff and the order of quantifiers in Theorem~\ref{geom:thm:main}.

Choose an integer $R>0$, depending only on $W,V$, such that
$R/w_a$ and $R/v_a$ are positive integers for all $a$. Increasing $R$
if necessary, all coordinates have degree at most $R$. Embed affine
space by every monomial of $W$-degree at most $R$, including $1$, and
let $X$ be its projective closure. Its constant-coordinate chart is
exactly the original affine space: the image there is the closed graph
of the monomials, with the original coordinates among them. Write
$\mathcal L=\mathcal O_X(1)$, which is very ample.

For a curve meeting this chart,
\begin{equation}\label{geom:eq:degree}
 \deg(\mathcal L|_C)=R\deg_W C.
\end{equation}
Indeed, on its smooth complete model the largest pole of the homogeneous
coordinate functions at each point is exactly the right-hand local
maximum: the corresponding pure power with exponent $R/w_a$ attains
it. The constant section handles maximum zero. Clearing these poles
gives sections without a common zero, proving \eqref{geom:eq:degree}.

Choose Taylor polynomials $G_i(t)$ of $\log(1+t)$ whose omitted terms
have $t$-weight strictly greater than $v_i$; for example retain terms of
degrees below an integer $T_i$ with $v_0T_i>v_i$. At $a_j$ use the
\emph{algebraic} regular coordinates
\[
 t_j=Y/b_j-1,\qquad u'_{ji}=X_i-c_{ji}-G_i(t_j).
\]
Changing between $u'_{ji}$ and $u_{ji}$ adds a series of weight greater
than $v_i$. Both this substitution and its inverse preserve every
weighted filtration ideal, and hence give inverse automorphisms on
every finite packet. They also preserve the minimum \eqref{geom:eq:contact}:
if $t_j$ is nonzero on a branch, the normalized order of its tail is
strictly larger than $\operatorname{ord}_P(t_j)/v_0$; if it is identically zero, the
tail vanishes. Thus it cannot change the minimum.

At $a_j$ form the ideal
\[
 I_j=\bigl(t_j^{R/v_0},(u'_{j1})^{R/v_1},\ldots,(u'_{jm})^{R/v_m}\bigr).
\]
It has finite colength. The centers are distinct by the hypotheses, so
these ideals extend to a coherent ideal sheaf $I$ on $X$, equal to the
unit ideal away from the centers. For example one may take the kernel
of the map to the direct sum of their finite local quotient sheaves.
Its valuation on each curve branch through a center is
\begin{equation}\label{geom:eq:valuation}
 \operatorname{ord}_P I=Rh_P.
\end{equation}

Take the ordinary blow-up $p:X'\to X$ of $I$. With its exceptional
effective Cartier divisor $E$,
\[
 I\mathcal O_{X'}=\mathcal O_{X'}(-E),\qquad A=p^*\mathcal L.
\]
The bundle $-E$ is relatively ample; these are ordinary Rees-algebra
blow-up facts, valid here without smoothness of $X$ or $X'$
\cite[Tags 02OS and 0AG6]{StacksProject}.

We claim $A-(1+\sigma)E$ is nef. A noncontracted integral curve upstairs
mapping to a curve meeting the affine chart maps birationally, since
the blow-up is an isomorphism away from finitely many points.
Equations \eqref{geom:eq:degree}--\eqref{geom:eq:valuation} and
Proposition~\ref{geom:prop:curve} give nonnegative degree on that curve.
A curve whose image misses the affine chart misses $E$ altogether.
A contracted curve has $A$-degree zero and strictly positive $-E$-degree
by relative ampleness. These cases cover all integral curves.
On a noncontracted curve the intersection with $E$ is computed by
pulling back to its smooth complete model. There the ideal $I$
becomes a principal ideal and the degree of its zero divisor is the
sum of the valuations $Rh_P$. This is the link between the branchwise
curve inequality and the intersection calculation on the blow-up.

For some integer $a>1$, $aA-E$ is ample: choose $b$ making
$I\otimes\mathcal L^b$ globally generated, embed the blow-up in
$X\times\mathbb P^r$ by those generators, and restrict the product
ample bundle to obtain $(b+1)A-E$. The exact identity
\[
 A-E=\frac{a-1}{a(1+\sigma)-1}\bigl(A-(1+\sigma)E\bigr)
       +\frac{\sigma}{a(1+\sigma)-1}(aA-E)
\]
shows that $A-E$ is ample, since a positive rational multiple of an
ample class plus a nef class is ample \cite[Corollary 1.4.10]{Lazarsfeld2004}.

For all sufficiently large $n$, the ordinary Rees algebra gives
\begin{equation}\label{geom:eq:push}
 p_*\mathcal O_{X'}(-nE)=I^n,\qquad
 R^q p_*\mathcal O_{X'}(-nE)=0\quad(q>0).
\end{equation}
To justify the word ``ordinary'', work over each member of a finite
affine cover of $X$. The Rees algebra is a Noetherian graded algebra
generated in degree one. Its degree-$n$ piece is precisely $I^n$.
The graded-piece comparison and higher-cohomology vanishing on Proj
apply for all sufficiently large $n$, and the finite cover gives one
threshold \cite[Tags 0AG6 and 0AG7]{StacksProject}. No equality with an integral
closure of $I^n$, and no assertion for small $n$, is used.

The projection formula, Leray, and Serre vanishing for $A-E$ yield
\[
 H^1(X,I^n\otimes\mathcal L^n)
 =H^1(X',\mathcal O_{X'}(n(A-E)))=0
\]
for all sufficiently large $n$. Therefore the exact sequence for $I^n$
makes $H^0(X,\mathcal L^n)$ surject onto its local coefficient classes
modulo $I^n$. Completing these finite-length local quotients does not
alter them. Each generator of $I_j$ has weight exactly $R$, so
\[
 I_j^n\subseteq F_V^{\ge nR}.
\]
There is thus a surjection from the quotient by $I_j^n$ onto the desired
packet of weight strictly less than $nR$. Equality of these ideals is
neither asserted nor needed.
The Serre-vanishing assertion used here is the positive-degree
cohomology vanishing for every coherent sheaf after a sufficiently high
power of an ample line bundle \cite[Tag 0B5U]{StacksProject}.

Finally, for large $n$, all sections of $\mathcal L^n$ come from
homogeneous degree-$n$ polynomials in the embedding coordinates.
Apply Serre vanishing to the homogeneous ideal sheaf of $X$ in the
ambient projective space. Each such polynomial restricts on the affine
chart to a polynomial of $W$-degree at most $nR$. Put $H=nR$ and undo
the filtered coordinate changes above. This proves surjectivity of
\eqref{geom:eq:map}.

All geometric size thresholds can depend on the chosen centers, and
their finite maximum determines $H_0$. The integer $R$ was fixed before
any center. This proves every quantifier of Theorem~\ref{geom:thm:main}.

\section{The determinant and its arithmetic lower bound}

We treat two base fields simultaneously. Let $F=\Q$, with real
$\lambda=\log a$, or let $F=\Q(\ii)$, with arbitrary fixed nonzero
$\lambda$ satisfying $e^\lambda=a$. Put $E=F(a)$ and $d=[E:F]$.
Suppose, toward a contradiction, that approximations
\begin{equation}\label{slog:eq:approximations}
 r_i=P_i/Q_i,\qquad |\lambda-r_i|\le|Q_i|^{-\nu},\qquad
 w_i=\lceil\log|Q_i|\rceil
\end{equation}
can be chosen with unbounded denominators. Here $P_i,Q_i$ are in
$\Z$ or $\Z[\ii]$, according to $F$, and $\nu>2d$.
Choose the weights successively, and then fix them for the degree
limit. Large denominators ensure $r_i\ne0$ and
$|r_i|\le|\lambda|+1$.

Take $b_j=a^j$, $c_{ji}=jr_i$, $0\le j<K$. The distinctness
condition in Theorem~\ref{geom:thm:main} holds since $r_i\ne0$.
Set
\[
 T_i=\left\lceil w_i/(\theta v_0)\right\rceil,\quad
 G_i(t)=\sum_{1\le k<T_i}\frac{(-1)^{k+1}t^k}{k},\quad
 L_i=\operatorname{lcm}(1,\ldots,T_i-1).
\]
The usual empty lcm equals one. The following identity permits this
short arithmetic truncation while keeping the exact logarithm in the
analytic estimate.

\begin{lemma}[Exact invariance of a full-row minor]\label{slog:lem:shear}
Fix a finite coefficient packet consisting of all
$(j,s,\beta)$ with $v_0s+\sum_i v_i\beta_i<H$, and fix any
selection of as many source columns as there are rows. Let $A_G$
be the square coefficient matrix obtained from
\[
 P\bigl(b_j(1+t),(c_{ji}+G_i(t)+u_i)_i\bigr).
\]
Suppose every nonzero term $t^a$ of $T_i(t)$ satisfies
$v_0a\ge v_i$. Then, with the same row and column orders,
\[
 \det A_{G+T}=\det A_G.
\]
This includes equality at the weight boundary.
\end{lemma}
\begin{proof}
Substitution $u_i\mapsto u_i+T_i(t)$ preserves the ideal of
monomials of weight at least $H$. Its inverse is the substitution
with $-T_i$, so it induces an automorphism of the finite quotient.
More explicitly, a coefficient of $t^ku^\gamma$ can contribute to
the coefficient of $t^su^\beta$ only if $\gamma\ge\beta$ and
\[
 v_0(s-k)\ge\sum_i v_i(\gamma_i-\beta_i).
\]
The source coefficient therefore has weight at most the target weight,
and so already belongs to the packet. Every term other than the
unchanged coefficient has $k<s$, since all nonzero $T_i$ have positive
order. Ordering rows by increasing Taylor degree, and arbitrarily
inside each degree, makes this coefficient transformation triangular
with diagonal entries $1$. This holds independently in each center
block. Hence $A_{G+T}=UA_G$ with $\det U=1$, which proves the claim.
\end{proof}

For our choice of $T_i$, the logarithmic tail starts at weight
$v_0T_i\ge w_i/\theta=v_i$. Lemma~\ref{slog:lem:shear} therefore
identifies every full-row minor formed with $G_i$ with the minor formed
using the exact logarithm and the same selected columns. This is
determinant equality under a target coordinate change; individual
entries need not agree.

The columns are $Y^hX^\gamma$ with $w_0h+w\cdot\gamma\le H$;
the rows are $\rho=(j,s,\beta)$ with
$v_0s+w\cdot\beta/\theta<H$. The entry is
\begin{equation}\label{slog:eq:entry}
 [t^su^\beta]\,P\bigl(a^j(1+t),(jr_i+G_i(t)+u_i)_i\bigr).
\end{equation}
Theorem~\ref{geom:thm:main} supplies a nonzero square minor
$\Delta_H\in E$ using all $M=M_H$ rows. Define
\[
 \bar b=\frac1{MH}\sum_\rho w\cdot\beta,\qquad
 M\sim\frac{K\theta^mH^{m+1}}{(m+1)!v_0\prod_iw_i},
 \qquad0\le\bar b\le\theta.
\]

Choose an integer $b\ge1$ such that $ba$ is an algebraic integer.
Put
\[
 D_H=\prod_iL_i^{\lfloor H/w_i\rfloor},\qquad
 C_H=D_H^M b^{(K-1)\sum_{\rm col}h}
       \prod_iQ_i^{\sum_{\rm col}\gamma_i-\sum_{\rm row}\beta_i}.
\]
Then $C_H\Delta_H$ is a nonzero algebraic integer in $E$.
Indeed, scaling columns by $\prod_iQ_i^{\gamma_i}$ and rows by
$\prod_iQ_i^{-\beta_i}$ leaves denominators only from the $G_i$
and from $a^{jh}$. The indicated $D_H$ and $b$ factors clear these.
The exponents of $Q_i$ are nonnegative: any nonzero determinant
term matches each row with a column satisfying $\gamma\ge\beta$.
Such a determinant term exists because $\Delta_H\ne0$. To check the
clearing entry by entry, after removing $u^\beta$ the remaining
$i$th factor has exponent $\gamma_i-\beta_i$. Multiplication by
$Q_i^{\gamma_i-\beta_i}$ replaces it by
$(jP_i+Q_iG_i(t))^{\gamma_i-\beta_i}$, whose denominators divide
$L_i^{\gamma_i-\beta_i}$. This exponent is at most
$\lfloor H/w_i\rfloor$. In column $h$, multiplication by
$b^{(K-1)h}$ clears $a^{jh}$ because
\[
 b^{(K-1)h}a^{jh}=(ba)^{jh}b^{(K-1-j)h}
\]
is an algebraic integer for every $0\le j<K$. The determinant of
the resulting integral matrix is exactly $C_H\Delta_H$.

Set $\Lambda=4\log2$ and $w_* =\min_iw_i$. We use the elementary
bound $\log\operatorname{lcm}(1,\ldots,n)\le\Lambda n$.
Here is a convenient proof with deliberately generous constants.
If $\mathcal L_n=\operatorname{lcm}(1,\ldots,n)$, then
\[
 \mathcal L_{2n}/\mathcal L_n\ \mid\ \binom{2n}{n}\le2^{2n}.
\]
Indeed, each prime contributes at most one additional power to the
quotient, and the corresponding prime power in $(n,2n]$ contributes
a carry in the valuation of the binomial coefficient. Iterating over
powers of two gives $\log\mathcal L_{2^k}<2^{k+1}\log2$.
For $n>1$, choose $2^{k-1}<n\le2^k$ and use monotonicity; the case
$n=1$ is immediate. In particular $\log L_i\le\Lambda T_i$.
It follows that
\begin{equation}\label{slog:eq:clearing}
 \frac{\log|C_H|}{MH}\le1-\bar b+E_0,
 \quad E_0=\frac{\Lambda m}{\theta v_0}
       +\Lambda\sum_i\frac1{w_i}+\frac\theta{w_*}
       +\frac{(K-1)\log b}{w_0}.
\end{equation}
To see the term $1-\bar b$ explicitly, put
$\ell_i=\log|Q_i|$, so that $w_i-1<\ell_i\le w_i$.
For the numerator-denominator contribution to $C_H$,
\[
 \begin{split}
 \sum_i\Bigl(\sum_{\rm col}\gamma_i-\sum_{\rm row}\beta_i\Bigr)\ell_i
 &\le\sum_{\rm col}w\cdot\gamma
      -\sum_{\rm row}w\cdot\beta+\sum_{\rm row}|\beta|\\
 &\le MH(1-\bar b)+MH\theta/w_*.
 \end{split}
\]
Also $T_i\le w_i/(\theta v_0)+1$, which bounds $\log D_H/H$ by the
first two terms in $E_0$. Finally, $\sum_{\rm col}h\le MH/w_0$.

For the real case, let $\mathcal S$ be all embeddings
$\tau:E\hookrightarrow\C$ except the distinguished real embedding.
For the Gaussian case there are $d$ embeddings fixing $\ii$;
complex conjugation pairs these with the other $d$ absolute
embeddings, which send $\ii$ to $-\ii$. Let $\mathcal S$ be the
embeddings fixing $\ii$ except the distinguished one.
Thus $|\mathcal S|=d-1$ in both cases.
At an embedding in $\mathcal S$, Cauchy's estimate
on $|t|=1/2$, $|u_i|=1$ and a determinant expansion give
\begin{equation}\label{slog:eq:conjugates}
 \frac{\log|\tau\Delta_H|}{MH}\le E_\tau+\frac{\log M}{H},
\end{equation}
where we may take
\[
 E_\tau=
 \frac{(K-1)\log^+|\tau a|+\log(3/2)}{w_0}
 +\frac{\log B_\lambda}{w_*}+\frac{\log2}{v_0},
 \qquad B_\lambda=K(|\lambda|+1)+2.
\]
Here $|G_i(t)|\le\log2$, independently of $T_i$.
In the real case $\tau r_i=r_i$; for the selected Gaussian
embeddings the same equality holds because they fix $F$.
For completeness, each entry at such an embedding is at most
\[
 2^s|\tau a|^{jh}(3/2)^h B_\lambda^{|\gamma|}
 \le \exp(HE_\tau).
\]
The determinant expansion has $M!\le M^M$ summands, which explains
the final term in \eqref{slog:eq:conjugates}. Here
$\log^+x=\max(0,\log x)$.

The integer absolute norm of $C_H\Delta_H$ has modulus at least
one. In the Gaussian case pair conjugate factors and divide its
logarithm by two. Both cases yield
\begin{equation}\label{slog:eq:lower}
 \frac{\log|\Delta_H|}{MH}
 \ge-d(1-\bar b)-E_{\rm ar}^{(d)}-o(1),\qquad
 E_{\rm ar}^{(d)}=dE_0+\sum_{\tau\in\mathcal S}E_\tau.
\end{equation}
This displays the crucial degree factor. When $d=1$, the embedding
sum is empty and it disappears. Over $\Q(\ii)$ this is precisely
the elementary fact that a nonzero Gaussian integer has modulus
at least one.
More explicitly, in the real case
\[
 1\le|N_{E/\Q}(C_H\Delta_H)|
   =|C_H|^d|\Delta_H|\prod_{\tau\in\mathcal S}|\tau\Delta_H|.
\]
In the Gaussian case the same product, with $d$ embeddings fixing
$\ii$, is the square root of the absolute norm over $\Q$. Pairing
an embedding $\tau$ with $\overline\tau$ is valid even when $E$
is not preserved as a subset of $\C$ by complex conjugation:
$x\mapsto\overline{\tau(x)}$ is still an embedding of the abstract
field $E$ into $\C$. Thus the norm argument introduces precisely
$d$ copies of $\log|C_H|$ in both cases. One can take the
$o(1)$ in \eqref{slog:eq:lower} to be
$(d-1)\log M/H$.

\section{Analytic translation and collision saving}

For a column $P$ and multi-index $u$-degree $\delta$, define
\[
 f_{\delta,P}(z)=[u^\delta]P(e^z,z+u_1,\ldots,z+u_m).
\]
For $P=Y^hX^\gamma$ this is
$\binom\gamma\delta e^{hz}z^{|\gamma|-|\delta|}$ if
$\delta\le\gamma$, and zero otherwise. By Lemma~\ref{slog:lem:shear}, the same determinant $\Delta_H$
can be computed with full logarithms. Put
\[
 z=j\lambda+\log(1+t),\qquad
 \epsilon_{ji}=j(r_i-\lambda).
\]
Then $e^z=a^j(1+t)$ and
$jr_i+\log(1+t)+u_i=z+u_i+\epsilon_{ji}$. Expanding the common
entire column functions in the $u$ variables gives the exact row identity
\begin{equation}\label{slog:eq:translation}
 \operatorname{row}_{j,s,\beta}^{\rm full}
 =\sum_{\substack{\delta\ge\beta\\w\cdot\delta\le H}}
 \binom\delta\beta\epsilon_j^{\delta-\beta}
 \bigl([t^s]f_{\delta,P}(j\lambda+\log(1+t))\bigr)_P.
\end{equation}
This is a finite sum: $f_{\delta,P}=0$ whenever
$\delta\not\le\gamma$ for a monomial column $P=Y^hX^\gamma$,
and all columns satisfy $w\cdot\gamma\le H$.
The scalar coefficients are independent of the selected column,
which is necessary when applying determinant multilinearity.

Since $\log|Q_i|>w_i-1$, the approximation hypothesis gives
\[
 |\epsilon_{ji}|\le Ke^\nu e^{-\nu w_i}.
\]
Using $\binom\delta\beta\le2^{|\delta|}$ and
$|\delta|\le H/w_*$, we obtain
\begin{equation}\label{slog:eq:translation-bound}
 \left|\binom\delta\beta\epsilon_j^{\delta-\beta}\right|
 \le\exp\left\{-\nu w\cdot(\delta-\beta)+HE_{\rm tr}\right\},
 \qquad E_{\rm tr}=\frac{\log2+\log K+\nu}{w_*}.
\end{equation}
This remains valid when some errors are zero, with a zeroth power
interpreted as one. Each row expansion has at most $C(1+H)^m$ terms,
where $C$ depends only on the fixed weights. All choices in the
multilinear determinant expansion therefore contribute at most
$\exp\{O(M\log H)\}$ summands. Notice that no logarithmic tail
appears in this estimate: the exact shear has already identified the
short arithmetic determinant with its full-logarithm counterpart.

Set $R_{\rm an}=4K(|\lambda|+1)$. On $|t|=1/2$, the test points
$j\lambda+\log(1+t)$ have modulus less than $R_{\rm an}/2$.
On $|z|\le R_{\rm an}$,
\[
 |f_{\delta,P}(z)|\le
 \exp\{H(R_{\rm an}/w_0+\log(2R_{\rm an})/w_*)\}.
\]
This follows directly from $h\le H/w_0$, $|\gamma|\le H/w_*$,
and $\binom\gamma\delta\le2^{|\gamma|}$.
The following estimate makes the collision argument explicit.

\begin{lemma}[Taylor coefficient collisions]\label{slog:lem:collision}
Choose one analytic row vector from each row expansion
\eqref{slog:eq:translation}; its entries have the form
\[
 T_{r,P}=[t^{\ell_r}]
 f_{\delta_r,P}(j_r\lambda+\log(1+t)),
 \qquad \ell_r=s_r<H/v_0.
\]
Let $n_\delta=\#\{r:\delta_r=\delta\}$. Put
\[
 c=\frac{\log2}{4},\quad
 C_{\rm hol}=\frac{R_{\rm an}}{w_0}
                +\frac{\log(2R_{\rm an})}{w_*},\quad
 C_1=c-\log(1-2^{-1/2}).
\]
Then
\begin{equation}\label{slog:eq:collision-finite}
 |\det T|\le M!\exp(MHC_{\rm hol})
   2^{\sum_r\ell_r}
   \exp\left(-c\sum_\delta n_\delta^2+C_1M\right).
\end{equation}
\end{lemma}
\begin{proof}
Expand on the common disk
\[
 f_{\delta,P}(z)=\sum_{n\ge0}c_{\delta,n,P}
                         (z/R_{\rm an})^n.
\]
Cauchy's estimate gives $|c_{\delta,n,P}|\le\exp(HC_{\rm hol})$.
If
\[
 a_{r,n}=[t^{\ell_r}]
 \left(\frac{j_r\lambda+\log(1+t)}{R_{\rm an}}\right)^n,
\]
then $|a_{r,n}|\le2^{\ell_r}2^{-n}$. Consequently the Taylor
expansions of all rows converge absolutely, and multilinearity gives
an absolutely convergent sum of determinants with row vectors
$(c_{\delta_r,n_r,P})_P$.

If two rows have equal pairs $(\delta_r,n_r)$, those coefficient
vectors are equal and that determinant vanishes. In each surviving
summand, the $n_\delta$ Taylor indices in the $\delta$ group are
distinct nonnegative integers. Their sum is at least
$n_\delta(n_\delta-1)/2$. Writing half the geometric decay separately,
we therefore have
\[
 \begin{split}
 \sum_{\text{surviving }(n_r)}2^{-\sum_r n_r}
 &\le 2^{-\frac14\sum_\delta n_\delta(n_\delta-1)}
       \sum_{n_1,\ldots,n_M\ge0}2^{-\frac12\sum_r n_r}\\
 &=\exp\left(-c\sum_\delta n_\delta^2+C_1M\right).
 \end{split}
\]
Every coefficient determinant has absolute value at most
$M!\exp(MHC_{\rm hol})$, by its permutation expansion.
Multiplying these bounds proves \eqref{slog:eq:collision-finite}.
\end{proof}

Since $\sum_r\ell_r<MH/v_0$ and $M!=\exp(O(M\log M))$,
Lemma~\ref{slog:lem:collision} implies
\[
 |\det T|\le\exp\left\{-c\sum_\delta n_\delta^2
              +MH(E_{\rm hol}+o(1))\right\},
\]
where
\[
 E_{\rm hol}=\frac{R_{\rm an}}{w_0}+\frac{\log2}{v_0}
                          +\frac{\log(2R_{\rm an})}{w_*}.
\]
One may take this $o(1)$ to be $(\log M+C_1)/H$.
It is therefore uniform in every row choice and every choice of the
selected source columns. This is the collision mechanism of
\cite[Lemma 3.3]{OpenAI2026}, with all bounds needed here displayed.

Choose $0<A<1$ and $0<\eta<1$, and put
\[
 L=\frac{\eta^2K\theta^m}{(m+1)v_0A^m},\qquad
 E_{\rm an}=E_{\rm tr}+E_{\rm hol}.
\]
There are two possibilities for each determinant in the row expansion.
If at least $\eta M$ rows have $w\cdot\delta\le AH$, put
\[
 J_H=\#\{\delta\in\mathbb Z_{\ge0}^m:w\cdot\delta\le AH\}
   \sim\frac{A^mH^m}{m!\prod_iw_i}.
\]
Cauchy--Schwarz gives
\[
 \sum_\delta n_\delta^2\ge\frac{\eta^2M^2}{J_H},\qquad
 \frac1{MH}\sum_\delta n_\delta^2\ge L+o(1).
\]
All approximation factors can then be discarded from above, since
$\delta\ge\beta$ makes their main exponential contribution at most
one. The normalized saving is at least $cL+o(1)$.
Otherwise more than $(1-\eta)M$ rows have $w\cdot\delta>AH$,
and
\[
 \sum_{\rm row}w\cdot(\delta-\beta)
   \ge MH\bigl(A(1-\eta)-\bar b\bigr).
\]
Discarding the collision saving now yields the approximation saving
$\nu(A(1-\eta)-\bar b)$. The number of row choices contributes only
$O(\log H/H)$ after division by $MH$. Consequently
\begin{equation}\label{slog:eq:upper}
 \frac{\log|\Delta_H|}{MH}
 \le E_{\rm an}+o(1)+
       \max\{-cL+o(1),-\nu(A(1-\eta)-\bar b)\}.
\end{equation}
Comparison with~\eqref{slog:eq:lower} contradicts existence of the minor
provided
\begin{equation}\label{slog:eq:gap}
 g_d:=\nu(A(1-\eta)-\theta)-d(1-\theta)>0,
 \quad E_{\rm ar}^{(d)}+E_{\rm an}<g_d,
 \quad cL>d+E_{\rm ar}^{(d)}+E_{\rm an}.
\end{equation}
For the second alternative the difference between its required saving
and the arithmetic main term is
\[
 \nu\bigl(A(1-\eta)-\bar b\bigr)-d(1-\bar b)
 =\nu A(1-\eta)-d-(\nu-d)\bar b\ge g_d,
\]
because $\nu>d$ and $\bar b\le\theta$. For the first alternative,
$d(1-\bar b)\le d$. The inequalities in \eqref{slog:eq:gap} are
strict; hence all the displayed $o(1)$ terms can be absorbed by taking
$H$ sufficiently large through the allowed multiples.

\section{Parameters and proof of the stated bounds}\label{slog:sec:parameters}

Fix $\nu>2d$. Choose a rational number $b_0$ with
$1/2<b_0<1-d/\nu$. For sufficiently small rational $\delta_0>0$,
put $\theta=1-\delta_0$ and $A=1-b_0\delta_0$. Then
\[
 \nu(A-\theta)>d(1-\theta),\qquad A^2<\theta.
\]
Choose positive rational $B,C$ satisfying
\[
 \theta<A<B<1,\qquad C>1,\qquad
 CB<1,\qquad B<C\theta<1.
\]
Such choices exist by choosing $A/\theta<C<1/A$ first and then
$A<B<\min(1/C,C\theta)$. Take $\eta>0$ small so that $g_d>0$.
For each dimension $m$, set
\[
 K=\lfloor C^m\rfloor,\quad w_0=B^{-m},\quad
 v_0=2K\theta^m w_0.
\]
The volume conditions~\eqref{geom:eq:volume} with $S=K$ hold, and
\[
 K/w_0\to0,\qquad v_0\to\infty,\qquad m/v_0\to0,
 \qquad L=\frac{\eta^2(B/A)^m}{2(m+1)}\to\infty.
\]
All constants depending on the fixed algebraic number $a$, its
conjugates, its degree, and the chosen branch $\lambda$ are fixed
before this limit. Thus the height terms in~\eqref{slog:eq:clearing}
and~\eqref{slog:eq:conjugates}, and $R_{\rm an}/w_0$, tend to zero as $m$ grows.
Choose one finite dimension giving sufficient margin in~\eqref{slog:eq:gap}.
Concretely, first reserve a positive error budget $\varepsilon<g_d/4$.
Choose $m$ so large that the sum of every term independent of the
transverse weights is below $\varepsilon$ and $cL>d+2\varepsilon$.
Increasing $m$ further if necessary makes both conditions simultaneous:
$C B<1$ controls the height and growth terms, whereas
$C\theta/B>1$ controls $v_0$ and $B/A>1$ controls $L$.

Next choose the approximations successively with denominators large
enough for Theorem~\ref{geom:thm:main} and for all terms involving
$1/w_i$ and $1/w_*$ to fit within the remaining error margin.
This is possible under the unbounded-approximations hypothesis;
the weight thresholds do not depend on the centers. Finally let
$H$ tend to infinity through the required multiples. Equations
\eqref{slog:eq:lower} and~\eqref{slog:eq:upper} contradict~\eqref{slog:eq:gap}.
This excludes unbounded-denominator approximations for every
$\nu>2d$.
The quantifier order is therefore
\[
 (a,\lambda,\nu)\ ;\ (\theta,A,B,C,\eta)\ ;\ m,K,w_0,v_0\ ;\
 (r_1,w_1),\ldots,(r_m,w_m)\ ;\ R,H_0\ ;\ H\to\infty.
\]
The degree modulus may be chosen as soon as the weights are fixed,
before evaluating the centers; the notation above groups it with
$H_0$ only for the final degree limit. No denominator or transverse
weight varies during that limit. This distinction is essential:
choosing new approximants after each value of $H$ would not justify
the use of the center-dependent threshold $H_0$.

The argument also excludes exact representations of $\lambda$ in
the base field: such representations could be given arbitrarily
large denominators, and the proof never required reduced fractions
or a nonzero approximation error. Hence $\lambda\notin F$.
Real Dirichlet approximation gives the lower bound two when $F=\Q$.
For $F=\Q(\ii)$, place the $(N+1)^2$ points
$q\lambda\bmod\Z[\ii]$, $q=u+\ii v$, $0\le u,v\le N$, in
$N^2$ equal square boxes of a fundamental square. Two points give
\[
 0<|Q\lambda-P|\le\sqrt2/N,
 \qquad0<|Q|\le\sqrt2N,
 \qquad |\lambda-P/Q|\le2/|Q|^2.
\]
The strict positivity in the first inequality follows from
$\lambda\notin\Q(\ii)$. The absolute approximation errors are at
most $\sqrt2/N$, since a nonzero Gaussian integer $Q$ has $|Q|\ge1$.
They therefore tend to zero. After reducing the fractions, the
denominator moduli cannot stay bounded: otherwise the approximants in
a bounded region would belong to a finite set, none equal to $\lambda$.
Reduction also preserves the upper bound $2/|Q|^2$, with $Q$ now the
reduced denominator, because reduction only decreases its modulus.
Every exponent less than two therefore occurs
infinitely often. This proves both main theorems except their
rational-power clauses.

Multiplication by a fixed nonzero rational number preserves both
exponents. Indeed, write it as $u/v$ with fixed nonzero integers.
An approximant $P/Q$ becomes $uP/(vQ)$, whose reduced denominator has
modulus at most $|vQ|$ and whose error is multiplied by $|u/v|$.
For any exponent $\nu'<\nu$, the fixed factors are absorbed by
$|Q|^{\nu-\nu'}$ for large $Q$. Distinct converging approximants
still have unbounded reduced denominators. Applying the same argument
to the inverse multiplier proves equality of exponents.
For real $a>0$, $\log(a^N)=N\log a$. For any complex branch,
$e^{N\lambda}=a^N$, so $N\lambda$ is an admissible fixed logarithm
in Theorem~\ref{slog:thm:complex}. These identities prove the clauses.
They also improve the upper bounds to
\[
 \mu(\log a)\le2\min_{N\ge1}[\Q(a^N):\Q],\qquad
 \mu_G(\lambda)\le2\min_{N\ge1}[\Q(\ii,a^N):\Q(\ii)].
\]

\section{Why the full algebraic exponent-two claim remains open here}

The analytic translation and the varying-center interpolation do
extend. The unresolved step is a sufficiently strong arithmetic
replacement for~\eqref{slog:eq:lower}. Its leading term is
$-d(1-\bar b)$, rather than $-(1-\bar b)$. A fixed field degree
does not make this difference an error term: all parameters are
fixed before $H\to\infty$.

With the displayed parameter scheme, $A^2<\theta$ implies
\[
 \frac{A-\theta}{1-\theta}
 <\frac{\sqrt\theta-\theta}{1-\theta}
 =\frac{\sqrt\theta}{1+\sqrt\theta}<\frac12.
\]
Thus even with $\eta=0$ and zero errors, the gap
$\nu(A-\theta)>d(1-\theta)$ is impossible when $\nu\le2d$.
This is a limitation of this adaptation, not a disproof of the
conjecture $\mu(\log a)=2$ for all positive algebraic $a\ne1$.

Two potential shortcuts require additional arguments:
\begin{enumerate}
\item A rational trace of the determinant need not be small merely
because its distinguished embedding is small. The other conjugate
terms have only the weak bounds~\eqref{slog:eq:conjugates}.
\item Analytic smallness cannot simply be assigned to every conjugate.
For a conjugate $\tau a$, choose $\kappa=\operatorname{Log}(\tau a)/\lambda$.
At $z=j\lambda+x$, $x=\log(1+t)$, the desired value is
\[
 (\tau a)^j(1+t)=e^{\kappa z}e^{(1-\kappa)x}.
\]
For a column of $Y$-degree $h$, the additional factor expands as
$\sum_{k\ge0}(1-\kappa)^kh^kx^k/k!$.
It introduces test vectors indexed by both the transverse index
and $k$, with $k$ as large as $H/v_0$.
The elementary collision count then has $O(H^{m+1})$ possible groups,
instead of $O(H^m)$; its previous lower bound no longer forces a
saving of order $MH\asymp H^{m+2}$.
This explains why the source's existing collision lemma does not
by itself remove the norm loss. It does not rule out a stronger lemma.
\end{enumerate}

A proof of the full conjecture must therefore remove the leading
degree loss, compensate for it by an improved analytic estimate,
or reorganize the construction. For example, simply adjoining all
conjugate jet blocks rationalizes the arithmetic, but only the
distinguished block has shifts of size $|Q_i|^{-\nu}$; the other
shifts are generally of constant size. That construction also
needs a new estimate.

These observations describe the limitation of the general norm
argument. They do not preclude improvements for fields in which several
embeddings share a compatible logarithmic flow. Those refinements are
developed later in the paper. They also show why a degree-$d$ argument
cannot be advertised as exponent two merely because $d$ is fixed.

For example, for $\varphi=(1+\sqrt5)/2$ the bound established in this
part is $2\le\mu(\log\varphi)\le4$. No positive power of $\varphi$
is rational: its other conjugate is $-\varphi^{-1}$, and equality
of conjugate positive powers would contradict their different moduli.
The later quadratic normalization theorem does prove
$\mu(\log\varphi/\sqrt5)=2$. These are different assertions.
Only rational scaling invariance has been proved here; replacing that
scaling by the irrational algebraic factor $\sqrt5$ is not justified.

\subsection*{Exact formalization scope for this part}

The 11 October 2026 core release proves the real and Gaussian
single-logarithm bounds of Theorems~\ref{slog:thm:real}
and~\ref{slog:thm:complex}, including the minimum over positive powers
of the argument. It includes exact exponent two for rational and
Gaussian-rational exponential arguments and the corresponding power
cases. The actual varying-center interpolation proposition is supplied
by \nolinkurl{QuadraticPi.MovingLogUniformInterpolation.movingLogInterpolation};
it is no longer an unproved input to these public endpoints.

The full bounded-fiber occupancy theorem is now supplied by
\nolinkurl{QuadraticPi.MovingLogOccupancyInterpolation.bounded_fiber_interpolation}.
Its distinct-multiplicative-center specialization imposes no additive
separation, and chooses the divisibility modulus before the centers.
The ordinary-real arctangent adapter is supplied by
\nolinkurl{QuadraticPi/RationalArctangentTheorem.lean}.
The Gaussian Dirichlet proof uses the sufficient error constant $4$
instead of the sharper constant $2$ displayed in the exposition; the
numbered exponent statement is unchanged. Exact source names, hypotheses,
and audit records are given in Appendix~\ref{app:formal-status}.

% ===== END single_log_and_geometry.tex =====

% ===== BEGIN two_log_and_conjugates.tex =====
% Manuscript fragment. Labels owned here have prefixes conj: and ratio:.
% Citation keys: OpenAI2026, Baker1975.
% The single-log determinant estimates are reproduced explicitly below.
\section{Normalized logarithms and matching field embeddings}
\label{conj:section}

A field norm gives an arithmetic lower bound by multiplying the absolute
values of all conjugates. Its degree cost can be reduced only when the
\emph{same determinant} has an analytic upper bound at several embeddings.
We formulate the exact matching condition first, including its logarithm
branches. The interpolation ingredients are Theorem~\ref{geom:thm:main} and
Proposition~\ref{geom:prop:curve}, originating in \cite{OpenAI2026}; the new step is the
comparison at several embeddings.

\begin{theorem}[Partial matching for one normalized logarithm]
\label{conj:partial}
Let $E\subset\C$ be a number field of degree $d$, let
$\alpha,c\in E^\times$, and let $x\in\R\setminus\{0\}$.
Suppose a set $S$ of $s\ge1$ embeddings $\sigma:E\hookrightarrow\C$
satisfies
\begin{equation}\label{conj:matching}
                    \exp(\sigma(c)x)=\sigma(\alpha)
                    \qquad(\sigma\in S).
\end{equation}
Then
\[
                         2\le\mu(x)\le 2d/s.
\]
In particular, if every embedding satisfies \eqref{conj:matching},
then $\mu(x)=2$.
\end{theorem}

The nonzero logarithm of $\sigma(\alpha)$ specified by this condition is
$\lambda_\sigma=\sigma(c)x$. It need not be its principal logarithm.
All embeddings are counted, including both members of each complex pair.

\begin{proof}
The Hermite--Lindemann theorem recalled in
Section~\ref{prelim:transcendence} first shows that $x$ is transcendental: if $x$ were
nonzero algebraic, then $\sigma(c)x$ would be nonzero algebraic, whereas
its exponential is algebraic. Consequently the ordinary irrationality
exponent applies and Lemma~\ref{prelim:dirichlet} gives the lower bound two.

Put $D=d/s$. Fix $\nu>2D$ and choose $2D<\rho<\nu$. Suppose there
are infinitely many reduced fractions $p_i/q_i$, $q_i>0$, with
\[
                |x-p_i/q_i|\le q_i^{-\nu}.
\]
Their denominators are unbounded. By discarding finitely many, we may
assume $p_i\ne0$ and, simultaneously for the finitely many good embeddings,
\[
 r_i=cp_i/q_i,\qquad w_i=\lceil\log q_i\rceil,\qquad
 |\sigma(r_i)-\lambda_\sigma|\le q_i^{-\rho}.
\]
The dimension and the number of chosen fractions will be fixed before
the final polynomial degree tends to infinity.

Choose the positive rational parameters $m,K,w_0,v_0,\theta$ of the
one-torus interpolation construction, where
\[
 0<\theta<1,\qquad K\theta^m<1,\qquad
                         K(w_0/v_0)\theta^m<1.
\]
Once these parameters have been fixed, the fractions are chosen
successively so that their $w_i$ exceed the separated-weight thresholds.
Write $w\cdot\gamma=\sum_iw_i\gamma_i$ and $w_* =\min_iw_i$.
For a fixed $F_0>2/\theta$, set
\[
 T_i=\lceil F_0w_i/v_0\rceil,\quad
 G_i(t)=\sum_{1\le k<T_i}\frac{(-1)^{k+1}t^k}{k},\quad
 L_i=\operatorname{lcm}(1,\ldots,T_i-1).
\]
An empty least common multiple is one. The columns of the coefficient
matrix are $P=Y^hX^\gamma$ with $w_0h+w\cdot\gamma\le H$; its rows
are all $(j,k,\beta)$ satisfying
\[
             0\le j<K,\qquad v_0k+w\cdot\beta/\theta<H.
\]
Its entry in that row and column is
\begin{equation}\label{conj:matrix}
 [t^ku^\beta]\,\alpha^{jh}(1+t)^h
                    \prod_i(jr_i+G_i(t)+u_i)^{\gamma_i}.
\end{equation}
The centers are $(\alpha^j,jr_1,\ldots,jr_m)$. Their torus coordinate
is nonzero and each additive coordinate separates the centers, since
$r_i\ne0$. Repeated values of $\alpha^j$ cause no difficulty in the
moving-center interpolation theorem. Taylor replacement acts by a
unitriangular filtered row change, as in Lemma~\ref{slog:lem:shear};
the selected full-row minor equals the corresponding full-logarithm minor. Thus, for all sufficiently large
divisible $H$, the matrix has a nonzero full-row minor $\Delta_H\in E$.
Keep its selected columns fixed when applying embeddings. Denote its
size by $M$ and put
\[
 \overline b_H=\frac1{MH}\sum_{\text{rows}}w\cdot\beta,
                         \qquad 0\le\overline b_H\le\theta.
\]

\emph{The arithmetic estimate.}
Choose integers $u_0,v_1\ge1$ such that $u_0c$ and $v_1\alpha$ are
algebraic integers. Let
\[
 J_i=\sum_{\text{selected columns}}\gamma_i
                         -\sum_{\text{rows}}\beta_i.
\]
One nonzero determinant term matches each row to a column with
$\beta\le\gamma$ componentwise, so every $J_i\ge0$. Scale a column
$(h,\gamma)$ by $v_1^{(K-1)h}\prod_i(u_0q_iL_i)^{\gamma_i}$ and a
row $\beta$ by $\prod_i(u_0q_iL_i)^{-\beta_i}$. The scaled entries
are algebraic integers: the additive factors have coefficients in
$\mathcal O_E$ because
\[
 u_0q_iL_i(jr_i+G_i(t))
            =jp_i(u_0c)L_i+u_0q_iL_iG_i(t),
\]
and the torus factor is
$v_1^{(K-1-j)h}(v_1\alpha)^{jh}$. It follows that
\[
 C_H\Delta_H\in\mathcal O_E\setminus\{0\},\qquad
 C_H=v_1^{(K-1)\sum_{\rm col}h}\prod_i(u_0q_iL_i)^{J_i}
                                      \in\Z_{>0}.
\]
Let $\Lambda=4\log2$. The elementary lcm estimate and the weighted
degree bounds give
\begin{align}
 \log L_i&\le\Lambda F_0w_i/v_0,\nonumber\\
 \sum_iw_iJ_i&\le MH(1-\overline b_H),\qquad
                    \sum_{\rm col}h\le MH/w_0,\nonumber\\
 \frac{\log C_H}{MH}&\le1-\overline b_H+E_{\rm ar},
 \quad E_{\rm ar}=\frac{\Lambda F_0}{v_0}
             +\frac{\log u_0}{w_*}+\frac{(K-1)\log v_1}{w_0}.
 \label{conj:clearing}
\end{align}
Taking the nonzero integer norm yields the exact lower bound
\begin{equation}\label{conj:norm}
 \sum_{\sigma:E\hookrightarrow\C}
       \frac{\log|\sigma(\Delta_H)|}{MH}
                 \ge-d(1-\overline b_H)-dE_{\rm ar}.
\end{equation}

\emph{Good embeddings.}
For a selected column $P$ use the entire row-vector families
\[
 f_{a,P}(z)=[u^a]P(e^z,z+u_1,\ldots,z+u_m).
\]
At a good embedding substitute
$z=j\lambda_\sigma+\log(1+t)$. The exponential identity and the
additive identity are exactly
\[
 e^z=\sigma(\alpha)^j(1+t),\qquad
 j\sigma(r_i)+\log(1+t)+u_i
     =z+u_i+j(\sigma(r_i)-\lambda_\sigma).
\]
We use the full-logarithm minor here, which equals \eqref{conj:matrix}
by the preceding unitriangular change. There is therefore one common coefficient family for all rows at a
given embedding, with precisely the same transverse grades as in the
one-logarithm collision lemma. Let
\[
 R=4K\bigl(1+\max_{\sigma\in S}|\lambda_\sigma|\bigr),\quad
 \kappa=\frac{\log2}{4},\quad
 L=\frac{\eta^2K\theta^m}{(m+1)v_0A^m}.
\]
Lemma~\ref{slog:lem:collision} and the translation estimate give, uniformly for $\sigma\in S$,
\begin{equation}\label{conj:good}
 \frac{\log|\sigma(\Delta_H)|}{MH}
 \le E_{\rm an}+o(1)+
 \max\{-\kappa L+o(1),-\rho(A(1-\eta)-\overline b_H)\},
\end{equation}
where a sufficient error bound is
\[
 E_{\rm an}=\frac{2\log2}{v_0}+
       \frac R{w_0}+
       \frac{\log4+\log K+\rho+\log(2R)}{w_*}.
\]
All $o(1)$ terms here are for $H\to\infty$ after every other datum is fixed.

\emph{Other embeddings.}
The values $\sigma(r_i)=\sigma(c)p_i/q_i$ stay bounded, because
$p_i/q_i\to x$. On $|t|=1/2$ the $G_i$ are uniformly bounded.
Cauchy's coefficient bound applied directly to \eqref{conj:matrix},
followed by Hadamard's inequality, gives
\begin{equation}\label{conj:bad}
 \frac{\log|\sigma(\Delta_H)|}{MH}
 \le E_{\rm bad}+o(1),\qquad
 E_{\rm bad}=O\left(\frac K{w_0}+\frac1{v_0}
                         +\frac{1+\log K}{w_*}\right).
\end{equation}
The fixed implied constant depends on the field and endpoint data.
Indeed $h\le H/w_0$, $k\le H/v_0$, $\sum_i\gamma_i\le H/w_*$,
and the determinant factor contributes only $\log M/H=o(1)$.
This estimate does not claim analytic smallness at a bad embedding.

Combining \eqref{conj:norm}--\eqref{conj:bad} and dividing by $s$,
the arithmetic leading term is $-D(1-\overline b_H)$. It suffices to have
\begin{align}
 g&=\rho(A(1-\eta)-\theta)-D(1-\theta)>0,\nonumber\\
 E&=DE_{\rm ar}+E_{\rm an}+(D-1)E_{\rm bad}<g,
             \qquad \kappa L>D+E.\label{conj:gap}
\end{align}
For the approximation alternative the worst value of
$\overline b_H$ is $\theta$, because $\rho>D$. For the collision
alternative the inequality $1-\overline b_H\le1$ suffices.

Here is the order of choices verifying \eqref{conj:gap}. Choose rational
$b_0$ with $1/2<b_0<1-D/\rho$ and small rational $\delta>0$; put
$\theta=1-\delta$, $A=1-b_0\delta$. Then
$A^2<\theta$ and $\rho(A-\theta)>D(1-\theta)$. Choose $\eta>0$
small enough to keep a strict gap, and fix any rational $F_0>2/\theta$. Choose rational $B,C$ so that
\[
 A<B<\sqrt\theta,\qquad B/\theta<C<1/B.
\]
In particular $C>1$, $CB<1$, and $B<C\theta<1$. Set
\[
 K=\lfloor C^m\rfloor,\quad w_0=B^{-m},\quad
 v_0=2K\theta^mw_0.
\]
As $m\to\infty$, $K/w_0\to0$, $v_0\to\infty$, and
\[
                  L=\frac{\eta^2(B/A)^m}{2(m+1)}\to\infty.
\]
First fix a sufficiently large finite $m$. Next choose the $m$ good
fractions successively so that all geometric thresholds and all $w_*^{-1}$
errors are satisfied. Finally let $H$ increase through the required
multiples. The strict inequalities contradict \eqref{conj:norm}.
Thus no $\nu>2d/s$ can occur infinitely often.
\end{proof}

\paragraph{Completed Lean theorem.}
The exact bound of Theorem~\ref{conj:partial}, with an actual nonempty
finite set of matching embeddings, is
\path{QuadraticPi.partialScalarMatchingLogarithm_irrationalityExponent_bounds}.
Its all-embedding endpoint is
\path{QuadraticPi.scalarMatchingLogarithm_irrationalityExponent_eq_two}.
Both full types and transitive axiom lists were rechecked for this revision.
In particular, when $2d\le3s$ this already formalized theorem gives a
bound in the range $2\le\mu(x)\le3$.

\subsection{Quadratic logarithms, circular values, and rational linear closure}

\begin{corollary}\label{conj:real-quadratic}
Let $D>0$ be rational and nonsquare, and let $\alpha\in\Q(\sqrt D)$
be positive, different from $1$, and have norm $1$ or $-1$. Then
\[
                         \mu\left(\frac{\log\alpha}{\sqrt D}\right)=2.
\]
\end{corollary}
\begin{proof}
For norm one take $c=\sqrt D$ and $x=\log\alpha/\sqrt D$.
The two logarithms are $\log\alpha$ and $-\log\alpha$, because
the conjugate of $\alpha$ is $\alpha^{-1}$. Thus both embeddings
are good. For norm minus one replace $\alpha$ by $\alpha^2$ and
$c$ by $2\sqrt D$, keeping $x$ unchanged.
\end{proof}
In particular, with $\varphi=(1+\sqrt5)/2$,
\[
 \mu\left(\frac{\log(1+\sqrt2)}{\sqrt2}\right)=2,
 \qquad \mu\left(\frac{\log\varphi}{\sqrt5}\right)=2.
\]
The algebraic normalizations are part of the conclusions. Invariance
under nonzero \emph{rational} scaling does not license removing a
factor $\sqrt D$. In the determinant proof the normalization has
an explicit function: conjugation changes both $\sqrt D$ and the chosen
logarithm to their negatives, so the approximation error is small at
both embeddings. Already in the norm-one case, for unnormalized
rational approximants to $\log\alpha$, conjugation fixes the approximants while negating the
logarithm. The second error would then tend to $2\log\alpha$, which
is nonzero. The averaged small-error argument would no longer apply.

For $D\in\Q^\times$, choosing $\sqrt D=\ii\sqrt{-D}$ when $D<0$, let
\begin{equation}\label{conj:fd}
 f_D(z)=\frac1{\sqrt D}\log\frac{1+\sqrt D\,z}{1-\sqrt D\,z}
       =2\sum_{n\ge0}\frac{D^nz^{2n+1}}{2n+1}
\end{equation}
denote its branch at zero. For $D>0$ and real $z$, use
$|\sqrt D\,z|<1$; for $D<0$ use its continuous real branch
$f_D(z)=2\arctan(\sqrt{-D}\,z)/\sqrt{-D}$ for all real $z$.
In either case $(1-Dz^2)f_D'(z)=2$.

\begin{corollary}\label{conj:fd-closure}
For fixed $D\in\Q^\times$, every nonzero rational linear combination
of finitely many admissible rational values $f_D(r)$ has irrationality
exponent two. If $D<0$, the combination may also include the value
$\pi/\sqrt{-D}$. In particular
\[
 \mu(f_D(r))=2\quad(r\ne0),\qquad
              \mu(\pi/\sqrt D)=2\quad(D\in\Q_{>0}).
\]
\end{corollary}
\begin{proof}
For a nonzero combination $x$, choose an integer $n\ge1$ clearing
its rational coefficients. Then $\exp(n\sqrt D\,x)=\alpha$ is a
product of integer powers of the algebraic torus coordinates in
\eqref{conj:fd}; the circular value contributes an integer power of $-1$.
If $\Q(\sqrt D)$ is quadratic, its other embedding sends $\sqrt D$
to $-\sqrt D$ and $\alpha$ to $\alpha^{-1}$, so both exponential
identities hold with $c=n\sqrt D$. If $D$ is a positive rational
square the field is $\Q$ and only its identity embedding is needed.
Apply Theorem~\ref{conj:partial} with all embeddings good.
For the last circular assertion use $E=\Q(\sqrt{-D})$, $c=\ii\sqrt D$,
$\alpha=-1$ and $x=\pi/\sqrt D$ directly.
\end{proof}
For orientation, the exact identities
\[
 f_1(1/3)=\log2,\quad
 \tfrac12 f_2(1/2)=\frac{\log(1+\sqrt2)}{\sqrt2},\quad
 \tfrac12 f_5(1/5)=\frac{\log\varphi}{\sqrt5}
\]
put familiar constants in one differential family. The fixed-$D$
restriction and the exclusion of zero are essential. The torus identity
behind the addition formula is rational, but logarithmic addition must
retain its chosen branches; arbitrary principal-branch replacements
after successive additions would lose periods.

\subsection{The invariant field and the limit of exact scalar matching}

\begin{theorem}\label{conj:invariant}
Let $\alpha,c$ be nonzero algebraic numbers and let $x\in\R\setminus\{0\}$
satisfy $e^{cx}=\alpha$. Set
\[
 E=\Q(\alpha,c),\qquad
 K=\Q\bigl(c^2,\alpha+\alpha^{-1},c(\alpha-\alpha^{-1})\bigr).
\]
Then
\[
                            \mu(x)\le2[K:\Q].
\]
Moreover the good embeddings of this minimal field $E$ are exactly
its embeddings fixing $K$, and their number is $[E:K]\in\{1,2\}$.
\end{theorem}
\begin{proof}
Write $S=\alpha+\alpha^{-1}$ and $T=c(\alpha-\alpha^{-1})$.
Since $\alpha=(S+T/c)/2$, one has $E=K(c)$ and $[E:K]\le2$.
If the index is two, the nontrivial $K$-automorphism sends
$(c,\alpha)$ to $(-c,\alpha^{-1})$; both it and the identity satisfy
the exponential matching condition. If the index is one, the identity
alone suffices. Theorem~\ref{conj:partial} gives the displayed bound.

To classify all good embeddings, suppose $\sigma$ is good. If
$\alpha\ne1$, then
\[
 \sigma(\alpha)=\exp\bigl((\sigma(c)/c)\,cx\bigr),
                       \qquad cx\ne0.
\]
The Gel'fond--Schneider theorem \cite[Chapters~1--2]{Baker1975}, used with this
specific logarithm $cx$ of $\alpha$, forces $\sigma(c)/c\in\Q$.
If $\alpha=1$, the two logarithms $cx$ and $\sigma(c)x$ are nonzero
integer multiples of $2\pi\ii$, and their ratio is again rational.
Extend $\sigma$ to an automorphism $\tau$ of a normal closure. If
$\tau(c)=rc$ with rational $r$, the finite order of $\tau$ gives
$r^n=1$ for some $n$, hence $r=1$ or $-1$. Thus a good embedding
has exactly one of the two actions
\[
        (c,\alpha)\longmapsto(c,\alpha),\qquad
        (c,\alpha)\longmapsto(-c,\alpha^{-1}).
\]
These elements generate $E$. Both actions fix $K$, and every
$K$-embedding has one of these actions and is good. This proves the claim.
\end{proof}

\paragraph{Lean scope.}
The degree bound, the index-one-or-two statement under minimal generation,
and both directions of the classification are now checked. The converse
uses the explicit standard Gel'fond--Schneider interface in
\nolinkurl{QuadraticPi.ScalarMatchingInvariantConverse.invariantField_theorem}.
The direct degree bound does not require that converse. The all-matching
rigidity consequence has its own checked endpoint.

An immediate rigidity statement is useful: if \emph{all} embeddings of
some number field containing $\alpha,c$ are good, then
$c^2\in\Q$ and $\alpha\in\Q(c)$. Indeed every automorphism of a
normal closure sends $c$ to $\pm c$, so fixes $c^2$; every automorphism
fixing $c$ also fixes $\alpha=e^{cx}$. Exact all-embedding matching
therefore descends to degree at most two. It is stronger than matching
absolute-value logarithms, and it is not a mechanism for arbitrarily
high-degree examples. The statement about at most two good embeddings
uses the \emph{minimal} field $E=\Q(\alpha,c)$; in a larger ambient
field the restrictions can have multiple extensions.

\begin{corollary}[Reciprocal trace degree]\label{conj:trace}
For every positive algebraic $\alpha\ne1$,
\[
 \mu\left(\frac{\log\alpha}{\alpha-\alpha^{-1}}\right)
                   \le2[\Q(\alpha+\alpha^{-1}):\Q].
\]
Equivalently, for real algebraic $t>2$ of degree $m$,
\[
 \mu\left(\frac{\log((t+\sqrt{t^2-4})/2)}{\sqrt{t^2-4}}\right)
                         \le2m.
\]
\end{corollary}
\begin{proof}
Take $c=\alpha-\alpha^{-1}$ in Theorem~\ref{conj:invariant}.
Then $c^2=(\alpha+\alpha^{-1})^2-4$ and
$c(\alpha-\alpha^{-1})=c^2$, so $K=\Q(\alpha+\alpha^{-1})$.
Replacing $\alpha$ by $\alpha^{-1}$ leaves the target unchanged;
we can therefore choose the root above one in the equivalent formulation.
\end{proof}
For a reciprocal irreducible minimal polynomial of degree $2m$,
the map $\alpha\mapsto\alpha^{-1}$ is a nontrivial involution,
whose fixed field is $\Q(\alpha+\alpha^{-1})$; the trace degree is
$m$. For example,
\[
 \alpha=\frac{2+\sqrt2+\sqrt{2+4\sqrt2}}2
\]
has irreducible polynomial $X^4-4X^3+4X^2-4X+1$. To check
irreducibility, the only possible rational roots are $\pm1$, neither
a root. A factorization into monic integral quadratics has constants
both $1$ or both $-1$. Constants $1$ would require integer roots of
$u^2+4u+2$, and constants $-1$ force opposite cubic and linear
coefficients. Both are impossible. Its trace is $2+\sqrt2$, giving
\[
 \mu\left(\frac{\log((2+\sqrt2+\sqrt{2+4\sqrt2})/2)}
                         {\sqrt{2+4\sqrt2}}\right)\le4.
\]

\begin{corollary}[Quadratic character combinations]\label{conj:character}
Let $E/\Q$ be a finite totally real Galois extension with group $G$,
fix a real embedding $\phi:E\hookrightarrow\R$, and let
$\chi:G\to\{1,-1\}$ be a nontrivial character. There is a positive
nonsquare rational $D_\chi$, chosen independently of the input $a$,
such that
\[
 \phi(E^{\ker\chi})=\Q(\sqrt{D_\chi}).
\]
For every $a\in E$ with $\phi(g(a))>0$ for all $g\in G$, put
\[
 \beta_\chi(a)=\prod_{g\in G}g(a)^{\chi(g)}.
\]
If $\beta_\chi(a)\ne1$, then
\[
 \mu\left(\frac1{\sqrt{D_\chi}}
       \sum_{g\in G}\chi(g)\log\phi(g(a))\right)=2.
\]
The same conclusion holds for a specified positive nonsquare rational
$D_\chi$ satisfying the displayed equality of real subfields.
\end{corollary}
\begin{proof}
The kernel has index two, so its fixed field is quadratic. A real
quadratic field is the field generated by the positive square root of
some positive nonsquare rational; this choice is made before $a$.
Reindexing the signed product gives
$h(\beta_\chi)=\beta_\chi^{\chi(h)}$. The product thus lies in
the quadratic fixed field and its nontrivial conjugate is its inverse,
so its field norm is one. Positivity of all factors gives
\[
 \log\phi(\beta_\chi)=\sum_{g\in G}\chi(g)\log\phi(g(a)).
\]
The product-not-one condition makes this logarithm nonzero. Apply
Corollary~\ref{conj:real-quadratic} in the compatible real quadratic
realization. The full proof, including the realization and norm
transport, is given in Theorem~\ref{ext:character-exponent}.
\end{proof}

This is a result about the displayed \emph{normalized} character sum.
It does not assert exponent two for arbitrary unnormalized logarithmic
combinations. If a negative-character automorphism fixes $a$, the
positive and negative product factors pair and cancel, giving
$\beta_\chi(a)=1$. More generally, a nontrivial product requires
$E^{\ker\chi}\subseteq\Q(a)$, and consequently $[\Q(a):\Q]$
must be even. In particular, positive inputs of odd degree have zero
character logarithmic sum; see Proposition~\ref{ext:character-input-field}.

\paragraph{Lean scope.}
This corollary is now a completed Lean endpoint, including the actual
product, quadratic degree, norm one, positive logarithmic identity and
compatible square-root realization. The declaration
\nolinkurl{QuadraticPi.QuadraticCharacterExponent.exists_logSum_irrationalityExponent_eq_two}
chooses $D_\chi$ and the compatible field isomorphism before all
admissible $a$. It assumes neither an unproved quadratic identification
nor an extra analytic interpolation statement. The specified-field
variant is \nolinkurl{logSum_irrationalityExponent_eq_two_of_fieldRange}
in the same namespace. No canonical integral or squarefree choice of
$D_\chi$ is claimed. The later stabilizer and input-degree cancellation
statements are also compiled. Exact source declarations and recorded
build/axiom evidence are listed in Section~\ref{ext:character-product}
and Appendix~\ref{app:formal-status}.

\section{Ratios of two logarithms: two torus coordinates}
\label{ratio:section}

\begin{theorem}\label{ratio:thm:ratio}
Let $a,b>0$ be rational and multiplicatively independent. Then
\[
 2\le\mu\!\left(\frac{\log a}{\log b}\right)\le3.
\]
In particular $\mu(\log2/\log3)\le3$.
\end{theorem}

\paragraph{Formal verification.}
The theorem is exported as
\nolinkurl{QuadraticPi.rationalLogRatio_irrationalityExponent_bounds}
in \nolinkurl{QuadraticPi/RationalLogRatioTheorem.lean}.
The actual two-torus interpolation and determinant estimates are proved;
no interpolation premise remains in this public endpoint.
The verified conclusion is the bound $2\le\mu\le3$.
See Appendix~\ref{app:formal-status} for dependencies and compiler evidence.
The completion claim in this paragraph is specific to the rational
ratio theorem and its rational homogeneous two-logarithm consequence.
Its source chain consists of the \nolinkurl{TwoLog*} developments and
these public rational endpoints; it does not identify the more general
number-field comparisons below as completed Lean theorems.

We prove the case $(a,b)=(2,3)$ in detail, and explain the rational
denominator changes at the end. Put
$\xi=\log2/\log3$. Unique factorization proves $\xi\notin\Q$.
To exclude exponent $\nu>3$, suppose there are arbitrarily good
approximations
\begin{equation}\label{ratio:eq:good}
 r_i=p_i/q_i,\qquad |\xi-r_i|\le q_i^{-\nu},\qquad
 w_i=\lceil\log q_i\rceil.
\end{equation}
The fractions can be chosen distinct, nonzero, reduced, and with
successively large denominators. All parameters and fractions will
be fixed before a final polynomial degree $H$ tends to infinity.
Here multiplicative independence means that
$a^u b^v=1$, for $(u,v)\in\Z^2$, implies $u=v=0$.
In particular $a\ne1$ and $b\ne1$, so both real logarithms are
nonzero. If $\log a/\log b=p/q$, with $q\ge1$, then
$q\log a-p\log b=0$ and hence $a^q b^{-p}=1$, a contradiction.
Thus the theorem concerns an irrational real number, and its lower
bound is the ordinary real Dirichlet bound.

To justify the contradiction formulation, if $\mu(\xi)>3$, choose
$3<\nu<\mu(\xi)$ and obtain infinitely many rational approximations
satisfying \eqref{ratio:eq:good}. Their reduced denominators are
unbounded: a bounded interval contains only finitely many reduced
fractions of bounded denominator. Reduction cannot spoil the inequality,
since replacing a denominator by a smaller one increases its negative
power. Discarding finitely many terms gives nonzero, pairwise distinct
slopes and permits each next denominator to exceed any fixed threshold.
The weights $w_i=\lceil\log q_i\rceil$ are positive integers. We use
both sides of the elementary relation
$w_i-1<\log q_i\le w_i$; its lower side accounts for the ceiling loss
in the analytic estimate.

\subsection{The new interpolation statement}

Throughout this section write $w\cdot a=\sum_iw_ia_i$ and
$w_*=\min_iw_i$. The ambient coordinates are $Y,Z,X_1,\ldots,X_m$, $m\ge2$.
Give them degree weights
$W=(w_0,w_0,w_1,\ldots,w_m)$.
At centers
\begin{equation}\label{ratio:eq:centers}
 (Y,Z,X)=(3^j,2^j,0),\qquad0\le j<K,
\end{equation}
use local coordinates
\[
 t=Y/3^j-1,\quad s=Z/2^j-1,\quad
 u_i=X_i-\log(1+s)+r_i\log(1+t),
\]
with weights $V=(v,v,w_1/\theta,\ldots,w_m/\theta)$.

\begin{lemma}\label{ratio:lem:two-torus}
Let $K\ge1$, $m\ge2$, $0<\theta<1$, and $w_0,v>0$ be rational.
Assume there is a rational $\sigma>0$ with
\begin{align}\label{ratio:eq:geom-margins}
 &(1+3\sigma)^{m+2}K(w_0/v)^2\theta^m<1,\nonumber\\
 &(1+3\sigma)^{m+1}(w_0/v)\theta^{m-1}<1,\nonumber\\
 &(1+3\sigma)^{m-1}\theta^{m-1}<1,
 \qquad(1+\sigma)\theta<1,\nonumber\\
 &w_0/v>4(m+3)/\sigma.
\end{align}
There are successive thresholds for $w_1,\ldots,w_m$, independent
of the distinct nonzero rational slopes $r_i$, such that the map
\begin{equation}\label{ratio:eq:two-torus-map}
 P\longmapsto
 \left[P\bigl(3^j(1+t),2^j(1+s),
          (\log(1+s)-r_i\log(1+t)+u_i)_i\bigr)\right]_{j<K}
\end{equation}
from polynomials of $W$-degree at most $H$ onto all coefficient
packets of weight less than $H$ is surjective for all sufficiently
large divisible $H$.
\end{lemma}
The target is the direct sum over the $K$ centers of the finite quotients
$\Q[[t,s,u]]/F_V^{\ge H}$, with basis the monomials
$t^\lambda s^\rho u^\beta$ of weight strictly less than $H$.
Equivalently one can first work over $\C$: because the resulting finite
matrix has rational entries, full row rank over $\C$ is full row rank
over $\Q$. The quantifiers are important. The base data
$(m,K,\theta,w_0,v,\sigma)$ are fixed first. Each threshold for $w_i$
depends only on those data and the earlier weights. After all weights
and all slopes have been fixed, the degree may be required to exceed
a threshold and to run through multiples of a fixed denominator.
Neither the slopes nor this final degree enter the successive weight
thresholds.

\begin{proof}
Put $L_0=m+3$. Label the torus directions $Y,Z$ below every positive
index $1,\ldots,m$. Write $W_Y=W_Z=w_0$, $V_Y=V_Z=v$,
$W_i=w_i$, and $V_i=w_i/\theta$ for $i>0$. Impose
\begin{equation}\label{ratio:tt:sep}
 \prod_{h\in A}W_h>
 C_*\prod_{h\in B}V_h,\qquad
 C_*=\max_{1\le k\le m+1}2k!(L_0/\sigma)^k,
\end{equation}
whenever $|A|=|B|$ lies between $1$ and $m+1$, some positive index
lies in $A\mathbin\triangle B$, and the largest such index belongs
to $A$. The two torus indices are tied: no condition is imposed merely
because they differ. This causes no conflict in the successive choice.
For largest differing positive index $i$, the ratio of the two products
is $w_i$ times a positive expression in previously chosen weights and
fixed torus weights. Higher common indices contribute the fixed factor
$\theta$; all other higher indices are absent. A sufficiently large
$w_i$ therefore meets every new inequality and preserves all earlier
ones. Enlarge these thresholds to meet the one-torus separation
conditions of \eqref{geom:eq:separate} after deleting any one positive index. This is possible for
all the finitely many deletions at once and is independent of $r_i$.

For an irreducible algebraic curve $C$ meeting the open chart $YZ\ne0$,
on its complete smooth model define
\[
 d_W(C)=\sum_P\max\left(0,-\frac{\operatorname{ord}_PY}{w_0},
 -\frac{\operatorname{ord}_PZ}{w_0},-\frac{\operatorname{ord}_PX_1}{w_1},\ldots,
 -\frac{\operatorname{ord}_PX_m}{w_m}\right),
\]
and at each branch through a center put
\[
 h_P=\min\left(\frac{\operatorname{ord}_Pt}{v},\frac{\operatorname{ord}_Ps}{v},
              \min_i\frac{\operatorname{ord}_Pu_i}{w_i/\theta}\right).
\]
We prove $d_W(C)\ge(1+\sigma)\sum_Ph_P$.
Suppose otherwise. The leading dimension count \eqref{ratio:eq:geom-margins} gives,
for every sufficiently large $N$, a nonzero polynomial of $W$-degree
at most $N$ having contact at least $(1+3\sigma)N$ at all centers.
The commuting, degree-preserving frame
\[
 D_Y=Y\partial_Y-\sum_i r_i\partial_{X_i},\quad
 D_Z=Z\partial_Z+\sum_i\partial_{X_i},\quad
 D_i=\partial_{X_i}
\]
is $(1+t)\partial_t,(1+s)\partial_s,\partial_{u_i}$ locally.
The zero--pole comparison therefore forces all derivatives of cost
at most $\sigma N$ to vanish on $C$.
In detail, differentiating by a frame vector of cost $V_h$ lowers
local weighted contact by at most $V_h$ and does not increase polynomial
$W$-degree. A derivative of total cost at most $\sigma N$ therefore
has local contact at least $(1+2\sigma)N$. If its restriction to $C$
were nonzero, its total zero order would be at least
$(1+2\sigma)N\sum_P h_P$, whereas its total pole order would be at
most $N d_W(C)$. The assumed strict failure of the curve inequality
makes the former strictly greater than the latter. A nonzero rational
function on a complete smooth curve has equal total zero and pole
orders, giving the contradiction. This also explains why the spare
contact margin is kept before the persistent-component argument.

For derivative-cost steps $\delta=\sigma N/L_0$, take nested irreducible
components containing $C$ backwards through the zero sets of derivatives
of cost at most $r\delta$, $0\le r\le L_0$. Their dimensions are between
$1$ and $m+1$, so two adjacent components coincide. More explicitly, strict inclusion
of irreducible components would strictly lower dimension, and there
are more stages than the available positive dimensions. Call a persistent
component $T$, and let its codimension be $k$. Lemma~\ref{geom:lem:mult}, the weighted
transverse multiplicity estimate, gives, for every coordinate normal basis
$A$ and every frame normal basis $B$,
\begin{equation}\label{ratio:tt:mult}
 \prod_{h\in A}W_h\le
       2k!(L_0/\sigma)^k\prod_{h\in B}V_h.
\end{equation}
These bases are evaluated on a common generic smooth open subset of $T$.
For the application of the multiplicity lemma, use the equations from
the earlier of the equal adjacent stages. Any further derivative of
cost less than $\delta$ belongs to the equations of the next stage,
so vanishes on the same component. Both stages are considered inside
$YZ\ne0$, where the displayed fields are a frame. The initial zero
set is proper because its defining polynomial is nonzero, and every
component in the chain contains the curve; consequently
$1\le k\le m+1$. Thus the constants used in
\eqref{ratio:tt:sep} cover every possible persistent component.

If all positive coordinate directions are tangent to $T$, then
$k\le2$. Modulo the tangent space, $D_Y$ and $D_Z$ are the nonzero
scalar multiples $Y\partial_Y$ and $Z\partial_Z$ of their coordinate
directions; the additive terms are tangent. Both types of normal bases
can thus be chosen entirely from the two torus indices. Equation
\eqref{ratio:tt:mult} gives
\[
 (w_0/v)^k\le2k!(L_0/\sigma)^k,
\]
contradicting \eqref{ratio:eq:geom-margins} for $k=1$ or $2$.

Otherwise let $i$ be the largest positive index whose coordinate
direction is nontangent. Its nonzero image in the normal quotient
extends to a coordinate normal basis $A$. Every positive index larger
than $i$ is tangent, hence occurs in no coordinate or frame normal
basis, since $D_j=\partial_{X_j}$ for positive $j$. If a frame normal
basis $B$ omitted $i$, the largest positive index in
$A\mathbin\triangle B$ would be exactly $i$, belonging to $A$.
Equation \eqref{ratio:tt:sep} would contradict \eqref{ratio:tt:mult}. Therefore
\emph{every} frame normal basis contains $D_i$.

Consequently the images of all other frame vectors do not span the
normal quotient. Their ambient span is the hyperplane
\[
 Q_i=\ker(dX_i+r_i\,dY/Y-dZ/Z).
\]
For an ambient hyperplane $Q$ and a tangent space $S$, its image in the
quotient by $S$ fails to be surjective exactly when $S\subseteq Q$:
if $S$ contains a vector outside $Q$, then $Q+S$ is the full ambient
space; otherwise $Q+S=Q$. Thus the tangent space of $T$ is contained
in $Q_i$, giving on $T$, with $r_i=p_i/q_i$ reduced,
\[
 q_i\,dX_i=d\log(Z^{q_i}/Y^{p_i}).
\]
The character on the right is constant on $T$. Indeed, if nonconstant,
restrict through a point of nonzero differential to an algebraic curve
and pass to its complete smooth model. A zero or pole of this character
has a nonzero residue in its logarithmic differential, whereas the exact
differential $q_i\,dX_i$ has residue zero. This is impossible. Then
$X_i$ too is constant. Its value on $C$ is zero, since $C$ meets a center.
The character takes the distinct values $(2^{q_i}/3^{p_i})^j$ at the
centers, so $C$ meets just one, say $c_j$.
To make the rigidity step explicit, the differential identity holds
in the function field of $T$. A nonconstant character can be restricted
to a curve on which it remains nonconstant; on the complete model of
that curve it has a zero or pole. In a local parameter $z$, writing
that character as $z^n u(z)$ gives residue $n\ne0$ for its logarithmic
differential. By contrast, differentiating any Laurent expansion has
zero coefficient of $dz/z$, so an exact differential has residue zero.
Once the character is constant, $dX_i=0$ forces $X_i$ to be constant
in characteristic zero. Finally, $2^{q_i}/3^{p_i}$ is a positive real
number different from one, by multiplicative independence. Its integer
powers are therefore distinct, even when $p_i$ is negative.

On this character fiber, the local relation is
\[
 \log(Z/2^j)=r_i\log(Y/3^j).
\]
Since $r_i\ne0$, the orders of the two torus shifts coincide, including
the case where both vanish identically. Also $u_i=0$ and, for $l\ne i$,
\[
 u_l=X_l-(r_i-r_l)\log(Y/3^j).
\]
Dropping $Z,X_i$, and dividing $X_l$ by $r_i-r_l$, gives precisely the
one-torus, one-center contact of Proposition~\ref{geom:prop:curve} in $m$ ambient coordinates. Its
volume hypotheses follow from \eqref{ratio:eq:geom-margins}, and its
successive separation conditions were included above.
There are $m-1$ remaining additive variables and one image center.
The one-torus volume conditions required there are
\[
 (1+3\sigma)^m(w_0/v)\theta^{m-1}<1,
 \qquad (1+3\sigma)^{m-1}\theta^{m-1}<1,
 \qquad (1+\sigma)\theta<1.
\]
The first follows from the stronger second line of
\eqref{ratio:eq:geom-margins}, and the other two are already among its
hypotheses. The nonzero rescaling $X_l\mapsto X_l/(r_i-r_l)$ changes
neither pole orders nor local contact orders, since $r_i\ne r_l$.
When the torus shifts vanish identically their orders are $+\infty$;
this convention is compatible with deleting the redundant shift.

This projection need not be birational. It is nonconstant and finite on
complete smooth curve models: $Z$ is algebraic over $Y$ and $X_i=0$;
if all remaining coordinates were constant, $C$ would be constant.
Let its degree be $e$. The reduced weighted pole degree on $C$ equals
$e$ times that of the image and is at most $d_W(C)$. Local contacts
pull back with ramification. The sum of contacts over the single
original center is at most $e$ times the full image-center contact.
Proposition~\ref{geom:prop:curve} on the image thus implies
\[
 d_W(C)\ge e\,d_{W'}(\mathrm{image}(C))
       \ge e(1+\sigma)\sum_{P'\text{ over image center}}h_{P'}
       \ge(1+\sigma)\sum_{P\mapsto c_j}h_P,
\]
the required contradiction.

Finally replace
 each logarithm by a Taylor polynomial whose omitted
terms have weight strictly greater than $w_i/\theta$. This is an
invertible coordinate change preserving the weighted ideal: a replacement
of $u_i$ adds only monomials of strictly higher weight, and its inverse
subtracts the same series. Thus contacts and the finite jet quotient
are preserved. The resulting coordinates are algebraic and regular.
The curve-to-jet construction in the proof of
Theorem~\ref{geom:thm:main} now applies: use ordinary blow-ups of the
finite-colength local ideals, the nef class supplied by the curve
inequality, its ample strict perturbation, eventual ordinary-power
pushdown, and Serre vanishing. The argument uses the curve inequality
and local regular parameters, not the number of torus coordinates.
It therefore supplies the asserted surjection here.
\end{proof}

\subsection{The arithmetic minor and the exact logarithm replacement}

Fix a rational $F_0>2/\theta$ and put
\[
 T_i=\lceil F_0w_i/v\rceil,\qquad
 G_i(t)=\sum_{1\le k<T_i}(-1)^{k+1}t^k/k,\qquad
 L_i=\operatorname{lcm}(1,\ldots,T_i-1).
\]
Use columns $P=Y^hZ^kX^\gamma$ of weighted degree
$w_0(h+k)+w\cdot\gamma\le H$, and all rows
\[
 (j,\lambda,\rho,\beta),\qquad 0\le j<K,\qquad
                   v(\lambda+\rho)+w\cdot\beta/\theta<H.
\]
Let $A_H^{\rm tr}$ be the matrix with entry
\begin{equation}\label{ratio:matrix}
 [t^\lambda s^\rho u^\beta]
 P\bigl(3^j(1+t),2^j(1+s),(G_i(s)-r_iG_i(t)+u_i)_i\bigr),
\end{equation}
and let $A_H^{\rm full}$ use full logarithms instead of $G_i$.

\begin{lemma}[The same selected minor uses full logarithms]
\label{ratio:shear}
For every choice of $M$ columns, where $M$ is the number of rows,
the corresponding determinants of $A_H^{\rm tr}$ and
$A_H^{\rm full}$ are equal. In particular the interpolation lemma
provides a nonzero rational full-row minor $\Delta_H$ of both matrices.
\end{lemma}
\begin{proof}
Put $R_i(t)=G_i(t)-\log(1+t)$. The truncated substitution is the
full substitution followed by
\[
                    u_i\longmapsto u_i+R_i(s)-r_iR_i(t).
\]
The added terms have weight at least $vT_i>w_i/\theta$.
Consequently this substitution preserves the ideal of monomials of
weight at least $H$ in $\Q[[t,s,u]]$. On the quotient, order the
monomials by increasing weight, breaking ties in any fixed way.
The substitution has a triangular coefficient matrix with diagonal
one: every nonidentity term has strictly larger weight. Its inverse
subtracts the same series and preserves the quotient as well.
The row change is independent of the selected column. Apply this at
each of the $K$ centers; the resulting block triangular matrix has
determinant one. Left multiplication by it takes the full matrix to
the truncated one, proving the assertion. This does not claim
coefficientwise equality of the two matrices.
\end{proof}

Elementary lattice-point counting in the weighted simplex gives
\begin{equation}\label{ratio:count}
 M\sim\frac{K\theta^mH^{m+2}}{(m+2)!v^2\prod_iw_i},\qquad
 \overline b_H=\frac1{MH}\sum_{\rm rows}w\cdot\beta,
                         \qquad 0\le\overline b_H\le\theta.
\end{equation}
The weights and $m$ are fixed in this asymptotic. Choose one nonzero
full-row minor and write
\[
                  J_i=\sum_{\rm columns}\gamma_i-\sum_{\rm rows}\beta_i.
\]
The number of selected columns is exactly $M$. Full row rank supplies
at least one such nonzero determinant for each sufficiently large
admissible $H$; no fixed choice of columns for all degrees is needed.
Within a fixed degree, however, that one column set is retained through
every subsequent arithmetic and analytic expansion.

As before, a nonzero determinant term proves $J_i\ge0$.
Indeed, a coefficient entry is zero unless $\beta_i\le\gamma_i$ for
every $i$, because the original monomial has degree $\gamma_i$ in its
$i$th additive variable. In a nonzero term of the determinant expansion,
one nonzero entry is chosen in each row and column. Summing these
inequalities over that permutation proves $J_i\ge0$ coordinatewise.
 Scale column
$\gamma$ by $\prod_i(q_iL_i)^{\gamma_i}$ and row $\beta$ by
$\prod_i(q_iL_i)^{-\beta_i}$. The identity
\[
 q_iL_i(G_i(s)-r_iG_i(t))=q_iL_iG_i(s)-p_iL_iG_i(t)
\]
shows that every scaled entry is integral.
For clarity, the nonzero entry with row $\beta$ and column $\gamma$
contains the factor
\[
 \prod_i\binom{\gamma_i}{\beta_i}
        (G_i(s)-r_iG_i(t))^{\gamma_i-\beta_i}.
\]
After the stated row and column scaling, its remaining powers are
powers of the integer-coefficient polynomials
$q_iL_iG_i(s)-p_iL_iG_i(t)$; the binomial factors and the coefficients
of $3^{jh}2^{jk}(1+t)^h(1+s)^k$ are integers as well. Zero entries
remain zero even when their row scaling has negative exponent.
This is an entrywise integrality statement for the truncated matrix.
Lemma~\ref{ratio:shear} then transfers its nonzero determinant, with
exactly the same value, to the full-logarithm matrix.
 Therefore
\[
 C_H\Delta_H\in\Z\setminus\{0\},\qquad
                       C_H=\prod_i(q_iL_i)^{J_i}\in\Z_{>0}.
\]
Since $\log q_i\le w_i$, $\log L_i\le\Lambda F_0w_i/v$ with
$\Lambda=4\log2$, and $\sum_iw_iJ_i\le MH(1-\overline b_H)$,
we obtain
\begin{equation}\label{ratio:lower}
 \frac{\log|\Delta_H|}{MH}
       \ge-(1-\overline b_H)-E_{\rm ar},
                              \qquad E_{\rm ar}=\Lambda F_0/v.
\end{equation}
The row factors matter: clearing all columns without returning the
row grades would lose the favorable term $\overline b_H$.
Indeed $\sum_iw_iJ_i\le MH-\sum_{\rm rows}w\cdot\beta
=MH(1-\overline b_H)$, by the degree bound on each selected column.
Thus the integer determinant gives the sharper intermediate estimate
\[
 \frac{\log|\Delta_H|}{MH}
 \ge -(1+E_{\rm ar})(1-\overline b_H),
\]
which implies \eqref{ratio:lower} because
$0\le\overline b_H\le\theta<1$. The least-common-multiple estimate
uses $T_i-1<F_0w_i/v$, so the ceiling in $T_i$ introduces no additional
term into the displayed arithmetic error.

\subsection{Translation to an entire family in two variables}

For a transverse multi-index $a\in\Z_{\ge0}^m$ and a selected column
$P$, define
\begin{equation}\label{ratio:entire}
 f_{a,P}(z,w)=[u^a]P(e^z,e^w,w-\xi z+u_1,\ldots,w-\xi z+u_m).
\end{equation}
All these functions are entire. The selected column set is the same
for every row. Set
\[
 z_j(t)=j\log3+\log(1+t),\qquad
 w_j(s)=j\log2+\log(1+s),\qquad
 \epsilon_i(t)=(\xi-r_i)\log(1+t).
\]
The central additive term vanishes because $j\log2-\xi j\log3=0$.
Expanding the additive variables gives the exact finite identity
\begin{align}
 &[u^\beta]P\bigl(3^j(1+t),2^j(1+s),
                   (\log(1+s)-r_i\log(1+t)+u_i)_i\bigr)
 \nonumber\\
 &\hspace{8mm}=
 \sum_{a\ge\beta}\binom a\beta\,
                 \epsilon(t)^{a-\beta}f_{a,P}(z_j(t),w_j(s)).
 \label{ratio:translation}
\end{align}
Only $a$ with $w\cdot a\le H$ occur, because the selected columns
have that bound.
Here $a\ge\beta$ means coordinatewise inequality, and
$\binom a\beta=\prod_i\binom{a_i}{\beta_i}$.
At $(z,w)=(z_j(t),w_j(s))$ the common additive variable in
\eqref{ratio:entire} is
$\log(1+s)-\xi\log(1+t)$.
Adding $\epsilon_i(t)$ changes it to
$\log(1+s)-r_i\log(1+t)$, which proves
\eqref{ratio:translation} directly by the multivariable binomial
formula. Although the functions are entire in two variables, this
translation in the transverse index is finite for every polynomial.
 Extracting $[t^\lambda s^\rho]$ introduces a finite
convolution over $0\le\ell\le\lambda$. Its scalars are
\[
 \binom a\beta[t^\ell]\epsilon(t)^{a-\beta},
\]
and its remaining row tests are
$[t^{\lambda-\ell}s^\rho]f_{a,P}(z_j(t),w_j(s))$.
This expansion is independent of the selected column.

On $|t|=1/2$, $|\log(1+t)|\le\log2<1$. Since
$q_i^{-\nu}\le\exp(-\nu w_i+\nu)$, Cauchy's estimate and
$\binom a\beta\le2^{|a|}$ give
\begin{equation}\label{ratio:scalar}
 \left|\binom a\beta[t^\ell]\epsilon(t)^{a-\beta}\right|
 \le 2^\ell\exp\left(-\nu w\cdot(a-\beta)
                              +\frac{(\nu+\log2)H}{w_*}\right).
\end{equation}
For fixed $m$ and weights there are only polynomially many choices
of $(a,\ell)$ per row. Their total contribution to the logarithm of
the determinant bound is $O(M\log H)=o(MH)$.
For example there are at most
$\prod_i(1+H/w_i)(1+H/v)$ choices per row, with constants depending
only on the already fixed weights. Raising this polynomial count to
the $M$th power contributes $O(M\log H)$ to the logarithm. This
counting statement uses the order of limits: $m$ and all weights are
fixed while $H$ grows.

Let
\[
 R=4K(1+\log3+\log2),\qquad B_R=2R(1+|\xi|).
\]
On $|t|=|s|=1/2$ the points $(z_j(t),w_j(s))$ lie in the
polydisc of radius $R/2$. On the radius-$R$ polydisc,
\begin{equation}\label{ratio:growth}
 |f_{a,P}(z,w)|\le
                  \exp\left(H\left(\frac R{w_0}
                                     +\frac{\log B_R}{w_*}\right)\right).
\end{equation}
One way to see the coefficient factor is to extract $u^a$ on
$|u_i|=1$: the additive factors have modulus at most
$R(1+|\xi|)+1\le B_R$. The torus exponents satisfy $h+k\le H/w_0$.

Expand each entire function in normalized Taylor monomials
$(z/R)^k(w/R)^l$. Cauchy's estimate bounds the coefficient by
\eqref{ratio:growth}, whereas a remaining row test is bounded by
\[
                  2^{\lambda-\ell+\rho}2^{-k-l}.
\]
Together with $2^\ell$ in \eqref{ratio:scalar}, the total row test
cost is at most $2^{\lambda+\rho}\le\exp(H\log2/v)$.

\begin{lemma}[Bivariate Taylor collision]\label{ratio:collision}
For $n$ distinct pairs $(k,l)\in\Z_{\ge0}^2$,
\[
                    \sum(k+l)\ge\frac{2\sqrt2}{3}n^{3/2}-n.
\]
If a term of the row expansion assigns $n_a$ rows to transverse
index $a$, the analytic determinant factor has the upper bound
\[
 \exp\left(-c_2\sum_an_a^{3/2}
                         +MH(E_{\rm hol}+o(1))\right),
 \quad c_2=\frac{\sqrt2\log2}{3},
 \quad E_{\rm hol}=\frac R{w_0}+\frac{\log B_R}{w_*}+\frac{\log2}{v}.
\]
\end{lemma}
\begin{proof}
Attach a unit square to each pair. The union has area $n$ and its
integral of $x+y$ is $\sum(k+l)+n$. Among sets of area $n$ in the
positive quadrant, the sublevel triangle $x+y\le\sqrt{2n}$
minimizes that integral. Integrating on the triangle gives
$(2\sqrt2/3)n^{3/2}$, proving the combinatorial inequality.

For a fixed $a$, the Taylor coefficient row indexed by $(k,l)$ is
the vector of those coefficients across the selected columns.
Consequently repeated triples $(a,k,l)$ give identical coefficient
rows and zero determinant. Every surviving term therefore has distinct
pairs within each $a$-block. Reserve one half of $2^{-k-l}$ to use the
combinatorial inequality; summing the other half contributes only a
fixed geometric factor per row. The coefficient determinant is bounded
by $M!$ times the product of its coefficient bounds.
There are only $M!$ summands in its permutation expansion, each with
one coefficient from each row. The relevant coefficients are those of
normalized Taylor monomials, so Cauchy's bound is exactly
\eqref{ratio:growth}, with no factorial derivative normalization left
to account for.
 Its logarithmic
factor is $O(M\log M)=o(MH)$, by \eqref{ratio:count}. The linear
correction $\sum_an_a=M$ is likewise negligible at that scale.

For completeness the Taylor expansion can be truncated at
$k+l\le J$ before applying multilinearity. The geometric majorant
just used is summable independently of $J$, while the omitted entire
series converge uniformly on the smaller testing polydisc. Passing to
the limit therefore preserves the determinant bound. No exchange of
a divergent determinant expansion and summation is involved.
\end{proof}

The compiled argument uses the sufficient weaker collision constant
$\log2/8$, and obtains the same endpoint three. The sharper displayed
constant in Lemma~\ref{ratio:collision} is not asserted to be the
identical formally verified numerical lemma.

There is a useful dichotomy for each term of the finite row translation.
Fix $0<A<1$ and $0<\eta<1$. The number of transverse indices with
$w\cdot a\le AH$ satisfies
\[
                 D_A(H)\sim\frac{A^mH^m}{m!\prod_iw_i}.
\]
If at least $\eta M$ rows have these indices, convexity gives
\[
 \sum_an_a^{3/2}\ge\frac{(\eta M)^{3/2}}{\sqrt{D_A(H)}}.
\]
After division by $MH$ its leading value is
\begin{equation}\label{ratio:L}
 L=\eta^{3/2}
       \sqrt{\frac{K\theta^m}{(m+1)(m+2)v^2A^m}}.
\end{equation}
If fewer than $\eta M$ rows have low transverse weight, then more
than $(1-\eta)M$ have $w\cdot a>AH$. The approximation factors in
\eqref{ratio:scalar} save at least
$\nu(A(1-\eta)-\overline b_H)$ per $MH$. The analytic estimate is
therefore
\begin{equation}\label{ratio:upper}
 \frac{\log|\Delta_H|}{MH}
 \le E_{\rm an}+o(1)+
  \max\{-c_2L+o(1),-\nu(A(1-\eta)-\overline b_H)\},
\end{equation}
where one can take
\[
 E_{\rm an}=\frac R{w_0}+\frac{\log2}{v}
                   +\frac{\log B_R+\nu+\log2}{w_*}.
\]
In the first alternative the omitted approximation factors are at
most one after removal of the displayed ceiling loss; in the second
alternative we simply omit the favorable collision factor. Thus the
two estimates apply term by term before summing all row choices.
The subtraction of $\overline b_H$ in the second alternative is
particularly useful. Summing the approximation exponents over all rows
gives
\[
 \frac1{MH}\sum_{\rm rows}w\cdot(a-\beta)
 =\frac1{MH}\sum_{\rm rows}w\cdot a-\overline b_H
 \ge A(1-\eta)-\overline b_H.
\]
In the first alternative, H\"older's inequality on the at most
$D_A(H)$ low-weight indices gives the stated collision bound regardless
of how the remaining rows are distributed. Substitution of
\eqref{ratio:count} into $\eta^{3/2}\sqrt{M/D_A(H)}/H$ yields exactly
\eqref{ratio:L}. The same arithmetic row grade therefore participates
in both alternatives; it is not a separately chosen approximation
parameter.

\subsection{The exponent-three parameter balance}
\label{ratio:parameters}

Comparing \eqref{ratio:lower} with \eqref{ratio:upper}, it suffices that
\begin{equation}\label{ratio:gap}
 \begin{split}
 g&=\nu(A(1-\eta)-\theta)-(1-\theta)>0,\\
 E_{\rm ar}+E_{\rm an}&<g,\qquad
                     c_2L>1+E_{\rm ar}+E_{\rm an}.
 \end{split}
\end{equation}
Indeed the approximation gap is
$\nu A(1-\eta)-1-(\nu-1)\overline b_H$, minimized at
$\overline b_H=\theta$, while $1-\overline b_H\le1$ handles the
collision alternative.
More explicitly, put $E=E_{\rm ar}+E_{\rm an}$. If the maximum in
\eqref{ratio:upper} is supplied by the approximation alternative,
comparison with \eqref{ratio:lower} would imply
\[
 \nu A(1-\eta)-1-(\nu-1)\overline b_H\le E+o(1).
\]
Its left side is at least $g$, contradicting $g>E$ for all sufficiently
large $H$. If the maximum is supplied by the collision alternative,
the two bounds would instead imply
\[
 c_2L\le1-\overline b_H+E+o(1)\le1+E+o(1),
\]
contradicting the other strict inequality in \eqref{ratio:gap}.
Thus neither alternative is possible. All $o(1)$ terms are taken only
after the fixed positive gaps have been established.

Fix $\nu>3$. Choose rational $b_0$ with
\[
                         2/3<b_0<1-1/\nu.
\]
For small rational $\delta>0$ put $\theta=1-\delta$ and
$A=1-b_0\delta$. Then
\[
                  \nu(A-\theta)>1-\theta,\qquad A^3<\theta^2.
\]
The first inequality is exactly $\nu(1-b_0)>1$, multiplied by
$\delta>0$. For the second, the function
$(1-b_0\delta)^3-(1-\delta)^2$ vanishes at zero and has derivative
$2-3b_0<0$ there. It is consequently negative for all sufficiently
small positive $\delta$. This is the only place that requires the
interval $2/3<b_0<1-1/\nu$ to be nonempty, and hence the condition
$\nu>3$.

The second strict inequality permits rational choices
\[
                    A<B<\theta^{2/3},\qquad
                    B/\theta<C<1/\sqrt B.
\]
The second interval is nonempty precisely when $B^3<\theta^2$.
Also $A>\theta$ because $b_0<1$, hence $B>\theta$ and
$B/\theta>1$. From $C<1/\sqrt B$ and $B>\theta>\theta^2$ one gets
$C\theta<\theta/\sqrt B<1$. These inequalities verify all four
bases of the exponential limits used below.

They satisfy $C>1$, $C\theta<1$, $C\sqrt B<1$, and $C\theta/B>1$.
Choose $\eta>0$ sufficiently small to preserve the positive gap,
and fix $F_0>2/\theta$.
One explicit admissible restriction is
$0<\eta<(\nu(A-\theta)-(1-\theta))/(2\nu A)$, with
$\eta<1$. It makes
$g=\nu(A-\theta)-(1-\theta)-\nu A\eta>0$ a fixed number independent
of $m$, the approximants, and $H$. Both $\eta$ and $F_0$ may be chosen
rational.
 For even integers $m$ set
\begin{equation}\label{ratio:parameter-values}
 K=\lfloor C^m\rfloor,\quad w_0=B^{-m/2},\quad
 v=\left\lceil\sqrt{2K\theta^mw_0^2}\right\rceil,
                  \qquad \sigma=\frac1{100(m+3)}.
\end{equation}
The weights are rational. Their relevant limits are
\begin{align*}
 K/w_0&\le(C\sqrt B)^m\longrightarrow0,\\
 v&\sim\sqrt2(C\theta/B)^{m/2}\longrightarrow\infty,\\
 w_0/v&\sim 2^{-1/2}(C\theta)^{-m/2}\longrightarrow\infty,\\
 (w_0/v)\theta^{m-1}
          &\sim(\sqrt2\theta)^{-1}(\theta/C)^{m/2}\longrightarrow0,\\
 L&\sim\frac{\eta^{3/2}(B/A)^{m/2}}{\sqrt{2(m+1)(m+2)}}
                                                   \longrightarrow\infty.
\end{align*}
Also $K(w_0/v)^2\theta^m\le1/2$. The factor
$(1+3\sigma)^{m+2}$ stays close to $e^{3/100}<2$, so all strict
geometric inequalities of Lemma~\ref{ratio:lem:two-torus} hold for
large $m$, including $w_0/v>4(m+3)/\sigma$. The remaining two
volume inequalities tend to zero exponentially. The errors
$R/w_0$ and $\Lambda F_0/v$ tend to zero, while $L$ tends to infinity.

We now fix one such finite dimension $m$. Only next do we choose
the $m$ rational approximants successively far out in the assumed
sequence, meeting every separation threshold and making $w_*^{-1}$
errors small. Finally $H$ tends to infinity through divisible values.
Thus \eqref{ratio:gap} is a contradiction for each $\nu>3$.
For an explicit ordering with margins, first enlarge the even $m$
until every geometric condition holds,
\[
 c_2L>3,
 \qquad R/w_0+(\Lambda F_0+\log2)/v
              <\min(g/4,1/4).
\]
With this $m$ fixed, $R$ and $B_R$ are fixed. Choose $w_1$, and then
the later $w_i$, far enough along the approximation sequence to satisfy
all successive separation thresholds and
\[
 (\log B_R+\nu+\log2)/w_*<\min(g/4,1/4).
\]
All later weights can be required to exceed the same common lower
bound, so enlarging them preserves this last inequality. Then
$E<\min(g/2,1/2)$ and $c_2L>1+E$.
Only now are the truncations, interpolation threshold, and nonzero
minors fixed degree by degree, and only now is $H$ sent to infinity.
There is no simultaneous limit in dimension, denominators, and degree.

For general positive multiplicatively independent rational $a,b$,
replace the centers by $(b^j,a^j,0)$ and put $\xi=\log a/\log b$.
The characters $(a^{q_i}/b^{p_i})^j$ are distinct. If $v_a a$ and
$v_b b$ are integral, an additional column factor
$v_b^{(K-1)h}v_a^{(K-1)k}$ clears the torus denominators. It adds
only $O(K/w_0)$ to the arithmetic error. The Cauchy radius changes
by a fixed constant depending on $|\log a|,|\log b|$. All limits
above persist.
Take $v_a,v_b$ to be fixed positive integers clearing the denominators
of $a,b$. The additional factor in an entry at center $j$ is
\[
 v_b^{(K-1)h}v_a^{(K-1)k}b^{jh}a^{jk}
 =(v_bb)^{jh}v_b^{(K-1-j)h}
  (v_aa)^{jk}v_a^{(K-1-j)k},
\]
which is integral since $0\le j\le K-1$. For the determinant its
logarithmic cost divided by $MH$ is at most
$(K-1)\max(\log v_a,\log v_b)/w_0$.
This tends to zero in the dimension selection, because
$K/w_0\to0$. For the analytic bounds one may replace $R$ by
$4K(1+|\log a|+|\log b|)$ and use the corresponding $B_R$.
The character $a^{q_i}/b^{p_i}$ remains positive and different from
one, so its powers still separate the centers. These facts cover also
positive rationals below one and ratios $\xi$ of either sign.
 Irrationality follows from multiplicative independence,
and Lemma~\ref{prelim:dirichlet} gives the lower bound two. This
completes Theorem~\ref{ratio:thm:ratio}.

\begin{corollary}[A fixed homogeneous two-logarithm measure]
\label{ratio:linear-form}
For positive multiplicatively independent rationals $a,b$ and every
$\varepsilon>0$, there is $c_{a,b,\varepsilon}>0$ such that
\[
 |p\log a+q\log b|\ge c_{a,b,\varepsilon}
                  \max(1,|p|,|q|)^{-2-\varepsilon}
                \quad((p,q)\in\Z^2\setminus\{(0,0)\}).
\]
\end{corollary}
\begin{proof}
For $p\ne0$, factor the form as $p\log b\,(\xi+q/p)$.
When $-q/p$ lies in a fixed bounded neighborhood of $\xi$,
the coefficient height is comparable to $|p|$ and the exponent-three
ratio bound gives the assertion after multiplication by $|p\log b|$.
Outside that neighborhood the form is bounded below by a fixed
positive multiple of the height, after separating the region where
$|q/p|$ is large. If $p=0$, use $|q\log b|$ directly. Adjusting the
constant handles the finitely many small denominators.
Here is a direct uniform version of that last step. The ratio theorem
implies, for every $\varepsilon>0$, a constant $c_0>0$ such that
\[
 |\xi-r/s|\ge c_0s^{-3-\varepsilon}
       \qquad(r\in\Z,\ s\in\Z_{\ge1}).
\]
Indeed infinitely many exceptions with unbounded reduced denominator
would give an irrationality exponent at least $3+\varepsilon$;
for bounded denominator the positive distances have a positive minimum
after multiplication by $s^{3+\varepsilon}$. Reducing a fraction only
strengthens this lower bound. Applying it with $s=|p|$ and the
corresponding signed numerator gives, for $p\ne0$,
\[
 |p\log a+q\log b|
 \ge c_0|\log b|\,|p|^{-2-\varepsilon}
 \ge c_0|\log b|\,\max(1,|p|,|q|)^{-2-\varepsilon}.
\]
For $p=0$ and $q\ne0$, the left side is at least $|\log b|$.
Thus $c_{a,b,\varepsilon}=|\log b|\min(c_0,1)$ is sufficient.
Multiplicative independence guarantees that none of the finitely many
forms used in choosing a constant vanishes.

\end{proof}
This measure concerns two homogeneous terms. It asserts no identical
exponent for a third, independently varying, integer constant term.
The uniform positive constant and the homogeneous two-logarithm
conclusion are now exported by
\nolinkurl{QuadraticPi/RationalTwoLogLinearForm.lean};
Appendix~\ref{app:formal-status} gives the exact declarations.
In particular, the completed endpoint is
\nolinkurl{QuadraticPi.rationalTwoLog_linearForm_lowerBound}.
The next section also has checked general number-field and algebraically normalized ratio endpoints, with their exact branch hypotheses.

\section{Conjugate ratios and algebraic normalizations}
\label{ratio:conjugate-section}

The $3d/s$ comparison and the algebraically normalized ratio families
below now have checked Lean endpoints. Their general number-field
interpolation, one-minor norm comparison, analytic estimates, and
parameter choices are assembled in the accompanying sources. Irrational
algebraic slopes are handled by the checked normalized interpolation
adapter. Appendix~\ref{app:formal-status} records the public declarations
and the exact matching, branch, and nonvanishing hypotheses.

The essential arithmetic point is to select one nonzero minor over
the number field and use its images at every embedding. Selecting
unrelated minors at different embeddings would not supply the norm
inequality used below. Exact matching of the chosen logarithm
branches, rather than matching their absolute values, supplies the
analytic smallness at each good embedding.

\begin{theorem}[Partial matching for a quotient of logarithms]
\label{ratio:matching}
Let $E$ have degree $d$ and let $a,b\in E^\times$.
Let $\xi\in\mathbb R\setminus\mathbb Q$ and suppose
there are $s$ distinct embeddings, with $1\le s\le d$,
$\sigma:E\hookrightarrow\mathbb C$ admitting
nonzero logarithms $\lambda_{a,\sigma},\lambda_{b,\sigma}$ with
\[
 e^{\lambda_{a,\sigma}}=\sigma(a),\qquad
 e^{\lambda_{b,\sigma}}=\sigma(b),\qquad
                 \lambda_{a,\sigma}=\xi\lambda_{b,\sigma}.
\]
Then
\[
                         2\le\mu(\xi)\le\frac{3d}{s}.
\]
\end{theorem}
\begin{proof}
Let $S$ be the specified set of good embeddings and put $D=d/s$.
Fix $\nu>3D$ and suppose there are infinitely many reduced fractions
$r=p/q$, with $q>0$, satisfying $|\xi-r|\le q^{-\nu}$.
Their denominators are unbounded and the fractions converge to $\xi$.
We shall choose finitely many distinct nonzero fractions
$r_i=p_i/q_i$ successively from this sequence, with
$w_i=\lceil\log q_i\rceil$ as large as required. As before, all
such choices precede the limit in $H$.

\emph{Exceptional slopes and interpolation.}
The elements $a,b$ cannot both be roots of unity. Indeed, every
logarithm of a root of unity is a rational multiple of $2\pi i$;
the quotient of two nonzero such logarithms is rational, contrary
to the hypothesis at any embedding in $S$. Consider the subgroup
\[
 \mathcal R=\{(p,q)\in\mathbb Z^2:
                         a^q/b^p\text{ is a root of unity}\}.
\]
If $(p_1,q_1),(p_2,q_2)\in\mathcal R$ are linearly independent,
eliminating $a$ and then $b$ from their two relations shows that
nonzero integer powers of both $a$ and $b$ are roots of unity.
Both elements would themselves be roots of unity. Thus
$\operatorname{rank}\mathcal R\le1$. All its elements with $q\ne0$
therefore have at most one rational slope $p/q$. This proves that
there is at most one exceptional rational slope for which
$a^q/b^p$ is a root of unity. A sequence converging to irrational
$\xi$ eventually avoids that slope and zero. We choose all $r_i$
from this tail; we may also require $|r_i|\le1+|\xi|$.

At the centers $(b^j,a^j,0)$, the character attached to $r_i$ takes
the values $(a^{q_i}/b^{p_i})^j$. They are pairwise distinct for
$0\le j<K$, since equality of two would make the base a root of
unity. These are exactly the character distinctions used in
Lemma~\ref{ratio:lem:two-torus}. The commuting frame, the separated
normal-basis inequalities, and the reduction of a character fiber
to a one-torus curve are unchanged over an algebraic closure of
$E$. No positivity or reality of the centers is used there. The
lemma consequently gives full row rank over that algebraic closure,
and hence a nonzero full-row minor over $E$.

\emph{One minor and its denominator.}
Use the truncations $G_i$, integers $L_i$, columns
$P=Y^hZ^kX^\gamma$, and rows $(j,\ell,\rho,\beta)$ of
Section~\ref{ratio:section}, with $3^j,2^j$ replaced by $b^j,a^j$.
The entry is
\[
 [t^\ell s^\rho u^\beta]\,
 P\bigl(b^j(1+t),a^j(1+s),
                   (G_i(s)-r_iG_i(t)+u_i)_i\bigr)\in E.
\]
For each sufficiently large divisible $H$, select one nonzero
$M\times M$ minor $\Delta_H$ and retain these same columns throughout
the proof. Lemma~\ref{ratio:shear} is a rational unitriangular row
change; it applies over $E$ and at every embedding. In particular,
the corresponding full-logarithm minor equals $\Delta_H$ exactly.
The row count and
\[
 \overline b_H=\frac1{MH}\sum_{\rm rows}w\cdot\beta,
                         \qquad0\le\overline b_H\le\theta
\]
are still those of \eqref{ratio:count}.

Choose positive integers $v_a,v_b$ such that $v_aa,v_bb$ are
algebraic integers. For the selected columns put
\[
 J_i=\sum_{\rm columns}\gamma_i-\sum_{\rm rows}\beta_i,\qquad
 C_H=\prod_i(q_iL_i)^{J_i}
       v_b^{(K-1)\sum_{\rm columns}h}
       v_a^{(K-1)\sum_{\rm columns}k}.
\]
A nonzero term of the determinant has $\beta_i\le\gamma_i$ at
each paired entry, so $J_i\ge0$. Multiply column $P$ by
$\prod_i(q_iL_i)^{\gamma_i}v_b^{(K-1)h}v_a^{(K-1)k}$ and row
$\beta$ by $\prod_i(q_iL_i)^{-\beta_i}$. The resulting entries
are algebraic integers: the additive coefficients are integral
because $q_iL_iG_i(s)-p_iL_iG_i(t)$ has integral coefficients,
and $j\le K-1$ ensures that the column torus factor clears
$b^{jh}a^{jk}$. Consequently
\[
                         C_H\Delta_H\in\mathcal O_E\setminus\{0\}.
\]
Set $\Lambda=4\log2$ and
\[
 E_{\rm ar}=\frac{\Lambda F_0}{v}
       +\frac{(K-1)\max(\log v_a,\log v_b)}{w_0}.
\]
Using $\log(q_iL_i)\le(1+\Lambda F_0/v)w_i$ and the column
weighted-degree bound gives
\[
 \frac{\log C_H}{MH}\le1-\overline b_H+E_{\rm ar}.
\]
The nonzero rational integer $N_{E/\mathbb Q}(C_H\Delta_H)$
has absolute value at least one. Thus the exact same minor satisfies
\begin{equation}\label{ratio:matching-norm}
 \sum_{\sigma:E\hookrightarrow\mathbb C}
       \frac{\log|\sigma(\Delta_H)|}{MH}
       \ge-\frac{d\log C_H}{MH}
       \ge-d(1-\overline b_H+E_{\rm ar}).
\end{equation}
Every embedding is counted once, including both members of a
nonreal conjugate pair. None of the embedded minors vanishes,
because each embedding is injective.

\emph{The good embeddings.}
For every selected column use the common entire families
\[
 f_{\gamma,P}(z,w)
 =[u^\gamma]P(e^z,e^w,w-\xi z+u_1,\ldots,w-\xi z+u_m).
\]
At $\sigma\in S$ their coefficient tests occur at
\[
 z=j\lambda_{b,\sigma}+\log(1+t),\qquad
 w=j\lambda_{a,\sigma}+\log(1+s).
\]
Here $w-\xi z=\log(1+s)-\xi\log(1+t)$, because the constant
term vanishes by the specified branch matching. In the equal
full-logarithm minor, the difference between its $i$th additive
coordinate and this analytic coordinate is precisely
$(\xi-r_i)\log(1+t)$. It does not depend on $\sigma$.
Therefore the translation identity \eqref{ratio:translation},
the scalar estimate \eqref{ratio:scalar}, and the bivariate Taylor
collision apply to this same selected minor at every good embedding.

For a uniform radius and growth constant take
\[
 R=4K\left(1+\max_{\sigma\in S}
             (|\lambda_{a,\sigma}|+|\lambda_{b,\sigma}|)\right),
                 \qquad B_R=2R(1+|\xi|).
\]
The testing points on $|t|=|s|=1/2$ lie in the radius-$R/2$
polydisc. The proof of \eqref{ratio:upper} consequently gives,
uniformly for $\sigma\in S$,
\begin{equation}\label{ratio:matching-good}
 \frac{\log|\sigma(\Delta_H)|}{MH}
 \le E_{\rm an}+o(1)+
 \max\{-c_2L+o(1),-\nu(A(1-\eta)-\overline b_H)\},
\end{equation}
where $c_2$ and $L$ are as in Lemma~\ref{ratio:collision} and
\eqref{ratio:L}, and a sufficient error is
\[
 E_{\rm an}=\frac R{w_0}+\frac{\log2}{v}
                +\frac{\log B_R+\nu+\log2}{w_*}.
\]
The $o(1)$ terms tend to zero as $H\to\infty$, after all other
parameters and fractions have been fixed.

\emph{The other embeddings.}
No matching logarithms are required here. Define the fixed constants
\[
 c_E=\max_{\sigma:E\hookrightarrow\mathbb C}
        \max(0,\log|\sigma(a)|,\log|\sigma(b)|),\qquad
 Q_\xi=1+(2+|\xi|)\log2.
\]
On $|t|=|s|=1/2$ one has $|G_i(t)|,|G_i(s)|\le\log2$.
On the additional circles $|u_i|=1$, each additive factor is
bounded by $Q_\xi$, using $|r_i|\le1+|\xi|$. Each torus factor
is bounded by $\frac32\exp((K-1)c_E)$. Cauchy's estimate for
$[t^\ell s^\rho u^\beta]$, followed by the determinant expansion,
therefore gives at every remaining embedding
\begin{equation}\label{ratio:matching-bad}
 \frac{\log|\sigma(\Delta_H)|}{MH}
       \le E_{\rm bad}+o(1),\qquad
 E_{\rm bad}=\frac{\log(3/2)+(K-1)c_E}{w_0}
                   +\frac{\log2}{v}+\frac{\log Q_\xi}{w_*}.
\end{equation}
Indeed $h+k\le H/w_0$, $\sum_i\gamma_i\le H/w_*$, and
$\ell+\rho<H/v$ bound each entry by $\exp(HE_{\rm bad})$.
The determinant factor contributes $\log(M!)/(MH)\le\log M/H=o(1)$
by \eqref{ratio:count}. This is a crude growth bound, not a claim
of analytic smallness at a bad embedding.

\emph{The averaged comparison.}
Combine \eqref{ratio:matching-norm}--\eqref{ratio:matching-bad}
and divide by $s$. With
\[
 E=DE_{\rm ar}+E_{\rm an}+(D-1)E_{\rm bad},
\]
it suffices to arrange
\begin{equation}\label{ratio:matching-gap}
 g=\nu(A(1-\eta)-\theta)-D(1-\theta)>0,\qquad
                          E<g,\qquad c_2L>D+E.
\end{equation}
To see this directly, the approximation alternative minus the
norm cost is
\[
 \nu(A(1-\eta)-\overline b_H)-D(1-\overline b_H)
 =\nu A(1-\eta)-D-(\nu-D)\overline b_H\ge g,
\]
since $\nu>D$ and $\overline b_H\le\theta$. The collision
alternative has norm cost at most $D$, since
$1-\overline b_H\le1$. Under \eqref{ratio:matching-gap} both
alternatives contradict the norm inequality for sufficiently large $H$.

\emph{The explicit order of parameter choices.}
Because $\nu>3D$, choose rational $b_0$ with
\[
                         \frac23<b_0<1-\frac D\nu.
\]
For sufficiently small positive rational $\delta$ put
$\theta=1-\delta$ and $A=1-b_0\delta$. Then
\[
 \nu(A-\theta)-D(1-\theta)
                 =\bigl(\nu(1-b_0)-D\bigr)\delta>0,
                  \qquad A^3<\theta^2.
\]
The second inequality follows since the derivative at $\delta=0$
of $(1-b_0\delta)^3-(1-\delta)^2$ is $2-3b_0<0$.
Choose $\eta>0$ small enough that the first strict margin persists,
and fix rational $F_0>2/\theta$. Choose rational $B,C$ satisfying
\[
              A<B<\theta^{2/3},\qquad B/\theta<C<1/\sqrt B.
\]
The interval for $C$ is nonempty because $B^3<\theta^2$.
As $A>\theta$, these choices give $C>1$, $C\theta<1$,
$C\sqrt B<1$, and $C\theta/B>1$. For even integers $m$ use
exactly \eqref{ratio:parameter-values}:
\[
 K=\lfloor C^m\rfloor,\quad w_0=B^{-m/2},\quad
 v=\left\lceil\sqrt{2K\theta^mw_0^2}\right\rceil,
                      \qquad\sigma=\frac1{100(m+3)}.
\]
The geometric inequalities hold for all sufficiently large $m$
by the calculation in Section~\ref{ratio:parameters}. More explicitly,
$K(w_0/v)^2\theta^m\le1/2$, the remaining volume expressions
tend to zero, and $w_0/v$ grows exponentially faster than the required
$4(m+3)/\sigma$. In addition,
\[
 \frac K{w_0}\longrightarrow0,\qquad v\longrightarrow\infty,
 \qquad
 L\sim\frac{\eta^{3/2}(B/A)^{m/2}}
                   {\sqrt{2(m+1)(m+2)}}\longrightarrow\infty.
\]
Thus the arithmetic torus cost, the common-radius cost $R/w_0$,
and all $1/v$ costs tend to zero, including their fixed factors
$D$ and $D-1$. First fix a sufficiently large finite even $m$
so these costs are smaller than the strict margin and $c_2L$ is
larger than the required norm cost. Next select the $m$ fractions
successively far out in the nonexceptional tail. This meets every
slope-independent interpolation threshold and makes all remaining
$w_*^{-1}$ errors small enough for \eqref{ratio:matching-gap}.
Finally let $H$ tend to infinity through sufficiently divisible values
to remove the $o(1)$ terms. This contradicts infinitely many
approximations of exponent $\nu$. Since $\nu>3d/s$ was arbitrary,
$\mu(\xi)\le3d/s$. Irrationality is a hypothesis, so
Lemma~\ref{prelim:dirichlet} supplies $\mu(\xi)\ge2$.
\end{proof}

\begin{corollary}[Full matching for a quotient of logarithms]
\label{ratio:full-matching}
Under the hypotheses of Theorem~\ref{ratio:matching}, if every
embedding of $E$ admits the specified nonzero logarithms with the
same real irrational quotient $\xi$, then
\[
                              2\le\mu(\xi)\le3.
\]
\end{corollary}
\begin{proof}
Take $s=d$ in Theorem~\ref{ratio:matching}. In this case the norm
averages precisely $d$ good-embedding estimates, with no
bad-embedding contribution. This is a bound between two and three;
the argument does not identify the exact exponent in general.
\end{proof}

\begin{corollary}[Quadratic norm-one ratios]\label{ratio:quadratic}
Let $E/\mathbb Q$ be quadratic with nontrivial involution $\tau$.
Let $a,b\in E^\times$ satisfy
$\tau(a)=a^{-1}$, $\tau(b)=b^{-1}$. If
$\lambda_a,\lambda_b$ are nonzero logarithms and
$\xi=\lambda_a/\lambda_b\in\mathbb R\setminus\mathbb Q$, then
\[
                              2\le\mu(\xi)\le3.
\]
\end{corollary}
\begin{proof}
At the given embedding use $\lambda_a,\lambda_b$; at its conjugate
use $-\lambda_a,-\lambda_b$. Their exponentials are $a^{-1},b^{-1}$
and their quotient is unchanged. Thus all $s=d=2$ embeddings are
good, and Corollary~\ref{ratio:full-matching} applies.
\end{proof}

For positive $a,b$ in a real quadratic field, individual norms
$\pm1$ suffice after replacing both by their squares. Their real
logarithm quotient is unchanged. Two positive integral units in a
real quadratic field have a rational real-logarithm quotient, because
the unit rank is one; norm-one elements with allowed denominators
can give irrational quotients. For instance
\[
 a=3+2\sqrt2,\qquad b=\frac{3+\sqrt2}{3-\sqrt2}
\]
are independent: $a$ is a unit, whereas $b$ has nonzero valuation
at a prime above $7$. A multiplicative relation first forces the
exponent of $b$ to vanish by that valuation, and then that of $a$
to vanish since $a>1$. A rational logarithm quotient would give
such a relation. Hence
\[
 2\le\mu\!\left(\frac{\log(3+2\sqrt2)}
                   {\log((3+\sqrt2)/(3-\sqrt2))}\right)\le3.
\]

\paragraph{An explicit pair of arctangents.}
In $\mathbb Q(i)$ take
\[
 a=\frac{2+i}{2-i},\qquad b=\frac{4+i}{4-i}.
\]
Choose $\lambda_a=2i\arctan(1/2)$ and
$\lambda_b=2i\arctan(1/4)$. Their conjugates are their negatives.
The Gaussian primes $2+i$ and $4+i$ have norms $5$ and $17$.
Their valuations in a relation $a^u b^v=\zeta$ for a root of unity
$\zeta$ force $u=0$ and then $v=0$. In particular the quotient
of the displayed logarithms is irrational: if it were $p/q$,
exponentiating $q\lambda_a=p\lambda_b$ would give
$a^q b^{-p}=1$. Therefore
\[
                         \boxed{\displaystyle
 2\le\mu\!\left(\frac{\arctan(1/2)}{\arctan(1/4)}\right)\le3.}
\]
Equivalently, for every $\varepsilon>0$ a positive constant $c$
satisfies
\[
 |p\arctan(1/2)+q\arctan(1/4)|
             \ge c\max(1,|p|,|q|)^{-2-\varepsilon}
\]
for all nonzero integer pairs $(p,q)$. This follows from the
exponent-three ratio bound by the same height comparison as in
Corollary~\ref{ratio:linear-form}.

\paragraph{Arctangent divided by the circular period.}
Let $r\in\mathbb Q\setminus\{0,1,-1\}$. In $\mathbb Q(i)$ put
\[
 a=\frac{1+ir}{1-ir},\qquad b=-1,\qquad
 \lambda_a=2i\arctan r,\qquad\lambda_b=i\pi.
\]
The only roots of unity in $\mathbb Q(i)$ are $1,-1,i,-i$;
the displayed $a$ is none of them for these finite $r$.
Therefore $2\arctan r/\pi$ is irrational, since a rational value
would make $a$ a root of unity. The two embeddings are good,
using the negative logarithms at complex conjugation. Apply
Corollary~\ref{ratio:full-matching} to $2\arctan r/\pi$ and then
the rational-scaling case of Lemma~\ref{prelim:mobius}. Hence
\[
                        \boxed{\displaystyle
                2\le\mu(\arctan r/\pi)\le3.}
\]
This is a statement for each fixed $r$; no uniform dependence of
the approximation constant on its height is asserted.

More generally, an involution in a degree-$2m$ field that inverts
both $a$ and $b$ supplies two good embeddings and gives
$2\le\mu(\lambda_a/\lambda_b)\le3m$. This includes relative
norm-one elements in quadratic extensions of higher-degree fields
whenever the specified logarithm quotient is real and irrational.
The upper bound improves to three if the required matching extends
to every embedding. In contrast with the single-logarithm scalar
matching problem, no unconditional classification of all possible
ratio-matching conjugates is asserted here.

\subsection{Algebraic slope normalization and mixed circular--real logarithms}

The irrational-algebraic-slope interpolation, normalized matching
ratios, and mixed circular--real applications in this subsection have
checked Lean endpoints. The endpoint map in
Appendix~\ref{app:formal-status} distinguishes their actual branch and
embedding assumptions from the rational-base special case.

\begin{lemma}[Irrational algebraic slopes in the two-torus interpolation]
\label{ratio:irrational-slope}
In the two-torus interpolation statement, one may replace all distinct
nonzero rational slopes $r_i$ by $cr_i$, where $c$ is any fixed
irrational algebraic number.  Suppose the paired centers $(b^j,a^j)$
are distinct for all $j\ge0$.  No exceptional character slope has
to be removed.  The same sufficient weight inequalities
and slope-independent successive affine-weight thresholds apply.
\end{lemma}
\begin{proof}
Work over an algebraic closure of the field containing the data.  The
commuting frame becomes
\[
 D_Y=Y\partial_Y-\sum_i cr_i\partial_{X_i},\qquad
 D_Z=Z\partial_Z+\sum_i\partial_{X_i},\qquad
 D_i=\partial_{X_i}.
\]
The coefficients $cr_i$ are constants.  Consequently the frame still
preserves weighted degree, and in the local parameters
\[
 t=Y/b^j-1,\quad s=Z/a^j-1,\quad
 u_i=X_i-\log(1+s)+cr_i\log(1+t)
\]
it is $(1+t)\partial_t,(1+s)\partial_s,\partial_{u_i}$.
The dimension count, zero--pole comparison, persistent-component
construction, and normal-basis inequalities in the proof of
Lemma~\ref{ratio:lem:two-torus} therefore apply without any new
coefficient-dependent weight restriction.

If every affine coordinate direction is tangent to the persistent
component $T$, the same torus-weight contradiction finishes the
argument.  Otherwise choose the largest index $i$ whose affine
coordinate direction is nontangent.  The separated normal-basis
argument says that every frame normal basis must contain $D_i$.
Equivalently, the tangent space of $T$ is contained in
\[
             \ker(dX_i+cr_i\,dY/Y-dZ/Z).
\]
Thus the displayed differential vanishes on $T$.

Here the irrationality of $cr_i$ gives a stronger conclusion than
the rational-character argument.  Restrict to an algebraic curve
meeting the smooth open part of $T$, and pass to its complete smooth
normalization.  At every point $P$ the residue of an exact
differential is zero, whereas the residue of $dY/Y$ is
$\operatorname{ord}_P Y$.  Hence
\[
             cr_i\operatorname{ord}_P Y=\operatorname{ord}_P Z.
\]
Both orders are integers and $cr_i\notin\mathbb Q$, since
$r_i$ is a nonzero rational number.  Both orders must therefore
vanish.  The rational functions $Y$ and $Z$ have no zeros or poles
on the complete curve, so they are constant.  The differential
identity then gives $dX_i=0$, and characteristic zero makes $X_i$
constant as well.

These functions are consequently constant on $T$: if one were
nonconstant, a curve through a general smooth point, chosen with
tangent direction on which its differential is nonzero, would give
a nonconstant restriction, contrary to the preceding argument.
In particular $Y,Z,X_i$ are constant on the putative violating
curve $C\subseteq T$.  Since the paired centers are distinct, $C$
can meet only one center, say $P_0$.

On $C$ both torus shifts and all their logarithmic shifts vanish.
Thus the remaining local parameters restrict to
$u_l=X_l-X_l(P_0)$.  Some $X_l$ is nonconstant, because otherwise
every ambient coordinate would be constant on the curve.  On the
complete normalization of $C$, its branch contacts satisfy
\[
 \sum_{P\mapsto P_0}h_P
 \le \frac{\theta}{w_l}
       \sum_{P\mapsto P_0}\operatorname{ord}_P(X_l-X_l(P_0))
 \le \theta\frac{\deg X_l}{w_l}
 \le\theta\,d_W(C).
\]
The middle inequality bounds the selected zeros by all zeros of
$X_l-X_l(P_0)$, whose total multiplicity is the pole degree
$\deg X_l$.  The last inequality follows directly from the
definition of the weighted pole degree $d_W(C)$.  Since
$(1+\sigma)\theta<1$, this contradicts
$d_W(C)<(1+\sigma)\sum h_P$.  No character-fiber projection is
needed in this case.

Finally, replacing the logarithms by sufficiently long Taylor
polynomials changes each $u_i$ only by terms of strictly larger
weight; multiplication by the fixed scalar $cr_i$ does not change
that fact.  The weighted jet ideals are preserved.  The same
curve-to-jet argument as in Theorem~\ref{geom:thm:main} therefore
proves the required interpolation surjection.  The determinant
equality in Lemma~\ref{ratio:shear} also holds over the present
coefficient field, because its proof is a triangular change of
rows with diagonal one.
\end{proof}

\begin{theorem}[Normalized quotient with matching embeddings]
\label{ratio:normalized}
Let $a,b,c\in E^\times$, $c\notin\mathbb Q$,
$d=[E:\mathbb Q]$, and suppose $a,b$
are not both roots of unity.  Let $x\in\mathbb R\setminus\mathbb Q$.
Suppose that $s\ge1$ embeddings $\sigma:E\hookrightarrow\mathbb C$
admit nonzero logarithms satisfying
\[
 e^{\lambda_{a,\sigma}}=\sigma(a),\quad
 e^{\lambda_{b,\sigma}}=\sigma(b),\quad
             \lambda_{a,\sigma}=\sigma(c)x\lambda_{b,\sigma}.
\]
Then
\[
                             2\le\mu(x)\le3d/s.
\]
\end{theorem}
\begin{proof}
Dirichlet's theorem, Lemma~\ref{prelim:dirichlet}, gives the lower
bound because $x$ is real and irrational.  For the upper bound,
suppose that an exponent $\nu>3d/s$ admits arbitrarily large reduced
approximants $r_i=p_i/q_i$ with
$|x-r_i|\le q_i^{-\nu}$.  Select distinct nonzero approximants
successively as in the rational-slope proof, and put
$w_i=\lceil\log q_i\rceil$.

Use the matrix of that proof with torus centers $(b^j,a^j)$ and
slopes $cr_i$.  Its entries lie in $E$.  The centers are distinct:
equality for two different indices would make both $a$ and $b$
roots of unity.  The preceding lemma therefore supplies a nonzero
full-row minor $\Delta_H\in E$, with one fixed set of selected
columns.  Those columns remain fixed throughout the norm argument.

We spell out the additional denominator clearing.  Choose positive
integers $u,v_a,v_b$ such that $uc,v_aa,v_bb$ are algebraic
integers.  With the previous $L_i$ and $G_i$, set
$U_i=uq_iL_i$.  Then
\[
 U_i\bigl(G_i(s)-cr_iG_i(t)\bigr)
       =uq_iL_iG_i(s)-p_i(uc)L_iG_i(t)
\]
has algebraic-integer coefficients.  Scale a column
$Y^hZ^kX^\gamma$ by
\[
 v_b^{(K-1)h}v_a^{(K-1)k}\prod_iU_i^{\gamma_i},
\]
and a row of transverse index $\beta$ by
$\prod_iU_i^{-\beta_i}$.  Every resulting entry is an algebraic
integer: the torus factors clear $(b^j)^h(a^j)^k$ for $j<K$,
and the additive factors clear precisely the powers remaining
after $[u^\beta]$ extraction.

As before, a nonzero determinant term gives
$J_i=\sum_{\rm columns}\gamma_i-\sum_{\rm rows}\beta_i\ge0$.
Thus the total determinant multiplier $C_H$ is a positive rational
integer, and $C_H\Delta_H$ is a nonzero algebraic integer.  In the
notation of the preceding section,
\[
 \frac{\log C_H}{MH}
 \le 1-\overline b_H+\frac{\Lambda F_0}{v}
       +\frac{\log u}{w_*}+O(K/w_0),
       \qquad 0\le\overline b_H\le\theta.
\]
Indeed the extra $u$ contributes
$(\log u)\sum_iJ_i/(MH)\le(\log u)/w_*$.  No degree factor
is introduced until the field norm is taken.

At a good embedding $\sigma$, replace the truncated logarithms
by full logarithms in this same selected minor.  The shear adds
only terms of weight greater than $w_i/\theta$, so its coefficient
matrix is triangular with diagonal one over $E$.  Applying
$\sigma$ preserves this determinant identity.  Define
\[
 f_{\gamma,P,\sigma}(z,w)
 =[v^\gamma]P(e^z,e^w,
             w-\sigma(c)xz+v_1,\ldots,w-\sigma(c)xz+v_m).
\]
The tests occur at
\[
 z=j\lambda_{b,\sigma}+\log(1+t),\qquad
 w=j\lambda_{a,\sigma}+\log(1+s).
\]
By the matching hypothesis the constant term of
$w-\sigma(c)xz$ is zero.  The exact difference between the full
matrix's $i$th additive coordinate and this common analytic
coordinate is
\[
                         \sigma(c)(x-r_i)\log(1+t).
\]
Consequently the previous finite translation identity and
bivariate Taylor collision estimate apply.  The extra factor
$\sigma(c)$ in each translation power contributes only
$O(1/w_*)$ to the normalized logarithmic bound.  The fixed complex
logarithms merely change the constant multiplying the Cauchy
radius $K$.  Entire-function coefficient vectors need only be
common across rows within a single embedding; the determinant
argument never requires them to agree at different embeddings.

There are finitely many embeddings.  We may therefore use one
upper error $E_{\rm an}$ for all good embeddings, obtaining
\[
 \frac{\log|\sigma\Delta_H|}{MH}
 \le E_{\rm an}+o(1)+
 \max\{-c_2L+o(1),-\nu(A(1-\eta)-\overline b_H)\}.
\]
At a bad embedding, bound the truncated matrix directly on
$|t|=|s|=1/2$ and $|u_i|=1$.  The truncated logarithms are bounded
uniformly in their lengths, the slopes $\sigma(c)r_i$ are bounded
once the approximants lie in a fixed neighborhood of $x$, and the
torus centers have logarithmic size $O(K)$.  Coefficient extraction
and the determinant inequality then give
\[
 \frac{\log|\sigma\Delta_H|}{MH}
      \le E_{\rm bad}+o(1),\qquad
 E_{\rm bad}=O(K/w_0+1/v+1/w_*).
\]
All constants here depend only on the fixed algebraic data and
chosen logarithms, not on $H$ or the heights of the selected
approximants.

Put $D=d/s$.  The nonzero integer norm satisfies
\[
 0\le\log\bigl|N_{E/\mathbb Q}(C_H\Delta_H)\bigr|
       =d\log C_H+\sum_\sigma\log|\sigma\Delta_H|.
\]
Using the preceding estimates and dividing by $sMH$ gives the
same two sufficient strict comparisons as in
Theorem~\ref{ratio:matching}:
\[
 \nu(A(1-\eta)-\theta)>D(1-\theta)+E,
                  \qquad c_2L>D+E,
\]
where $E=D E_{\rm ar}+E_{\rm an}+(D-1)E_{\rm bad}$ and
$E_{\rm ar}$ includes the displayed denominator costs.  For the
first comparison the gap at row grade $\overline b_H$ is
\[
 \nu A(1-\eta)-D-(\nu-D)\overline b_H.
\]
Because $\nu>D$, its minimum on $[0,\theta]$ occurs at $\theta$.
The second comparison uses $D(1-\overline b_H)\le D$.

Finally choose $2/3<b_0<1-D/\nu$, put
$\theta=1-\delta$ and $A=1-b_0\delta$, and take $\delta>0$
small.  Then $\nu(A-\theta)>D(1-\theta)$ and $A^3<\theta^2$.
Take $\eta$ small enough to retain the first strict gap.  The
parameter choices of Section~\ref{ratio:parameters} make $L$
arbitrarily large and $K/w_0,1/v$ arbitrarily small.  After fixing
that finite dimension, choose the approximants successively to
meet the interpolation thresholds and make all $1/w_*$ errors
small; finally let $H$ tend to infinity.  Both comparisons hold,
contradicting the norm inequality.  This excludes every
$\nu>3d/s$ and proves the upper bound.
\end{proof}

For nonzero rational $c$, the same conclusion follows from the
unnormalized quotient theorem applied to $cx$, and the
rational-scaling case of Lemma~\ref{prelim:mobius}.

\begin{corollary}[A circular period divided by a rational logarithm]
\label{ratio:pi-log}
For every positive rational $a\ne1$ and every positive rational $D$,
\[
                         \boxed{\displaystyle
             2\le\mu\!\left(\frac{\pi}{\sqrt D\log a}\right)\le3.}
\]
In particular $2\le\mu(\pi/\log2)\le3$.
\end{corollary}
\begin{proof}
Write $x=\pi/(\sqrt D\log a)$.  The denominator is nonzero, so
$x$ is real and nonzero.  First verify the irrationality required
by Theorem~\ref{ratio:normalized}.  If $x$ were algebraic, then
$i\sqrt D\,x$ would be a nonreal algebraic number, hence an
algebraic irrational number.  Gel'fond--Schneider applied to the
algebraic base $a\ne0,1$ says that its power defined using the
real logarithm of $a$ is transcendental.  But that power is
\[
       \exp(i\sqrt D\,x\log a)=\exp(i\pi)=-1,
\]
a contradiction.  Thus $x$ is transcendental, in particular
irrational.

Now work in the imaginary quadratic field
$E=\mathbb Q(\sqrt{-D})$, using torus bases $-1,a$ and
normalization $c=i\sqrt D$.  The second torus base is not a root
of unity, since it is positive and different from one.  At the
identity choose logarithms $i\pi,\log a$.  At conjugation choose
$-i\pi,\log a$; conjugation also sends $c$ to $-c$.  Thus at
both embeddings the first logarithm equals $\sigma(c)x$ times
the second.  Both embeddings are good, so $s=d=2$ and the
theorem gives $2\le\mu(x)\le3$.
\end{proof}

For $D=1$ the arithmetic also has a direct description: the
cleared nonzero determinant lies in $\mathbb Z[i]$, whence
\[
                         |C_H\Delta_H|\ge1.
\]
More generally, a nonzero algebraic integer $\alpha$ in an
imaginary quadratic field has
$|\alpha|^2=N_{E/\mathbb Q}(\alpha)\in\mathbb Z_{>0}$.
Thus $|\alpha|\ge1$, and the principal arithmetic cost is
$1-\overline b_H$, exactly as in the rational two-logarithm
argument.  This agrees with the averaged-norm calculation at
$d/s=1$.

\begin{corollary}[Arctangent divided by a rational logarithm]
\label{ratio:atan-log}
For every nonzero rational $r$ and every positive rational $a\ne1$,
\[
                         \boxed{\displaystyle
              2\le\mu\!\left(\frac{\arctan r}{\log a}\right)\le3.}
\]
\end{corollary}
\begin{proof}
Set $x=\arctan r/\log a$, which is real and nonzero.  If $x$
were algebraic, the exponent $2ix$ would be nonreal algebraic
and therefore irrational.  Gel'fond--Schneider would make
$\exp(2ix\log a)$ transcendental, whereas
\[
          \exp(2ix\log a)=\exp(2i\arctan r)
                         =\frac{1+ir}{1-ir}\in\mathbb Q(i).
\]
This contradiction proves transcendence, hence irrationality.

In $E=\mathbb Q(i)$ use torus bases $(1+ir)/(1-ir)$ and $a$,
normalization $c=2i$, and logarithms $2i\arctan r,\log a$.
They are both nonzero.  Conjugation negates the first logarithm
and $c$, while preserving the second logarithm.  Hence both
embeddings give the same normalized target $x$.  The base $a$
is not a root of unity, even when the first base is one (as can
happen for $r=\pm1$).  The normalized theorem applies with
$s=d=2$ and proves both bounds.
\end{proof}

\begin{corollary}[A real quadratic logarithm divided by a rational logarithm]
\label{ratio:quadratic-log}
Let $D>0$ be rational and nonsquare, and let
$a\in\mathbb Q(\sqrt D)$ be positive, different from $1$, and
have norm $1$ or $-1$.  For every positive rational $b\ne1$,
\[
                         \boxed{\displaystyle
       2\le\mu\!\left(\frac{\log a}{\sqrt D\,\log b}\right)\le3.}
\]
In particular, writing $\varphi=(1+\sqrt5)/2$,
\[
 2\le\mu\!\left(\frac{\log(1+\sqrt2)}{\sqrt2\log2}\right)\le3,
 \qquad
 2\le\mu\!\left(\frac{\log\varphi}{\sqrt5\log2}\right)\le3.
\]
\end{corollary}
\begin{proof}
First suppose $N_{E/\mathbb Q}(a)=1$, where
$E=\mathbb Q(\sqrt D)$, and write
$x=\log a/(\sqrt D\log b)$.  This is a nonzero real number.
If $x$ were algebraic and $\sqrt D\,x$ were irrational, then
Gel'fond--Schneider applied to $b$ would contradict
\[
                  \exp(\sqrt D\,x\log b)=a.
\]
The remaining possibility is a nonzero rational exponent
$\sqrt D\,x=u/v$, with integers $u\ne0$ and $v>0$.  The real
logarithm equality then gives $a^v=b^u$.  Taking quadratic norms
gives
\[
                 1=N(a)^v=N(b)^u=b^{2u},
\]
which is impossible for positive $b\ne1$ and $u\ne0$.
Thus $x$ is transcendental and in particular irrational.

Use torus bases $a,b$ and normalization $c=\sqrt D$.
At the identity take real logarithms $\log a,\log b$.
The other conjugate of $a$ is $a^{-1}>0$, while $b$ is fixed;
at that embedding take $-\log a,\log b$.  Since conjugation
also negates $c$, both embeddings satisfy the matching identity
for $x$.  The rational base $b$ is not a root of unity.
Theorem~\ref{ratio:normalized}, with $s=d=2$, therefore gives
$2\le\mu(x)\le3$.

If $N(a)=-1$, put $a'=a^2$.  Then $a'>0$, $a'\ne1$, and
$N(a')=1$.  The norm-one case already proved that
$\log a'/(\sqrt D\log b)=2x$ is transcendental; hence $x$
is also transcendental.  Apply the normalized theorem directly
to torus bases $a',b$ and normalization $c'=2\sqrt D$.
At the two embeddings the first logarithms are
$2\log a$ and $-2\log a$, respectively, while the second is
$\log b$ and $c'$ changes sign.  They therefore match the
unchanged target $x$, proving the same two bounds.
Finally $1+\sqrt2$ and $\varphi$ have norm minus one, which
establishes the displayed examples.
\end{proof}

For example, the circular-period corollary with $D=1$ and $a=2$
is equivalent to the following fixed two-term bound: for every
$\varepsilon>0$ a constant $C_\varepsilon>0$ satisfies
\[
 |p\pi+q\log2|
    \ge C_\varepsilon\max(1,|p|,|q|)^{-2-\varepsilon}
                 \qquad((p,q)\in\mathbb Z^2\setminus\{(0,0)\}).
\]
For $p\ne0$, factor the form as
$p\log2\,(\pi/\log2+q/p)$ and apply the exponent-three bound;
when $-q/p$ stays near the fixed target, the height is comparable
to $|p|$.  Outside such a neighborhood, the form is bounded below
by a positive multiple of the height after also separating the
region of large $|q/p|$.  The case $p=0$ is immediate.  Conversely,
substitution of $(p,q)=(q,-p)$ in this linear-form inequality and
division by $q|\log2|$ gives the corresponding exponent-three
lower bounds for rational approximation.  These statements concern
fixed bases; no uniform dependence of $C_\varepsilon$ on their
heights is asserted.

\section{Quadratic logarithmic families: arithmetic and analytic proofs}
\label{extquad:section}

This section supplies the arithmetic and analytic details for the quadratic
families used in the preceding discussion. The arithmetic assertions concern
rational parameters and ordinary rational approximation to real numbers.
The analytic assertions hold more generally for nonzero real parameters,
on the domains stated below. Each proof is accompanied by its exact Lean
source correspondence. The general interpolation and determinant theorem
is used through its established scalar-matching consequence; the work here
is to construct its algebraic data, verify every embedding, and identify
the actual real and complex functions to which it applies.

\subsection{The common-scalar principle and rational linear combinations}
\label{extquad:matching}

For clarity, the all-embedding consequence needed throughout this section
has the following form. Let $E$ be a number field, let $A,c\in E$ with
$c\ne0$, and let $x\in\R\setminus\{0\}$. Suppose that
\begin{equation}\label{extquad:all-matching}
 \exp\bigl(\sigma(c)x\bigr)=\sigma(A)
 \qquad\text{for every embedding }\sigma:E\hookrightarrow\C.
\end{equation}
Then, for every $\nu>2$, there is an integer $Q\ge2$ such that
\begin{equation}\label{extquad:all-denominators}
 q^{-\nu}\le\left|x-\frac pq\right|
 \qquad(p\in\Z,\ q\in\Z,\ q\ge Q).
\end{equation}
This is the precise all-denominator lower bound used here. It first implies
that $x$ is irrational: a rational $x$ would have exact representations
with arbitrarily large denominator multiples. It then implies
$\mu(x)\le2$, and the real Dirichlet lower bound gives $\mu(x)=2$.

The scalar $c$ need not be an algebraic integer. Choose a positive integer
$M$ for which $Mc$ is integral. Equation~\eqref{extquad:all-matching} gives
\[
 \exp\bigl(\sigma(Mc)x\bigr)=\sigma(A^M),
\]
so the integral-scalar theorem applies without changing $x$. In particular,
none of the applications below presupposes that a rational square root
is integral, or that multiplication by an arbitrary algebraic number
preserves an irrationality exponent.

\paragraph{Lean correspondence.}
The endpoint is
\path{QuadraticPi.scalarMatchingLogarithm_eventualLowerBound}
in \path{QuadraticPi/FullEmbeddingLogarithmTheorem.lean}; its hypotheses
are exactly the nonzero scalar, nonzero real target, and the identities
for all complex embeddings above. The integral-denominator reduction is
\path{QuadraticPi.ScalarMatchingPower.exists_integral_matching_power}.
The property \eqref{extquad:all-denominators} is
\path{QuadraticPi.EventualLowerBound}, defined in
\path{QuadraticPi/ExponentConclusion.lean}; the conclusion uses
\path{QuadraticPi.irrationalityExponent_eq_two_of_eventualLowerBound}.

\begin{proposition}[Rational closure for one common scalar]
\label{extquad:common-scalar-closure}
Let $I$ be a finite index set. Suppose $c\in E\setminus\{0\}$,
$A_j\in E$, and $x_j\in\R$ satisfy
\[
 \exp\bigl(\sigma(c)x_j\bigr)=\sigma(A_j)
 \qquad(j\in I,\ \sigma:E\hookrightarrow\C).
\]
For $q_j\in\Q$, put $X=\sum_{j\in I}q_jx_j$. If $X\ne0$, then
$X$ satisfies \eqref{extquad:all-denominators} and $\mu(X)=2$.
\end{proposition}
\begin{proof}
Choose one positive integer $N$ with $Nq_j=k_j\in\Z$ for every $j$.
Each $A_j$ is nonzero, since its image under a complex embedding is an
exponential. Define $B=\prod_{j\in I}A_j^{k_j}\in E$. For each $\sigma$,
\begin{align*}
 \exp\bigl(\sigma(Nc)X\bigr)
 &=\exp\left(\sum_{j\in I} k_j\sigma(c)x_j\right)\\
 &=\prod_{j\in I}\exp\bigl(\sigma(c)x_j\bigr)^{k_j}
 =\prod_{j\in I}\sigma(A_j)^{k_j}=\sigma(B).
\end{align*}
These equalities hold also for negative $k_j$, by the exponential identity
$\exp(-z)=\exp(z)^{-1}$. Since $Nc\ne0$ and $X\ne0$, the common-scalar
principle applies. The argument is an identity for the actual real sum
$X$; it requires no addition formula for principal logarithms.
\end{proof}

\paragraph{Lean correspondence.}
The integer $N$ and the coefficients $k_j$ are supplied by
\path{QuadraticPi.RationalSumDenominators.exists_common_positive_denominator}
in \path{QuadraticPi/RationalSumDenominators.lean}. The product and its
embedding identities are proved by
\path{QuadraticPi.scalarMatching_rational_sum}; the lower bound and exact
exponent are
\path{QuadraticPi.scalarMatching_rational_sum_eventualLowerBound} and
\path{QuadraticPi.scalarMatching_rational_sum_irrationalityExponent_eq_two},
all in \path{QuadraticPi/ScalarMatchingLinearCombination.lean}.

The following conclusion also appears in
\cite[Corollary~6]{LiuJiangZhang2026}.
\begin{corollary}[Rational linear combinations of rational logarithms]
\label{p1:rational-log-sums}
Let $a_1,\ldots,a_n$ be positive rational numbers and
$r_1,\ldots,r_n\in\Q$. If $X=\sum_jr_j\log a_j\ne0$, then
$\mu(X)=2$.
\end{corollary}
\begin{proof}
Choose $N\ge1$ with $k_j=Nr_j\in\Z$ and set
$A=\prod_ja_j^{k_j}\in\Q_{>0}$. The real product and integer-power
identities give $NX=\log A$. In particular $A\ne1$.
The rational-logarithm case of Theorem~\ref{slog:thm:real} gives
$\mu(\log A)=2$, and invariance under multiplication by the nonzero
rational number $1/N$ gives the claim. Alternatively this is the case
$E=\Q$, $c=1$, $x_j=\log a_j$ of
Proposition~\ref{extquad:common-scalar-closure}.
\end{proof}
The latter formulation is a direct specialization of the completed Lean
theorem \path{QuadraticPi.scalarMatching_rational_sum_irrationalityExponent_eq_two}.
Thus a linear combination is permitted to involve several distinct
rational arguments. The hypothesis is that its actual value is nonzero,
not that each coefficient or each logarithm is nonzero.

\subsection{Normalized real quadratic logarithms and the actual field norm}
\label{extquad:real-quadratic}

\begin{theorem}[The coordinate form, including rational square parameters]
\label{extquad:coordinate-theorem}
Let $D\in\Q$, $D>0$, and $u,v\in\Q$. Put
$\alpha=u+v\sqrt D$. If
\begin{equation}\label{extquad:coordinate-assumptions}
 \alpha>0,\qquad \alpha\ne1,\qquad (u^2-Dv^2)^2=1,
\end{equation}
then
\[
 \mu\left(\frac{\log\alpha}{\sqrt D}\right)=2.
\]
The target also satisfies the all-denominator bound
\eqref{extquad:all-denominators}.
\end{theorem}
\begin{proof}
Take the actual field $E=\Q(\sqrt D)\subset\R$ and write
$\delta=\sqrt D\in E$. Use
\[
 x=\frac{\log\alpha}{\sqrt D},\qquad
 c=2\delta,\qquad A=(u+v\delta)^2.
\]
The parameter condition gives $\delta\ne0$, and positivity together with
$\alpha\ne1$ gives $x\ne0$. Every embedding $\sigma:E\hookrightarrow\C$
satisfies $\sigma(\delta)^2=D$, hence
$\sigma(\delta)=\sqrt D$ or $-\sqrt D$.

If $\sigma(\delta)=\sqrt D$, then
$\sigma(A)=\alpha^2=\exp(2\log\alpha)=\exp(\sigma(c)x)$.
For the other sign, put $\alpha'=u-v\sqrt D$. The squared norm condition
gives
\[
 \alpha^2(\alpha')^2=(u^2-Dv^2)^2=1,
 \qquad (\alpha')^2=\alpha^{-2}.
\]
Consequently
$\sigma(A)=\alpha^{-2}=\exp(-2\log\alpha)
=\exp(\sigma(c)x)$. This establishes
\eqref{extquad:all-matching} for every embedding.

Squaring $\alpha$ is essential to this proof when $u^2-Dv^2=-1$:
the conjugate of $\alpha$ is then $-\alpha^{-1}$, whereas the conjugate
of $A$ is still exactly $A^{-1}$. No logarithm of a negative real
conjugate is introduced. If $D$ is a rational square, $E$ has degree one;
the same proof checks its actual embedding and does not posit a second
embedding that the field does not possess.
\end{proof}

\paragraph{Lean correspondence.}
In \path{QuadraticPi/RationalNormalizedQuadraticLogarithmTheorem.lean},
the namespace \path{QuadraticPi.RationalNormalizedQuadraticLogField}
defines \path{FieldOf}, \path{delta}, \path{gamma}, \path{algebraicArgument},
and \path{target}. Its lemmas \path{embedding_delta},
\path{argument_sq_mul_conjugate_sq}, \path{conjugate_sq_eq_inv}, and
\path{all_embeddings_match} prove the steps above. The final declarations are
\path{QuadraticPi.rationalNormalizedQuadraticLogarithm_eventualLowerBound},
\path{QuadraticPi.rationalNormalizedQuadraticLogarithm_irrational}, and
\path{QuadraticPi.rationalNormalizedQuadraticLogarithm_irrationalityExponent_eq_two}.

\begin{theorem}[The actual quadratic-field norm formulation]
\label{extquad:actual-norm}
Let $D>0$ be rational and not a square in $\Q$, and let
$E=\Q(\sqrt D)\subset\R$. For $\alpha\in E$ with
\[
 \alpha>0,\qquad \alpha\ne1,\qquad
 N_{E/\Q}(\alpha)\in\{1,-1\},
\]
one has $\mu(\log\alpha/\sqrt D)=2$ and
\eqref{extquad:all-denominators}.
\end{theorem}
\begin{proof}
Here the field norm is the determinant of multiplication by $\alpha$
on the actual two-dimensional $\Q$-vector space $E$.
The map from the quadratic algebra $\Q[T]/(T^2-D)$ to $E$ sending
$T$ to $\sqrt D$ is surjective because its image contains the field
generator. The nonsquare hypothesis makes the source a field, so this
nonzero homomorphism is injective and is an algebra isomorphism.
Thus $\alpha=u+v\sqrt D$ for rational $u,v$. Relative to the basis
$1,\sqrt D$, multiplication by $\alpha$ has matrix
\[
 \begin{pmatrix}u&Dv\\v&u\end{pmatrix}.
\]
Its determinant is $u^2-Dv^2$. Transport of the determinant norm through
the algebra isomorphism therefore gives
$N_{E/\Q}(\alpha)=u^2-Dv^2$.
Either allowed value of the field norm implies
$(u^2-Dv^2)^2=1$, and Theorem~\ref{extquad:coordinate-theorem} applies.
\end{proof}

\paragraph{Lean correspondence.}
The actual field and the coordinate algebra are related by
\path{QuadraticPi.RealQuadraticFieldCoordinates.coordinateEquiv}
in \path{QuadraticPi/RealQuadraticFieldCoordinates.lean}.
The module \path{QuadraticPi/RealQuadraticFieldNorm.lean} proves
\path{QuadraticPi.RealQuadraticFieldCoordinates.quadraticAlgebra_norm}
and \path{QuadraticPi.RealQuadraticFieldCoordinates.exists_coordinates_norm}
using the determinant-defined \path{Algebra.norm} and its invariance under
algebra equivalence. The endpoints in
\path{QuadraticPi/RealQuadraticLogarithmNormTheorem.lean} are
\path{QuadraticPi.realQuadraticLogarithm_norm_eventualLowerBound},
\path{QuadraticPi.realQuadraticLogarithm_norm_irrational}, and
\path{QuadraticPi.realQuadraticLogarithm_norm_irrationalityExponent_eq_two}.

\begin{remark}\label{extquad:square-distinction}
The nonsquare hypothesis in Theorem~\ref{extquad:actual-norm} cannot simply
be deleted from its norm interpretation. For square $D$, the embedded
field $\Q(\sqrt D)$ is $\Q$, and its field norm is the identity, rather
than a determinant in two independent coordinates. Theorem
\ref{extquad:coordinate-theorem} is the separately proved statement that
allows square $D$, with its explicit coordinate condition. Neither
theorem assumes that $\alpha$ is an algebraic integer.
\end{remark}

\begin{corollary}[Rational linear combinations of conjugate logarithms]
\label{extquad:normalized-sums}
Fix $D\in\Q_{>0}$. For finitely many $u_j,v_j,q_j\in\Q$, suppose
\[
 \alpha_j=u_j+v_j\sqrt D>0,
 \qquad (u_j^2-Dv_j^2)^2=1.
\]
If
\[
 X=\frac1{\sqrt D}\sum_j q_j\log\alpha_j\ne0,
\]
then $\mu(X)=2$ and the all-denominator lower bound
\eqref{extquad:all-denominators} holds for $X$.
\end{corollary}
\begin{proof}
Write $\alpha_j^*=u_j-v_j\sqrt D$ and
$\varepsilon_j=u_j^2-Dv_j^2\in\{1,-1\}$. The elementary identity
$\alpha_j\alpha_j^*=\varepsilon_j$ gives
\[
 (\alpha_j^*)^2=\alpha_j^{-2}.
\]
Squaring is useful here because it treats both norm signs with the same
positive real logarithms. We do not take the real logarithm of
$\alpha_j^*$ when that conjugate is negative.

Choose a positive integer $N$ such that $k_j=Nq_j\in\Z$ for every $j$,
and put
\[
 E=\Q(\sqrt D),\qquad c=2N\sqrt D,\qquad
 B=\prod_j\alpha_j^{2k_j}\in E^\times.
\]
Negative $k_j$ cause no difficulty, since every $\alpha_j$ is positive.
At the specified real embedding, the ordinary real logarithm gives
\[
 \exp(cX)=\exp\left(2\sum_jk_j\log\alpha_j\right)=B.
\]
If $D$ is nonsquare, the other embedding $\tau$ sends $\sqrt D$ to
$-\sqrt D$. Consequently
\[
 \tau(c)=-c,\qquad
 \tau(B)=\prod_j(\alpha_j^*)^{2k_j}
         =\prod_j\alpha_j^{-2k_j}=B^{-1}
         =\exp(\tau(c)X).
\]
Thus the same real target $X$ has the required exponential identity at
both embeddings of $E$. If $D$ is square, $E=\Q$ has only its identity
embedding, and the first identity already verifies every embedding.
Since $c\ne0$ and $X\ne0$, Theorem~\ref{conj:partial} with all embeddings
matched proves the upper bound two. Its irrationality argument and
Dirichlet's theorem give the lower bound two. Equivalently, the
all-embedding endpoint in Section~\ref{extquad:matching} gives
\eqref{extquad:all-denominators} directly.
\end{proof}

\paragraph{Completed Lean theorem.}
This entire assertion, including rational coefficients of either sign,
is formalized in
\path{QuadraticPi/NormalizedQuadraticLogarithmLinearCombination.lean} as
\path{QuadraticPi.normalizedQuadraticLogarithm_rational_sum_eventualLowerBound}
and
\path{QuadraticPi.normalizedQuadraticLogarithm_rational_sum_irrationalityExponent_eq_two}.
The exact exponent declaration was rechecked with its transitive axioms
in the audit of 12 October 2026 described in
Section~\ref{app:formal-status}. Individual terms, or their coefficients,
may vanish; the hypothesis concerns the final sum. The parameter $D$ is
common to all terms. No assertion about unrelated quadratic fields is
needed or used.

\subsection{Rational values of the positive and circular families}
\label{extquad:rational-families}

For $D>0$, define the real branch
\begin{equation}\label{extquad:positive-family}
 f_D(x)=\frac1{\sqrt D}
 \log\frac{1+\sqrt D\,x}{1-\sqrt D\,x},
 \qquad |\sqrt D\,x|<1.
\end{equation}
For $d>0$, the corresponding negative-parameter branch is
\begin{equation}\label{extquad:circular-family}
 f_{-d}(x)=\frac{2}{\sqrt d}\arctan(\sqrt d\,x),
 \qquad x\in\R.
\end{equation}
In the next two theorems the parameters and evaluation points are
rational; the wider real-variable definitions will be used for analysis.

\begin{theorem}[Positive rational parameters]
\label{extquad:positive-rational}
Let $D\in\Q_{>0}$ and $r\in\Q\setminus\{0\}$ with
$|\sqrt D\,r|<1$. Then $\mu(f_D(r))=2$.
More generally, every nonzero finite rational linear combination of
admissible $f_D(r_j)$, at the same $D$, has exponent two.
Both conclusions have the all-denominator lower bound
\eqref{extquad:all-denominators}.
\end{theorem}
\begin{proof}
The branch condition makes $1+\sqrt D\,r$ and $1-\sqrt D\,r$ positive.
Rationalizing the denominator gives
\begin{equation}\label{extquad:rational-coordinates}
 \frac{1+\sqrt D\,r}{1-\sqrt D\,r}
 =u+v\sqrt D,\qquad
 u=\frac{1+Dr^2}{1-Dr^2},\quad v=\frac{2r}{1-Dr^2}.
\end{equation}
Here $1-Dr^2=(1-\sqrt D\,r)(1+\sqrt D\,r)>0$, so all displayed
denominators are nonzero. Direct calculation gives
\[
 u^2-Dv^2
 =\frac{(1+Dr^2)^2-4Dr^2}{(1-Dr^2)^2}=1.
\]
The quotient in \eqref{extquad:rational-coordinates} equals one only
when $2\sqrt D\,r=0$, which would force $r=0$.
Theorem~\ref{extquad:coordinate-theorem} proves the single-value assertion.
For a finite sum, use the same coordinates separately for each $r_j$ and
apply Corollary~\ref{extquad:normalized-sums}. Evaluation points equal to
zero are allowed in the sum, provided that its total value is nonzero.
\end{proof}

\paragraph{Lean correspondence.}
The coordinate and domain facts are \path{denominator_pos},
\path{argument_pos}, \path{argument_eq_coordinates}, \path{coordinate_norm},
and \path{argument_ne_one} in the namespace
\path{QuadraticPi.PositiveQuadraticLogarithm}, module
\path{QuadraticPi/PositiveQuadraticLogarithmCoordinates.lean}.
The module \path{QuadraticPi/PositiveQuadraticLogarithmTheorem.lean}
proves \path{QuadraticPi.positiveQuadraticLogarithm_irrationalityExponent_eq_two}
and
\path{QuadraticPi.positiveQuadraticLogarithm_rational_sum_irrationalityExponent_eq_two},
with corresponding \path{eventualLowerBound} declarations.

\begin{theorem}[Circular rational parameters and an additional normalized constant]
\label{extquad:circular-rational}
Let $d\in\Q_{>0}$. For every nonzero $r\in\Q$,
\[
 \mu\left(\frac{2\arctan(\sqrt d\,r)}{\sqrt d}\right)=2,
 \qquad \mu\left(\frac{\pi}{\sqrt d}\right)=2.
\]
More generally, if $q_0,q_j,r_j\in\Q$ and
\[
 X=q_0\frac{\pi}{\sqrt d}+\sum_jq_j f_{-d}(r_j)\ne0,
\]
then $\mu(X)=2$. Every assertion also has
\eqref{extquad:all-denominators}. There is no size restriction on the
rational $r_j$ in this theorem.
\end{theorem}
\begin{proof}
Work in the actual imaginary quadratic field
$E=\Q(\ii\sqrt d)\subset\C$, and put $\delta=\ii\sqrt d$.
Every embedding sends $\delta$ to $\delta$ or $-\delta$, because
$\delta^2=-d$. For a rational $r$, use the algebraic base
\[
 A_r=\frac{1+r\delta}{1-r\delta}\in E.
\]
The denominator is nonzero: at the distinguished embedding its real
part is one. The elementary identity
\begin{equation}\label{extquad:arctan-exponential}
 \exp(2\ii\arctan t)=\frac{1+\ii t}{1-\ii t}
 \qquad(t\in\R)
\end{equation}
shows that $\exp(\delta f_{-d}(r))=A_r$ at that embedding. Under the
negative sign, the base becomes $A_r^{-1}$ and the exponential also
becomes its reciprocal. Thus every embedding matches with the same
scalar $c=\delta$. If $r\ne0$, then $\arctan(\sqrt d\,r)\ne0$, so
the common-scalar theorem yields exponent two.

For $x_0=\pi/\sqrt d$, the same scalar gives
$\exp(\delta x_0)=\exp(\ii\pi)=-1$. At the other embedding the value
is $\exp(-\ii\pi)=-1$ as well, so the algebraic base is $A_0=-1$.
This proves the assertion for $x_0$, and adjoining this one extra term
to the finite family permits Proposition~\ref{extquad:common-scalar-closure}
to prove the final statement.
\end{proof}

\paragraph{Lean correspondence.}
The actual field, generator, and two-sign embedding statement occur in
\path{QuadraticPi/CircularQuadraticLogField.lean}. The identities above
are \path{exp_generator_value}, \path{all_embeddings_match},
\path{exp_generator_period}, and \path{all_embeddings_match_period}
in the namespace \path{QuadraticPi.CircularQuadraticLogarithm}, module
\path{QuadraticPi/CircularQuadraticLogarithmTheorem.lean}.
The exact-exponent endpoints are
\path{QuadraticPi.circularQuadraticLogarithm_irrationalityExponent_eq_two},
\path{QuadraticPi.rationalNormalizedPi_irrationalityExponent_eq_two}, and
\path{QuadraticPi.circularQuadraticLogarithm_rational_sum_irrationalityExponent_eq_two}.
The source calls $\pi/\sqrt d$ \path{period}; its stated exponential is
$-1$. An additive period of $x\mapsto\exp(\ii\sqrt d\,x)$ is
$2\pi/\sqrt d$, so no factor of two is suppressed in the proof.

\subsection{Real differentiation and the convergent real series}
\label{extquad:real-analysis}

In the remainder of this section, $D$ may be any nonzero real number.
Define $f_D$ by \eqref{extquad:positive-family} for $D>0$ and by
\eqref{extquad:circular-family} with $d=-D$ for $D<0$.
Its stated real domain is
\begin{equation}\label{extquad:real-domain}
 \mathcal I_D=
 \begin{cases}
 \{x\in\R:|\sqrt D\,x|<1\},&D>0,\\
 \R,&D<0.
 \end{cases}
\end{equation}

\begin{proposition}[Derivative and differential identity]
\label{extquad:real-derivative}
For $x\in\mathcal I_D$,
\[
 1-Dx^2>0,\qquad
 f_D'(x)=\frac{2}{1-Dx^2},\qquad
 (1-Dx^2)f_D'(x)=2,
 \qquad f_D(0)=0.
\]
\end{proposition}
\begin{proof}
If $D>0$, put $s=\sqrt D$. The two factors $1\pm sx$ are positive
on $\mathcal I_D$, and hence $(1-sx)(1+sx)=1-Dx^2>0$.
With $Q(x)=(1+sx)/(1-sx)$, the quotient rule gives
\[
 Q'(x)=\frac{2s}{(1-sx)^2},\qquad
 \frac1s\frac{Q'(x)}{Q(x)}=\frac2{1-Dx^2}.
\]
The real logarithm is differentiable at the positive value $Q(x)$,
which proves the desired formula.

For $D=-d<0$, the denominator is $1+dx^2>0$ for every real $x$.
The chain rule for the arctangent gives
\[
 \frac{d}{dx}\left(\frac{2\arctan(\sqrt d\,x)}{\sqrt d}\right)
 =\frac2{1+(\sqrt d\,x)^2}=\frac2{1+dx^2}.
\]
Multiplication by the nonzero denominator proves the differential
identity in either case. Direct substitution gives $f_D(0)=0$.
\end{proof}

\paragraph{Lean correspondence.}
The module \path{QuadraticPi/QuadraticLogarithmicFamilyDerivative.lean}
defines the real functions \path{positiveValue}, \path{circularValue},
\path{value}, and the domain predicate \path{Admissible}, all under
\path{QuadraticPi.QuadraticLogarithmicFamily}. Its endpoints are
\path{value_hasDerivAt}, \path{value_deriv}, \path{value_ode}, and
\path{value_zero}; \path{value_pos_rat} and \path{value_neg_rat}
identify these real functions with the arithmetic rational-input
definitions. The domain predicate excludes $D=0$. No derivative-two
claim at that excluded parameter is part of the theorem.

\begin{theorem}[The real power series]
\label{extquad:real-taylor}
For $D\in\R\setminus\{0\}$ and $|D|x^2<1$,
\begin{equation}\label{extquad:real-series}
 f_D(x)=2\sum_{n=0}^{\infty}\frac{D^n x^{2n+1}}{2n+1},
\end{equation}
and the displayed series is absolutely convergent.
\end{theorem}
\begin{proof}
For $D>0$, set $t=\sqrt D\,x$, so $|t|<1$. The standard convergent
logarithm series gives
\[
 \log(1+t)-\log(1-t)=2\sum_{n\ge0}\frac{t^{2n+1}}{2n+1}.
\]
Both $1+t$ and $1-t$ are positive, so the left side is exactly
$\log((1+t)/(1-t))$. Divide by $\sqrt D$ and use
\[
 (\sqrt D\,x)^{2n+1}=\sqrt D\,D^n x^{2n+1}
\]
to obtain \eqref{extquad:real-series}.
For $D=-d<0$, use instead
$\arctan t=\sum_{n\ge0}(-1)^nt^{2n+1}/(2n+1)$ with
$t=\sqrt d\,x$ and $|t|<1$. Multiplication by $2/\sqrt d$ gives
the coefficient $(-d)^n=D^n$.

For either sign, put $\rho=|D|x^2<1$. The absolute value of the
$n$th summand is at most $2|x|\rho^n$, so a convergent geometric
series proves absolute convergence. The circular function itself is
defined for all real $x$; the series equality asserted here uses the
smaller open interval $|D|x^2<1$.
\end{proof}

\paragraph{Lean correspondence.}
In \path{QuadraticPi/QuadraticLogarithmicFamilyTaylor.lean},
\path{positiveValue_hasSum} uses
\path{Real.hasSum_log_sub_log_of_abs_lt_one}, and
\path{circularValue_hasSum} uses \path{Real.hasSum_arctan}.
The combined declarations in \path{QuadraticPi.QuadraticLogarithmicFamily}
are \path{value_hasSum}, \path{value_eq_tsum}, and
\path{value_series_absolutely_summable}. Their Lean hypothesis
$|Dx^2|<1$ is equivalent to the real condition above.

\subsection{The principal complex branch and its derivative}
\label{extquad:complex-analysis}

Fix $D\in\R\setminus\{0\}$ and prescribe the square root
\[
 \delta_D=\begin{cases}\sqrt D,&D>0,\\ \ii\sqrt{-D},&D<0.
 \end{cases}
 \qquad \delta_D^2=D,\quad \delta_D\ne0.
\]
Let $\operatorname{Log}$ denote the principal complex logarithm and set
\begin{equation}\label{extquad:complex-function}
 F_D(z)=\frac1{\delta_D}\operatorname{Log}
       \frac{1+\delta_Dz}{1-\delta_Dz}.
\end{equation}
The logarithm is holomorphic on the slit plane
$\C\setminus(-\infty,0]$. All complex differentiability assertions
below refer to the displayed principal branch.

\begin{proposition}[Agreement with the real branches]
\label{extquad:branch-agreement}
For every $x\in\mathcal I_D$, the value $F_D(x)$ is real and equals
$f_D(x)$. When $D<0$, this holds for every finite real $x$.
\end{proposition}
\begin{proof}
For $D>0$, the quotient is positive real, so its principal complex
logarithm is the real logarithm. Suppose $D=-d<0$ and put
$t=\sqrt d\,x$. Identity~\eqref{extquad:arctan-exponential} gives
\[
 \frac{1+\ii t}{1-\ii t}=\exp(2\ii\arctan t).
\]
The strict inequalities $-\pi/2<\arctan t<\pi/2$ imply
$-\pi<2\arctan t<\pi$. Therefore the principal logarithm of this
exponential is exactly $2\ii\arctan t$, with no multiple of $2\pi\ii$.
Dividing by $\ii\sqrt d$ yields $f_{-d}(x)$.
These inequalities also show that the quotient lies in the slit plane
for every real $t$.
\end{proof}

\paragraph{Lean correspondence.}
The definitions \path{complexRoot} and \path{complexValue}, and the
lemmas \path{principal_log_cayley}, \path{cayley_mem_slitPlane},
\path{circularValue_complex_branch}, \path{positiveValue_complex_branch},
and \path{value_complex_branch}, are in the namespace
\path{QuadraticPi.QuadraticLogarithmicFamily}, module
\path{QuadraticPi/QuadraticLogarithmicFamilyComplexBranch.lean}.
The key branch step applies \path{Complex.log_exp} with the explicit
arctangent bounds; agreement is not inferred from equality of exponentials
alone.

\begin{theorem}[Principal-branch differentiation on a proved domain]
\label{extquad:complex-derivative}
If
\[
 \frac{1+\delta_Dz}{1-\delta_Dz}\in\C\setminus(-\infty,0],
\]
then $1-Dz^2\ne0$ and
\[
 F_D'(z)=\frac2{1-Dz^2}.
\]
In particular, $F_D$ is holomorphic on the open disc
\begin{equation}\label{extquad:complex-disc}
 \mathcal D_D=\{z\in\C:|\delta_Dz|<1\},
\end{equation}
and there it satisfies $(1-Dz^2)F_D'(z)=2$ and $F_D(0)=0$.
\end{theorem}
\begin{proof}
Write $w=\delta_Dz$. Membership of the quotient in the slit plane
ensures that the quotient is nonzero and that both $1-w$ and $1+w$
are nonzero. Hence
\[
 1-Dz^2=(1-w)(1+w)\ne0.
\]
The complex quotient and chain rules, followed by the derivative of
the principal logarithm on the slit plane, give
\[
 F_D'(z)
 =\frac1{\delta_D}
   \frac{2\delta_D}{(1-w)^2}
   \frac{1-w}{1+w}
 =\frac2{1-Dz^2}.
\]
To check the entire disc, rationalize the complex denominator:
\[
 \operatorname{Re}\frac{1+w}{1-w}
 =\frac{1-|w|^2}{|1-w|^2}>0
 \qquad(|w|<1).
\]
Thus the quotient lies in the right half-plane and in particular in
the slit plane. The set \eqref{extquad:complex-disc} is open and contains
zero. The derivative theorem holds at each of its points, giving
holomorphy and the differential identity. Substitution at zero gives
$\operatorname{Log}1/\delta_D=0$.
\end{proof}

\paragraph{Lean correspondence.}
The source \path{QuadraticPi/QuadraticLogarithmicFamilyHolomorphic.lean}
defines \path{complexDisc} and proves \path{complexDisc_isOpen},
\path{zero_mem_complexDisc}, \path{disc_quotient_re_pos},
\path{disc_quotient_mem_slitPlane}, and \path{complex_denominator_ne_zero}.
The stronger pointwise derivative is
\path{QuadraticPi.QuadraticLogarithmicFamily.complexValue_hasDerivAt_of_slitPlane}.
Its disc corollaries are \path{complexValue_hasDerivAt},
\path{complexValue_differentiableOn}, \path{complexValue_deriv}, and
\path{complexValue_ode}; the basepoint identity is \path{complexValue_zero}.

\subsection{The complex Taylor expansion of the specified branch}
\label{extquad:complex-taylor}

\begin{theorem}[Complex Taylor expansion with absolute convergence]
\label{extquad:complex-series-theorem}
For real $D\ne0$ and $z\in\mathcal D_D$,
\begin{equation}\label{extquad:complex-series}
 F_D(z)=2\sum_{n=0}^{\infty}\frac{D^nz^{2n+1}}{2n+1},
 \qquad
 \sum_{n=0}^{\infty}
 \left|\frac{2D^nz^{2n+1}}{2n+1}\right|<\infty.
\end{equation}
The sum is the same principal-logarithm function as in
\eqref{extquad:complex-function}.
\end{theorem}
\begin{proof}
First let $|w|<1$. Apply the convergent complex arctangent series to
$\zeta=-\ii w$:
\[
 \operatorname{Arctan}\zeta
 =\sum_{n\ge0}\frac{(-1)^n\zeta^{2n+1}}{2n+1}.
\]
With the principal-logarithm convention used in that series theorem,
\[
 \operatorname{Arctan}\zeta
 =-\frac{\ii}{2}\operatorname{Log}
      \frac{1+\ii\zeta}{1-\ii\zeta}.
\]
Consequently $2\ii\operatorname{Arctan}(-\ii w)
=\operatorname{Log}((1+w)/(1-w))$.
For each $n$,
\[
 2\ii(-1)^n(-\ii w)^{2n+1}=2w^{2n+1},
\]
so multiplication of the convergent series by $2\ii$ yields
\[
 \operatorname{Log}\frac{1+w}{1-w}
 =\sum_{n\ge0}\frac{2w^{2n+1}}{2n+1}.
\]
This derivation identifies the actual principal logarithm and does not
use an unrestricted identity $\operatorname{Log}(a/b)
=\operatorname{Log}a-\operatorname{Log}b$.

Now take $w=\delta_Dz$ and divide by $\delta_D\ne0$. The relation
$\delta_D^2=D$ gives
\[
 (\delta_Dz)^{2n+1}=\delta_D D^n z^{2n+1},
\]
which proves the first identity in \eqref{extquad:complex-series}.
Also $\rho=|D||z|^2=|\delta_Dz|^2<1$, and the absolute value of its
$n$th summand is bounded by $2|z|\rho^n$. The geometric series proves
the second assertion.
\end{proof}

\paragraph{Lean correspondence.}
The source \path{QuadraticPi/QuadraticLogarithmicFamilyComplexTaylor.lean}
proves \path{rotated_arctan_odd_pow}, \path{principal_log_cayley_hasSum},
and \path{complexRoot_mul_odd_pow} in the namespace
\path{QuadraticPi.QuadraticLogarithmicFamily}. The final convergence,
sum identity, and absolute convergence statements are, respectively,
\path{complexValue_hasSum}, \path{complexValue_eq_tsum}, and
\path{complexValue_series_absolutely_summable}. The input series theorem
is Mathlib's \path{Complex.hasSum_arctan}.

\begin{remark}[Scope of the analytic identification]
\label{extquad:scope}
The arithmetic exponent-two results above apply to rational parameters
under their stated nonvanishing and branch hypotheses. The analytic
proofs establish the displayed functions, derivatives, and convergent
series for nonzero real parameters. The circular real branch is identified
with the principal complex expression on the entire real axis, whereas
the Taylor statements use the open convergence disc. No assertion about
boundary sums, a classification of maximal complex continuation domains,
or uniqueness for the differential equation is needed for these results
or included in their stated formal endpoints.
\end{remark}
% ===== END verified_quadratic_analysis.tex =====

% ===== BEGIN verified_extensions.tex =====
\section{Verified field reductions and reciprocal logarithms}
\label{ext:verified-field-reductions}

This section supplies a detailed verification companion to
Theorem~\ref{conj:invariant} and Corollaries~\ref{conj:trace}
and~\ref{conj:character}. It records the additional algebraic arguments
that have now been checked in Lean, including actual intermediate fields,
the counts of their embeddings, explicit polynomial identities, and the
choice of the real square-root branch. The invariant-field bound,
reciprocal-trace consequences, and normalized character-sum application
are checked. The converse classification of all matching embeddings is
also checked relative to the explicit standard Gel'fond--Schneider
input; that input is not needed for the direct upper-bound proof.

All filenames in this section are relative to
\nolinkurl{anc/lean/QuadraticPi/} in the accompanying
source distribution. Unless another namespace is specified, a module
\nolinkurl{M.lean} uses the namespace \nolinkurl{QuadraticPi.M}.
The sources and their separate \nolinkurl{#check} and
\nolinkurl{#print axioms} audit files give the exact Lean types.

\subsection{The precise approximation conclusion}
\label{ext:approximation-conclusion}

For clarity, the upper estimates proved in the new modules are first
established in the following stronger, explicit form. For a real number
$x$ and a finite real cutoff $\kappa$, let $\mathcal E(x,\kappa)$ denote
\begin{equation}\label{ext:eventual-definition}
 \forall\nu>\kappa\quad \exists Q\ge2\quad
 \forall p\in\Z\quad\forall q\in\Z_{\ge Q},\qquad
 q^{-\nu}\le\left|x-\frac pq\right|.
\end{equation}
No coprimality condition is imposed on $p,q$. In Lean this predicate is
\nolinkurl{QuadraticPi.EventualLowerBoundAbove}.

The passage from~\eqref{ext:eventual-definition} to an exponent bound
has two logically distinct parts. First $x$ is irrational: if
$x=p_0/q_0$, arbitrarily large multiples $(np_0,nq_0)$ give zero error,
contradicting the positive lower bound. Second, for each $\nu>\kappa$,
every reduced rational approximation with error less than $q^{-\nu}$
has denominator less than the corresponding $Q$. There are only
finitely many such rationals in a bounded neighborhood of $x$.
Consequently the approximation spectrum is bounded above by $\kappa$.
Dirichlet's theorem supplies the lower endpoint $2$. Thus
\begin{equation}\label{ext:eventual-to-exponent}
 \mathcal E(x,\kappa)\quad\Longrightarrow\quad
 x\notin\Q,\qquad 2\le\mu(x)\le\kappa.
\end{equation}
This order of reasoning avoids using the real-valued supremum before
its nonempty, bounded approximation spectrum has been established.

\paragraph{Lean correspondence.}
The definition and the implication are in
\nolinkurl{GeneralExponentConclusion.lean}, in the namespace
\nolinkurl{QuadraticPi}; the relevant declarations
are \nolinkurl{irrational_of_eventualLowerBoundAbove},
\nolinkurl{bddAbove_approximationExponents_of_eventualLowerBoundAbove},
and \nolinkurl{irrationalityExponent_bounds_of_eventualLowerBoundAbove}.
The Diophantine input used below is the previously checked partial
embedding theorem
\nolinkurl{QuadraticPi.partialScalarMatchingLogarithm_eventualLowerBound}
in \nolinkurl{PartialEmbeddingLogarithmTheorem.lean}. The field
reductions do not replace that input by an assumption about a
logarithm's exponent.

\subsection{Invariant coordinates and compatible embedding actions}
\label{ext:invariant-coordinates}

Let $E$ be a field containing $\Q$, let $a,c\in E$ with $c\ne0$, and
let $\iota,\sigma:E\longrightarrow\C$ be field embeddings. Define
the actual intermediate field
\begin{equation}\label{ext:invariant-field}
 K=\Q\bigl(c^2,\ a+a^{-1},\ c(a-a^{-1})\bigr)\subseteq E.
\end{equation}
For applications $a$ is nonzero; the coordinate lemma in Lean is even
formulated without this additional hypothesis, using the total inverse
operation in a field. Throughout the mathematical discussion of
logarithms we retain $a\ne0$.

\begin{proposition}[Exact coordinate criterion]\label{ext:coordinate-criterion}
The embeddings $\sigma$ and $\iota$ agree on $K$ if and only if
one of the following two compatible actions holds:
\begin{align*}
 &\sigma(c)=\iota(c),&\quad \sigma(a)&=\iota(a),\\
 &\sigma(c)=-\iota(c),&\quad \sigma(a)&=\iota(a)^{-1}.
\end{align*}
If $x\in\R$ satisfies
$\exp(\iota(c)x)=\iota(a)$, every such $\sigma$ also satisfies
$\exp(\sigma(c)x)=\sigma(a)$.
\end{proposition}
\begin{proof}
Agreement on $K$ is equivalent to agreement on its three displayed
generators: addition, multiplication and inversion preserve agreement,
and both embeddings agree on $\Q$. Write
$A=\iota(a)$, $B=\sigma(a)$, $C=\iota(c)$ and $C'=\sigma(c)$.
The generator equalities are
\[
 (C')^2=C^2,\qquad B+B^{-1}=A+A^{-1},\qquad
 C'(B-B^{-1})=C(A-A^{-1}).
\]
Since $C\ne0$, the first equality gives $C'=C$ or $C'=-C$.
In the first case cancel $C$ in the third equality and add it to
the second; this gives $2B=2A$. In the second case the same cancellation
gives $B-B^{-1}=-(A-A^{-1})$, so adding the second equality gives
$2B=2A^{-1}$. These are exactly the asserted actions. Conversely,
substitution of either action fixes all three generators.

The first action preserves the exponential equation directly. For the
second action use
\[
 \exp(-\iota(c)x)=\exp(\iota(c)x)^{-1}=\iota(a)^{-1}.
\]
No statement about other matching embeddings is needed in this argument.
\end{proof}

If $E=\Q(a,c)$, then $E=K(c)$, because
\begin{equation}\label{ext:recover-a}
 a=\frac12\left((a+a^{-1})+
                   \frac{c(a-a^{-1})}{c}\right).
\end{equation}
Also $c$ is a root of the monic polynomial $X^2-c^2\in K[X]$.
It follows that $E/K$ is finite and
\begin{equation}\label{ext:index-one-two}
                         [E:K]\in\{1,2\}.
\end{equation}
The finite-dimensionality assertion matters in the implementation:
Lean's finite-rank value is zero for an infinite-dimensional vector
space, so proving only a numerical inequality $\operatorname{finrank}
\le2$ would not suffice. The construction $E=K(c)$ and integrality
of $c$ establish finite-dimensionality and positive rank explicitly.

\paragraph{Lean correspondence.}
In \nolinkurl{ScalarMatchingInvariantField.lean}, the definitions are
\nolinkurl{invariantField} and \nolinkurl{FixesInvariantField}.
The coordinate and action equivalences are
\nolinkurl{fixes_iff_coordinates},
\nolinkurl{coordinates_iff_actions}, and
\nolinkurl{fixes_iff_actions}; the exponential implication is
\nolinkurl{fixing_embedding_matches}. Formula~\eqref{ext:recover-a}
is used in \nolinkurl{adjoin_invariantField_c_eq_top}.
The degree argument is recorded by
\nolinkurl{c_isIntegral_invariantField},
\nolinkurl{finrank_invariantField_le_two},
\nolinkurl{finrank_invariantField_pos},
\nolinkurl{finiteDimensional_invariantField}, and
\nolinkurl{finrank_invariantField_eq_one_or_two}.

\subsection{Counting a proved matching set and cancelling the field degree}
\label{ext:invariant-bound}

\begin{theorem}[Verified invariant-field bound]\label{ext:invariant-bound-theorem}
Let $E$ be a number field, $a,c\in E$, $c\ne0$, and let
$\iota:E\longrightarrow\C$ be an embedding. Suppose $x\in\R\setminus\{0\}$
and $\exp(\iota(c)x)=\iota(a)$. With $K$ defined
by~\eqref{ext:invariant-field}, one has
\[
 \mathcal E(x,2[K:\Q]),\qquad
 x\notin\Q,\qquad 2\le\mu(x)\le2[K:\Q].
\]
In particular, $[K:\Q]=1$ implies $\mu(x)=2$.
The hypothesis $E=\Q(a,c)$ is not required.
\end{theorem}
\begin{proof}
Consider the set of embeddings
\[
 S=\{\sigma:E\longrightarrow\C:
                 \sigma|_K=\iota|_K\}.
\]
There are exactly $[E:K]$ such embeddings. Indeed, use $\iota|_K$
to regard $\C$ as a $K$-algebra. The extensions of this embedding
are precisely the $K$-algebra embeddings of $E$ into $\C$.
The finite extension $E/K$ is separable in characteristic zero, and
the target is algebraically closed, so their number is $[E:K]$.
The set is nonempty because it contains $\iota$.

By Proposition~\ref{ext:coordinate-criterion}, every member of $S$
is a matching embedding. Apply the partial embedding bound to this
specific set. It gives
\[
 \mathcal E\left(x,\frac{2[E:\Q]}{|S|}\right).
\]
The field tower identity now gives
\[
 \frac{2[E:\Q]}{|S|}
    =\frac{2[K:\Q][E:K]}{[E:K]}=2[K:\Q].
\]
The denominator is positive. Apply~\eqref{ext:eventual-to-exponent};
when $[K:\Q]=1$, its two endpoint inequalities are equalities.
\end{proof}

This proof requires a sufficiently large set of matching embeddings,
not a classification of every matching embedding. In the minimal field
$E=\Q(a,c)$ the constructed set has one or two members by
\eqref{ext:index-one-two}. In a larger ambient field it can have more
members, and the same tower cancellation still applies.
The converse assertion in Theorem~\ref{conj:invariant}, that an arbitrary
matching embedding agrees with $\iota$ on $K$, now has a checked adapter
to the explicit Gel'fond--Schneider input. It is not required by the
checked proof of Theorem~\ref{ext:invariant-bound-theorem}.

\paragraph{Lean correspondence.}
\nolinkurl{ScalarMatchingInvariantEmbeddings.lean} constructs the finite
set of actual ring embeddings and proves its exact membership and
cardinality in \nolinkurl{exists_embeddings_over_intermediateField}.
Its minimal-field specializations are
\nolinkurl{exists_fixing_embeddings} and
\nolinkurl{exists_matching_embeddings}.
In \nolinkurl{ScalarMatchingInvariantDegreeTheorem.lean}, the tower
cancellation is \nolinkurl{degree_ratio}; the conclusions are
\nolinkurl{invariantField_eventualLowerBound},
\nolinkurl{invariantField_irrational},
\nolinkurl{invariantField_irrationalityExponent_bounds}, and
\nolinkurl{invariantField_irrationalityExponent_eq_two}.

\subsection{The reciprocal trace and its square-root parameter}
\label{ext:reciprocal-parameter}

For $a>0$, $a\ne1$, put
\[
 T(a)=a+a^{-1},\qquad X(a)=\frac{\log a}{a-a^{-1}}.
\]
The denominator is nonzero: equality $a=a^{-1}$ would give $a^2=1$,
and positivity would force $a=1$. Likewise $\log a\ne0$, so $X(a)\ne0$.

\begin{theorem}[Reciprocal-trace form]\label{ext:reciprocal-trace}
If $a>0$ is algebraic and $a\ne1$, then
\[
 \mathcal E\bigl(X(a),2[\Q(T(a)):\Q]\bigr),\qquad
                 2\le\mu(X(a))\le2[\Q(T(a)):\Q].
\]
These statements do not assume a reciprocal minimal polynomial.
\end{theorem}
\begin{proof}
Take $E=\Q(a)\subseteq\R$ and $c=a-a^{-1}$. At the distinguished
real embedding, $cX(a)=\log a$, so the exponential matching equation
is $e^{\log a}=a$. Moreover
\[
 c^2=T(a)^2-4,\qquad c(a-a^{-1})=c^2.
\]
The invariant field~\eqref{ext:invariant-field} is therefore exactly
$\Q(T(a))$. Apply Theorem~\ref{ext:invariant-bound-theorem}.
The formal proof also transports this field from its representation
inside $\Q(a)$ to the actual subfield of $\R$ generated by $T(a)$,
so the degree in the displayed result is the degree of the stated real
number, rather than an unconnected auxiliary field.
\end{proof}

For a real algebraic parameter $t>2$, define
\[
 s=\sqrt{t^2-4},\qquad A(t)=\frac{t+s}{2},\qquad
 V(t)=\frac{\log A(t)}{s}.
\]
Here $s>0$ and $A(t)>1$. Since
$(t+s)(t-s)=4$, one has the exact identities
\begin{equation}\label{ext:parameter-identities}
 A(t)^{-1}=\frac{t-s}{2},\qquad
 A(t)+A(t)^{-1}=t,\qquad A(t)-A(t)^{-1}=s.
\end{equation}
The number $s$ is algebraic because its square is algebraic; hence
$A(t)$ is algebraic. Applying Theorem~\ref{ext:reciprocal-trace} gives
\begin{equation}\label{ext:parameter-bound}
 \mathcal E(V(t),2[\Q(t):\Q]),\qquad
                    2\le\mu(V(t))\le2[\Q(t):\Q].
\end{equation}
In particular, for every rational $t>2$,
\begin{equation}\label{ext:rational-parameter}
 \boxed{\displaystyle
 \mu\left(\frac{\log((t+\sqrt{t^2-4})/2)}{\sqrt{t^2-4}}\right)=2.}
\end{equation}
This rational specialization needs no nonsquare hypothesis on $t^2-4$.

\paragraph{Lean correspondence.}
\nolinkurl{ReciprocalTraceLogarithmTheorem.lean} defines
\nolinkurl{traceField} and \nolinkurl{target}. Its algebraic identity
is \nolinkurl{invariantField_eq_traceField}, its degree transport is
\nolinkurl{traceField_finrank}, and its exponential equation is
\nolinkurl{distinguished_matching}. The resulting statements are
\nolinkurl{reciprocalTrace_eventualLowerBound},
\nolinkurl{reciprocalTrace_irrational},
\nolinkurl{reciprocalTrace_irrationalityExponent_bounds}, and
\nolinkurl{reciprocalTrace_irrationalityExponent_eq_two}.
\nolinkurl{ReciprocalTraceParameterTheorem.lean} defines
\nolinkurl{radical}, \nolinkurl{argument}, and \nolinkurl{value}.
Its declarations \nolinkurl{argument_inv},
\nolinkurl{argument_trace}, \nolinkurl{argument_difference}, and
\nolinkurl{value_eq_reciprocalTarget} verify
\eqref{ext:parameter-identities} and the equality $V(t)=X(A(t))$.
The endpoint declarations are
\nolinkurl{parameter_eventualLowerBound},
\nolinkurl{parameter_irrationalityExponent_bounds}, and
\nolinkurl{rational_parameter_irrationalityExponent_eq_two};
\nolinkurl{rational_field_finrank} proves the degree-one specialization.

\subsection{Inversion and the unambiguous root above one}
\label{ext:inversion}

Both the target and trace are invariant under inversion:
\[
 X(a^{-1})=
 \frac{-\log a}{-(a-a^{-1})}=X(a),\qquad T(a^{-1})=T(a).
\]
For $a>0$, $a\ne1$, one has $T(a)>2$, since
$a(T(a)-2)=(a-1)^2>0$. Furthermore
\[
 \sqrt{T(a)^2-4}=|a-a^{-1}|.
\]
Thus $A(T(a))=a$ when $a>1$, and $A(T(a))=a^{-1}$ when $0<a<1$.
In either case
\begin{equation}\label{ext:target-parameter}
                              X(a)=V(T(a)).
\end{equation}
This proves the equivalence of the two formulations for all admissible
positive $a$, including the interval $(0,1)$. No branch choice is left
implicit. The chosen representative $b\in\{a,a^{-1}\}$ above one
has the same trace and target; if $a$ is algebraic, so is $b$.

\paragraph{Lean correspondence.}
\nolinkurl{ReciprocalTraceInversion.lean} proves
\nolinkurl{target_inv}, \nolinkurl{trace_inv},
\nolinkurl{trace_radicand}, \nolinkurl{trace_gt_two},
\nolinkurl{radical_trace}, \nolinkurl{argument_trace_of_one_lt},
\nolinkurl{argument_trace_of_lt_one},
\nolinkurl{target_eq_parameter},
\nolinkurl{exists_one_lt_representative}, and
\nolinkurl{exists_algebraic_one_lt_representative}.
The two first identities are stated even for Mathlib's total real
operations, but their logarithmic interpretation here is restricted
to $a>0$, $a\ne1$.

\subsection{Reciprocal minimal polynomials and the full fixed field}
\label{ext:reciprocal-fixed-field}

Suppose additionally that the actual monic minimal polynomial
$f\in\Q[X]$ of $a$ satisfies
\begin{equation}\label{ext:reciprocal-minpoly}
                              f^{\mathrm{rev}}=f.
\end{equation}
The hypothesis concerns the minimal polynomial itself; a reciprocal
polynomial that merely annihilates $a$ does not supply this assumption.

\begin{theorem}[Verified reciprocal field structure]
\label{ext:reciprocal-field-structure}
Let $a>0$ be algebraic with $a\ne1$ and assume
\eqref{ext:reciprocal-minpoly}. In $E=\Q(a)$ there is an involutive
$\Q$-automorphism $\tau$ with $\tau(a)=a^{-1}$. If
$K=\Q(a+a^{-1})$, then
\[
 [E:K]=2,\qquad E^{\langle\tau\rangle}=K,
 \qquad [E:\Q]=2[K:\Q].
\]
Consequently, if $\deg f=2m$, then $[K:\Q]=m$ and
\[
 \mathcal E(X(a),\deg f),\qquad 2\le\mu(X(a))\le\deg f.
\]
\end{theorem}
\begin{proof}
Because $a\ne0$, reversal of a polynomial gives
$f^{\mathrm{rev}}(a^{-1})=a^{-\deg f}f(a)=0$.
By reciprocity, $a^{-1}$ is therefore another root of $f$.
The minimal polynomials of $a$ and $a^{-1}$ agree. Sending the
generator $a$ to $a^{-1}$ gives an isomorphism
$\Q(a)\longrightarrow\Q(a^{-1})$; these are the same subfield,
since each generator is the inverse of the other. This constructs
$\tau$ as an automorphism of the actual simple field. Applying $\tau$
twice fixes $a$, and hence fixes every element of $E$.

The automorphism fixes $a+a^{-1}$, so it fixes $K$. It is nontrivial:
if $a=a^{-1}$, positivity would give $a=1$. On the other hand $a$
satisfies $X^2-(a+a^{-1})X+1\in K[X]$, so $[E:K]\le2$.
Nontriviality excludes degree one, proving $[E:K]=2$.

Let $J=E^{\langle\tau\rangle}$ be the standard fixed field.
Then $K\subseteq J\subseteq E$. It is a proper subfield because
$\tau(a)\ne a$. The tower identity, or equivalently the divisibility
$[E:J]\mid [E:K]=2$, now forces $[E:J]=2$ and $J=K$.
This is a full fixed-field equality, not just the observation that the
trace is fixed. Finally apply the tower identity over $\Q$ and the
bound of Theorem~\ref{ext:reciprocal-trace}.
\end{proof}

\paragraph{Lean correspondence.}
\nolinkurl{ReciprocalPolynomialDegree.lean} constructs
\nolinkurl{inverseAutomorphism} after proving
\nolinkurl{inverse_minpoly_root}, \nolinkurl{inverse_minpoly_eq},
and \nolinkurl{adjoin_inverse_eq}. The general construction is over
an arbitrary base field; the real logarithmic application is over
$\Q$. The generator action and square of the automorphism are
\nolinkurl{inverseAutomorphism_generator} and
\nolinkurl{inverseAutomorphism_involutive}.
The degree conclusions are \nolinkurl{reciprocal_trace_index},
\nolinkurl{reciprocal_field_degree}, and
\nolinkurl{reciprocal_trace_degree}; the Diophantine conclusion is
\nolinkurl{reciprocalPolynomial_irrationalityExponent_bounds},
with its stronger preceding
\nolinkurl{reciprocalPolynomial_eventualLowerBound}.

\nolinkurl{ReciprocalTraceFixedField.lean} uses Mathlib's
\nolinkurl{IntermediateField.fixedField} and the subgroup
\nolinkurl{Subgroup.zpowers} generated by the automorphism.
The main equality is
\nolinkurl{reciprocal_fixedField_eq_traceField};
\nolinkurl{reciprocal_fixed_iff_mem_traceField} gives its elementwise
form, and \nolinkurl{reciprocal_fixedField_map_eq_realTraceField}
identifies its image under the actual inclusion into $\R$ with
$\Q(a+a^{-1})$. The proof reduces to
\nolinkurl{fixedField_eq_traceField_of_index_two}.

\subsection{A completely explicit quartic example}
\label{ext:quartic}

Take
\[
 t=2+\sqrt2,\qquad
 a=\frac{2+\sqrt2+\sqrt{2+4\sqrt2}}2,\qquad
 x=\frac{\log a}{\sqrt{2+4\sqrt2}}.
\]
Since $t^2-4=2+4\sqrt2$, these are precisely $a=A(t)$ and $x=V(t)$.
The trace field is
$\Q(t)=\Q(\sqrt2)$, of degree two. Therefore the checked parameter
theorem gives
\begin{equation}\label{ext:quartic-exponent}
                 \mathcal E(x,4),\qquad 2\le\mu(x)\le4.
\end{equation}
The exact degree of the logarithm's algebraic argument can also be
verified without an unproved irreducibility assumption. From
$a+a^{-1}=2+\sqrt2$, multiply by $a$ to get
\[
                       a^2-2a+1=\sqrt2\,a.
\]
Squaring and collecting coefficients gives
\[
 f(a)=0,\qquad f(X)=X^4-4X^3+4X^2-4X+1.
\]
The translation by one is
\[
                       f(X+1)=X^4-2X^2-4X-2.
\]
This polynomial is Eisenstein at $2$: its leading coefficient is not
divisible by $2$, all other coefficients are divisible by $2$, and
its constant coefficient is not divisible by $4$. Thus it is
irreducible over $\Z$ and over $\Q$. Translation is a polynomial
ring automorphism, so $f$ is irreducible over $\Q$ as well. Since it
is monic and annihilates $a$, it is the minimal polynomial of $a$.
Consequently
\[
                 [\Q(a):\Q]=4,\qquad [\Q(t):\Q]=2.
\]
This degree-four computation is additional algebraic information;
it does not sharpen~\eqref{ext:quartic-exponent} to $\mu(x)=2$.
The exact exponent of this example remains undetermined by the results
proved here.

\paragraph{Lean correspondence.}
\nolinkurl{ReciprocalQuarticExample.lean} proves
\nolinkurl{trace_field_eq}, \nolinkurl{trace_field_finrank},
\nolinkurl{alpha_trace}, \nolinkurl{alpha_quartic_eq_zero},
\nolinkurl{definingPolynomial_aeval},
\nolinkurl{quartic_eventualLowerBound}, and
\nolinkurl{quartic_irrationalityExponent_bounds}.
The irreducibility extension
\nolinkurl{ReciprocalQuarticIrreducibility.lean} uses the same namespace,
\nolinkurl{QuadraticPi.ReciprocalQuarticExample}.
Its declarations \nolinkurl{definingPolynomial_taylor_one},
\nolinkurl{shiftedPolynomial_isEisensteinAt},
\nolinkurl{shiftedPolynomial_map_irreducible}, and
\nolinkurl{definingPolynomial_irreducible} check the translated
Eisenstein argument. The exact minimal polynomial and degree are
\nolinkurl{alpha_minpoly}, \nolinkurl{alpha_minpoly_natDegree},
and \nolinkurl{alpha_field_finrank}.

\subsection{Quadratic-character logarithmic combinations: complete descent}
\label{ext:character-product}

We now give the complete argument for Corollary~\ref{conj:character},
including the quadratic field, the actual field norm, the real logarithmic
identity, and the exponent-two conclusion. All these steps have formal
counterparts in the saved Lean sources. The earlier description of this
family as only a completed product construction is superseded by the
degree, norm, logarithm, exponent, stabilizer and input-field increments
of 11 October 2026. The normalization and the nonzero-sum condition are
essential parts of the statement.

\paragraph{The actual product and its action.}
Let $E/F$ be a field extension whose full automorphism group
$G=\operatorname{Aut}_F(E)$ is finite. Let
$\chi:G\to\{1,-1\}$ be a group homomorphism. For $a\in E^\times$ set
\[
 \beta_\chi(a)=\prod_{g\in G}g(a)^{\chi(g)}.
\]
The product is over the full automorphism group; when $a$ does not
generate $E$, repeated conjugate values retain their multiplicities.

\begin{proposition}[Signed product and kernel descent]
\label{ext:character-product-proposition}
The element $\beta_\chi(a)$ is nonzero and, for every $h\in G$,
\begin{equation}\label{ext:character-equivariance}
 h(\beta_\chi(a))=\beta_\chi(a)^{\chi(h)}.
\end{equation}
Consequently $\beta_\chi(a)\in E^{\ker\chi}$; every automorphism
outside the kernel sends it to its reciprocal.
\end{proposition}
\begin{proof}
Every $g(a)$ is nonzero, so every signed power and their product is
nonzero. Applying $h$ to the product, and setting $k=hg$, gives
\[
 h(\beta_\chi(a))
 =\prod_{g\in G}(hg)(a)^{\chi(g)}
 =\prod_{k\in G}k(a)^{\chi(h^{-1}k)}.
\]
Since $\chi(h)^{-1}=\chi(h)$, the last exponent is
$\chi(h)\chi(k)$. Integer-power laws in $E^\times$ therefore give
\eqref{ext:character-equivariance}. For $h\in\ker\chi$ the value
is fixed, which is exactly membership in the kernel fixed field.
\end{proof}

\paragraph{Degree and norm.}
Suppose now that $E/F$ is finite Galois and that $\chi$ is nontrivial.
It takes both values $1,-1$, so
$[G:\ker\chi]=2$. The finite Galois correspondence gives
\begin{equation}\label{ext:character-degree}
 K:=E^{\ker\chi},\qquad [K:F]=2.
\end{equation}
The kernel is normal, so $K/F$ is Galois. Choose $h\notin\ker\chi$.
Its restriction $\iota$ to $K$ is the nonidentity automorphism of
$K/F$: otherwise $h$ would belong to the subgroup fixing $K$, which
is precisely $\ker\chi$. Equation~\eqref{ext:character-equivariance}
gives $\iota(\beta_\chi)=\beta_\chi^{-1}$. Thus the genuine
field norm, defined as the determinant of multiplication, satisfies
\begin{equation}\label{ext:character-norm}
 N_{K/F}(\beta_\chi)
   =\beta_\chi\,\iota(\beta_\chi)=1.
\end{equation}
Here we used the standard Galois expression of the norm as the product
over all embeddings of $K/F$. There are exactly two such embeddings.
The norm is taken in $K/F$, not in the larger ambient extension $E/F$.

\begin{lemma}[Compatible real quadratic realization]
\label{ext:character-realization}
Let $K/\Q$ have degree two and let $\psi:K\hookrightarrow\R$ be
a field embedding. There is a positive nonsquare rational number $D$
and a $\Q$-isomorphism
\[
 e:K\longrightarrow\Q(\sqrt D)\subset\R
 \quad\hbox{such that}\quad e(x)=\psi(x)\quad(x\in K).
\]
These choices depend on $K,\psi$, and can be made before any element
whose logarithm is to be considered.
\end{lemma}
\begin{proof}
Choose $r\in K\setminus\Q$ and write its monic minimal polynomial
as $X^2+bX+c$, with $b,c\in\Q$. The element $s=r+b/2$ generates
$K$ and satisfies
\[
 s^2=D:=b^2/4-c\in\Q.
\]
It is nonzero, since otherwise $r$ would be rational. Injectivity of
$\psi$ gives $D=\psi(s)^2>0$. If $D=q^2$ for $q\in\Q$, then
$(s-q)(s+q)=0$ would make $s$ rational, a contradiction. Thus $D$
is nonsquare. The real identity $\psi(s)^2=D$ gives
$\psi(s)=\sqrt D$ or $-\sqrt D$, and either root generates the
same subfield of $\R$. It follows that $\psi(K)=\Q(\sqrt D)$.
Regard $\psi$ as an isomorphism onto this image to obtain $e$.
The Lean proof obtains $s^2=D$ using the quadratic instance of
Mathlib's cyclic/Kummer extension theorem and proves the same positivity,
nonsquareness and compatibility statements.
\end{proof}

\begin{theorem}[Exponent two for the character logarithmic sum]
\label{ext:character-exponent}
Let $E/\Q$ be finite Galois, let
$\chi:\operatorname{Aut}_{\Q}(E)\to\{1,-1\}$ be nontrivial,
and fix a field embedding $\phi:E\hookrightarrow\R$.
There exists a positive nonsquare $D\in\Q$, chosen before $a$, such
that every $a\in E$ satisfying
\begin{equation}\label{ext:character-admissible}
 \phi(g(a))>0\quad(g\in G),\qquad \beta_\chi(a)\ne1
\end{equation}
obeys
\begin{equation}\label{ext:character-mu}
 \mu\left(\frac{1}{\sqrt D}
       \sum_{g\in G}\chi(g)\log\phi(g(a))\right)=2.
\end{equation}
One may instead specify any positive nonsquare rational $D$ for which
\[
 \phi(E^{\ker\chi})=\Q(\sqrt D)
 \quad\hbox{as subfields of }\R.
\]
No canonical, integral or squarefree choice of $D$ is asserted.
\end{theorem}
\begin{proof}
First form $K=E^{\ker\chi}$, whose degree is two by
\eqref{ext:character-degree}. Apply
Lemma~\ref{ext:character-realization} to the restriction of $\phi$
to $K$. This chooses $D$ and the compatible map $e$ once for the
field and the character, before choosing $a$.

Now take an admissible $a$. The identity automorphism in
\eqref{ext:character-admissible} shows $a\ne0$. Let
$b=\beta_\chi(a)\in K$. All factors of $\phi(b)$ are positive,
so $\phi(b)>0$. The ordinary real logarithm is additive on finite
products of positive numbers and satisfies $\log(z^{-1})=-\log z$
for $z>0$. Therefore
\begin{equation}\label{ext:character-log-identity}
 \log\phi(b)=\sum_{g\in G}\chi(g)\log\phi(g(a)).
\end{equation}
The hypothesis $b\ne1$ and injectivity of $\phi$ imply
$\phi(b)\ne1$; its logarithm is consequently nonzero.

The element $e(b)$ of the actual real field $\Q(\sqrt D)$ has
real value $\phi(b)>0$, is different from one, and has norm one.
The last statement follows from \eqref{ext:character-norm} and
invariance of the field norm under $\Q$-algebra isomorphisms.
The real quadratic norm logarithm theorem
(Corollary~\ref{conj:real-quadratic}) applies to $e(b)$ and gives
\[
 \mu\bigl(\log(e(b))/\sqrt D\bigr)=2.
\]
Compatibility of $e$ with $\phi$ and
\eqref{ext:character-log-identity} identify its argument with that
in \eqref{ext:character-mu}. For a specified $D$ with the asserted
field-image equality, $\phi|_K$ itself supplies the compatible
isomorphism onto $\Q(\sqrt D)$; the remaining steps are identical.
\end{proof}

This theorem concerns a \emph{normalized} signed sum. It does not say
that every unnormalized combination of conjugate logarithms has exponent
two, nor that arbitrary rational coefficients may replace the character
without another argument. The assumptions that all the evaluated
conjugates are positive and that the actual signed product is not one
are explicit. Under positivity the latter is equivalent to nonvanishing
of the real logarithmic sum, by \eqref{ext:character-log-identity}.

\begin{proposition}[Cancellation and the necessary input field]
\label{ext:character-input-field}
For finite Galois $E/F$ and $a\ne0$, a nontrivial product
$\beta_\chi(a)\ne1$ implies
\[
 E^{\ker\chi}\subseteq F(a).
\]
If $\chi$ is nontrivial, it follows that $2\mid[F(a):F]$.
Thus every nonzero input of odd degree has $\beta_\chi(a)=1$.
With a real embedding and positive evaluated conjugates, its character
logarithmic sum is zero.
\end{proposition}
\begin{proof}
We first establish cancellation in the more general setting of a finite
automorphism group, without assuming finite degree or Galoisness.
Suppose $h(a)=a$ and $\chi(h)=-1$. Write
\[
 P_+=\prod_{\chi(g)=1}g(a),\qquad
 P_-=\prod_{\chi(g)=-1}g(a),\qquad
 \beta_\chi(a)=P_+/P_-.
\]
Right multiplication $g\mapsto gh$ bijects the two sign classes,
while $(gh)(a)=g(h(a))=g(a)$. Hence $P_+=P_-$. All factors are
nonzero, so $P_-\ne0$ and $\beta_\chi(a)=1$. This argument uses
right multiplication and does not assume that $G$ is commutative.

By contraposition, $\beta_\chi(a)\ne1$ forces every automorphism
fixing $a$ into $\ker\chi$. For finite Galois $E/F$ this says
$\operatorname{Gal}(E/F(a))\subseteq\ker\chi$. Fixed fields
reverse inclusions, and the Galois correspondence now gives
$E^{\ker\chi}\subseteq F(a)$. If $\chi$ is nontrivial, the first
field has degree two over $F$, so the tower law makes $[F(a):F]$
even. Odd degree therefore forces product one. Under positivity,
\eqref{ext:character-log-identity} then gives the zero sum.
\end{proof}
Containment of the quadratic field, or even input degree, is necessary
for a nontrivial product; neither condition is asserted sufficient.

\paragraph{Exact Lean correspondence and recorded verification.}
All the following filenames are in \nolinkurl{QuadraticPi/}.
The product is a definition, \nolinkurl{QuadraticCharacterProduct.product},
and its descended value is \nolinkurl{fixedFieldProduct}. The proofs
\nolinkurl{product_equivariant}, \nolinkurl{product_mem_fixedField},
\nolinkurl{product_ne_zero} and \nolinkurl{fixedFieldProduct_ne_zero}
are in \nolinkurl{QuadraticCharacterProduct.lean}.

The field degree is
\nolinkurl{QuadraticPi.QuadraticCharacterDegree.character_fixedField_finrank};
the actual norm is
\nolinkurl{QuadraticPi.QuadraticCharacterNorm.norm_fixedFieldProduct}.
The positivity, logarithmic identity and nonzero sum are respectively
\nolinkurl{product_pos}, \nolinkurl{log_product} and
\nolinkurl{logSum_ne_zero} in namespace
\nolinkurl{QuadraticPi.QuadraticCharacterLogarithm}.
The compatible realization is
\nolinkurl{QuadraticPi.RealQuadraticClassification.exists_positive_nonsquare_realization}.

The complete endpoint in \nolinkurl{QuadraticCharacterExponentTheorem.lean}
is
\nolinkurl{QuadraticPi.QuadraticCharacterExponent.exists_logSum_irrationalityExponent_eq_two}.
Its type has $\exists D\,\exists e$ \emph{before} $\forall a$.
The variants
\nolinkurl{logSum_irrationalityExponent_eq_two_of_fieldRange} and
\nolinkurl{logSum_irrationalityExponent_eq_two_of_equiv} in that namespace
accept, respectively, the specified real-field equality and an explicitly
compatible isomorphism. The existence endpoint proves that such a
realization exists, rather than leaving it as a premise. It applies
\nolinkurl{QuadraticPi.realQuadraticLogarithm_norm_irrationalityExponent_eq_two}.

The cancellation module \nolinkurl{QuadraticCharacterStabilizer.lean}
proves \nolinkurl{product_eq_one_of_fixed_negative},
\nolinkurl{stabilizer_le_kernel_of_product_ne_one} and
\nolinkurl{logSum_eq_zero_of_fixed_negative}. The later
\nolinkurl{QuadraticCharacterInputField.lean} proves, in its matching
namespace, \nolinkurl{fixedField_le_adjoin_of_product_ne_one},
\nolinkurl{two_dvd_finrank_adjoin_of_product_ne_one},
\nolinkurl{product_eq_one_of_odd_finrank_adjoin} and
\nolinkurl{logSum_eq_zero_of_odd_finrank_adjoin}.

The saved exponent build completed successfully with 5586 Lake jobs;
the input-field build completed with 2411 jobs. Their saved transitive
axiom reports for the four realization/exponent declarations and four
input-field declarations list only \nolinkurl{propext},
\nolinkurl{Classical.choice} and \nolinkurl{Quot.sound}. They list
no \nolinkurl{sorryAx} or additional mathematical axiom. These are
historical build and audit results. The fresh selected-endpoint audit
for the present revision is recorded in Appendix~\ref{app:formal-status}.
Completion of this particular family does
not imply completion of every theorem in the consolidated manuscript.

% ===== END verified_extensions.tex =====

% ===== BEGIN graded_logarithmic_measures.tex =====
% Consolidated and independently audited, 11 October 2026.
% This file has no dependency on the stopped investigation of powers of pi.
\section{Graded approximation and a coefficient-uniform principle}
\label{graded:framework}

The preceding fixed-value theorems and the results in this section have
different quantifiers. Here the evaluation point is required to be
sufficiently small. Once that point has been fixed, the resulting
inequalities hold simultaneously for a varying target polynomial and a
varying rational approximant. We give the construction in detail because
this uniformity cannot be deduced from the irrationality exponent of each
individual target.

The rank input is the general differential-system part of
Fischler--Rivoal's work \cite{FR2025}, not their irrationality theorem for
values of $E$-functions. Logarithms are not $E$-functions. We use
Theorem~3, Lemma~1, the lift of Section~4.5, and Theorem~4 of that paper.
In particular the required independence is over
$\overline{\mathbb Q}$, not merely over $\mathbb Q$.

\subsection{Adjacent grades and the rank theorem}

Fix a positive integer $m$ and analytic functions $f_1,\ldots,f_m$
satisfying a rational inhomogeneous differential system
\begin{equation}\label{graded:inhomogeneous}
 f'=s_0+Sf,\qquad s_0\in\mathbb Q(z)^m,
 \quad S\in M_m(\mathbb Q(z)).
\end{equation}
For an integer $G\ge1$, write
\[
 \Theta=\{\kappa\in\mathbb Z_{\ge0}^m:|\kappa|=G\},\qquad
 \Omega=\Theta\cup\{\lambda:|\lambda|=G-1\},
\]
\[
 \theta=\binom{G+m-1}{m-1},\qquad
 \omega=\theta+\binom{G+m-2}{m-1},\qquad
 \frac{\omega}{\theta}=1+\frac{G}{G+m-1}.
\]
The letters $\Omega,\Theta$ in this section denote finite index sets,
not periods. A polynomial indexed by a vector with a negative coordinate
is, by convention, zero.

The lift can be defined without a lengthy matrix formula. Solve the
signed dual system
\[
 a'=-S^{\mathsf T}a,\qquad c'=s_0^{\mathsf T}a.
\]
Differentiation preserves the list of coordinates
\begin{equation}\label{graded:lift-coordinates}
 a^\kappa\quad(|\kappa|=G),\qquad
 ca^\lambda\quad(|\lambda|=G-1).
\end{equation}
Consequently this list satisfies $Y'=\mathcal A Y$ for an
$\omega\times\omega$ rational matrix $\mathcal A$.
For a polynomial column $P=(P_\lambda)_{\lambda\in\Omega}$ put
\[
 \nabla P=P'+\mathcal A^{\mathsf T}P.
\]
If $Y'=\mathcal A Y$, the product rule gives
\begin{equation}\label{graded:pairing}
 \frac{d^k}{dz^k}(P^{\mathsf T}Y)
                 =(\nabla^kP)^{\mathsf T}Y.
\end{equation}
This identity explains why covariant derivatives of the polynomial
column produce derivatives of the same small analytic remainder.

\paragraph{Continuation convention for the rank statement.}
The value at the ordinary point is obtained by one coherent analytic
continuation of the original solution germ along a path avoiding the
system's poles. The lifted Nilsson solutions and contact estimates refer
to that same continuation. This is the branch convention used by the
checked rank endpoint; unrelated values on different branches are not
combined.

\begin{lemma}[Short derivatives give full rank]\label{graded:rank}
Let $r\in\mathbb Q\setminus\{0\}$ be an ordinary point of the
system, and suppose that $1,f_1(r),\ldots,f_m(r)$ are linearly
independent over $\overline{\mathbb Q}$. Suppose also that the
solutions of the lifted system used below belong to the Nilsson class
at zero, and that multiplication by their coefficients lowers the
order of a holomorphic function by at most a fixed constant.
Choose
\[
 0<\eta\le\frac1{\omega+1},\qquad
 \sigma=\frac{\omega-\eta}{\theta}>1.
\]
There is a constant $C_r$ such that the following holds for all
sufficiently large integers $M$. If $P_\lambda\in\mathbb Q[z]$,
not all zero, have degrees less than $M$, and
\begin{equation}\label{graded:contact}
 R_\kappa=P_\kappa+\sum_{j=1}^mP_{\kappa-e_j}f_j,
 \qquad \operatorname{ord}_0R_\kappa\ge
 K:=\lfloor\sigma M\rfloor\quad(\kappa\in\Theta),
\end{equation}
then the columns
\[
 \frac{\nabla^kP(r)}{k!},\qquad
 0\le k\le\lfloor\eta M\rfloor+C_r,
\]
span $\mathbb Q^\Omega$.
\end{lemma}

Here a Nilsson function is a finite sum of terms
$z^\alpha(\log z)^j$ multiplied by holomorphic functions, with the
order interpreted as the smallest relevant real part of $\alpha$.
Ordinary points at zero cause no loss of order. The explicit transport
used later has only powers of $\sqrt z$ and is also covered.

\begin{proof}
Normalize the homogeneous dual solution by $a(r)=\rho$, where
$\rho_1,\ldots,\rho_m$ are independent formal parameters, and set
$c(z)=\sum_ja_j(z)f_j(z)$. Equation
\eqref{graded:inhomogeneous} gives $c'=s_0^{\mathsf T}a$.
Taking the coefficient of $\rho^\kappa$, $\kappa\in\Theta$, in
\eqref{graded:lift-coordinates} gives $\theta$ solutions
$Y_\kappa$ of the lifted system. Their endpoint values are
\[
 Z_\kappa=E_\kappa+
              \sum_jf_j(r)E_{\kappa-e_j}.
\]
They are independent, since their upper-grade coordinates form the
identity matrix. Their solution span $\mathcal F$ therefore has
dimension $\theta$. Pairing any member of $\mathcal F$ with $P$
expresses its remainder as a fixed linear combination of the
$R_\kappa$, with coefficients coming from the dual transport. Hence
all these remainders have order at least $K-w$ for a fixed $w\ge0$.

Theorem~4 of \cite{FR2025}, applied to the verified
$\overline{\mathbb Q}$-independence of the endpoint values, says
that every proper nonzero subspace $R$ defined over
$\overline{\mathbb Q}$ satisfies
\[
 \dim R>\frac{\omega}{\theta}
                \dim\bigl(R\cap\operatorname{span}\{Z_\kappa\}\bigr).
\]
Lemma~1 of that source says that the evaluation image at $r$ of
the solutions whose remainder vanishes identically is defined over
$\mathbb Q$. The strict inequality consequently proves the
nondegeneracy required by its Theorem~3.

Apply that theorem with polynomial degree bound $M-1$, dimension
$\omega$, parameter $\eta$, and the fixed loss $w$. Because
$\sigma>1$, one has
$\lfloor\sigma M\rfloor\ge\sigma(M-1)$, so its vanishing-order
hypothesis holds. Its derivative budget is at most
$\eta(M-1)+c_1+(w+1)\omega$, up to the source's indexing
convention. Absorb the fixed terms into $C_r$. Division by $k!$
does not alter rank. All the evaluated columns are rational, so
their rational and complex ranks agree.
\end{proof}

The field in this lemma is essential. Rational independence of a
collection of algebraic numbers does not meet the hypothesis. For
powers of one nonzero logarithm of an algebraic number,
Hermite--Lindemann gives transcendence, hence independence of its
distinct powers over $\overline{\mathbb Q}$. For the two logarithms
treated below we use Baker's theorem. Neither argument proves
independence of arbitrary products of several logarithms.

\subsection{The elementary passage from forms to a measure}

\begin{lemma}[Two independent forms]\label{graded:two-form}
Let $x\in\mathbb R$. Suppose that for every sufficiently large
integer $M$ there are two integer pairs $(A_{M,s},B_{M,s})$,
$s=1,2$, with nonzero determinant, such that, for fixed $a,c>0$,
\[
 |B_{M,s}|\le e^{aM+o(M)},\qquad
 |A_{M,s}+B_{M,s}x|\le e^{-cM+o(M)}.
\]
Then $x$ is irrational and $\mu(x)\le1+a/c$.
\end{lemma}
\begin{proof}
For any integers $p,q$ with $q\ge1$, the two integers
$qA_{M,s}+pB_{M,s}$ cannot both vanish. Fix a small
$\tau\in(0,c)$ and choose
$M=\lceil\log(2q)/(c-\tau)\rceil$, increasing a fixed lower
cutoff when necessary. For sufficiently large $q$ the small-form
bound and one of these nonzero integers give
\[
 1\le\frac12+q e^{(a+\tau)M}|x-p/q|.
\]
This proves the asserted exponent on letting $\tau\downarrow0$.
If $x$ were rational, use its fixed numerator and denominator and
let $M$ tend to infinity in the same nonzero-integer argument.
\end{proof}

The next formulation retains the height of a varying target. It will
be used both for polynomial values and for linear combinations of two
logarithms.

\begin{proposition}[Uniform coefficient contraction]
\label{graded:contraction}
Fix an endpoint vector $\xi=(\xi_1,\ldots,\xi_m)\in\mathbb R^m$.
Suppose that for each sufficiently large $M$ a family of integer
columns $v_{M,k}\in\mathbb Z^\Omega$ spans the whole coordinate
space. For $c=(c_1,\ldots,c_m)\in\mathbb Z^m\setminus\{0\}$,
put $H=\max(1,|c_1|,\ldots,|c_m|)$ and
\[
 A_{M,k}(c)=\sum_{|\kappa|=G}c^\kappa(v_{M,k})_\kappa,
 \qquad
 B_{M,k}(c)=\sum_{|\lambda|=G-1}c^\lambda(v_{M,k})_\lambda.
\]
Assume fixed rates $a,c_*>0$ give, uniformly in the coefficient
vector and the indicated jets,
\[
 |B_{M,k}(c)|\le C H^{G-1}e^{aM+o(M)},\qquad
 |A_{M,k}(c)+B_{M,k}(c)\,c\cdot\xi|
                   \le C H^Ge^{-c_*M+o(M)}.
\]
Set $\mu_*=1+a/c_*$ and $\nu_*=G\mu_*-1$. For every
$\delta>0$ there is $C_\delta>0$, independent of $c,p,Q$, such
that
\begin{equation}\label{graded:contraction-bound}
 \left|c\cdot\xi-\frac pQ\right|
 \ge C_\delta H^{-\nu_*-\delta}Q^{-\mu_*-\delta}
 \quad(p\in\mathbb Z, Q\ge1).
\end{equation}
Equivalently one may use any positive linked slack $\zeta$ and
the exponents $\nu_*+G\zeta$ and $\mu_*+\zeta$.
\end{proposition}
\begin{proof}
The functionals defining $A$ and $B$ are nonzero and have disjoint
supports. For example $c_j\ne0$ gives a nonzero upper coefficient
$c_j^G$ and lower coefficient $c_j^{G-1}$; when $G=1$ the latter
is the constant monomial one. Full rank therefore gives two
independent integer pairs $(A,B)$, for every $c\ne0$, at every
$M$ above the same fixed cutoff.

Fix $\tau\in(0,c_*)$. Enlarge a constant $C\ge1$ and a cutoff
$M_0$ so that the two displayed estimates hold with rates
$a+\tau$ and $c_*-\tau$ for all $M\ge M_0$. Choose
\[
 M=M_0+\left\lceil
 \frac{\log(2C)+\log Q+G\log H}{c_*-\tau}\right\rceil.
\]
One of $QA+pB$ is a nonzero integer. For that pair,
\[
 1\le Q|A+B\,c\cdot\xi|
       +Q|B|\,|c\cdot\xi-p/Q|
 \le\frac12+Q|B|\,|c\cdot\xi-p/Q|.
\]
In particular its $B$ cannot be zero. Substitution of the upper
bound for $M$ gives a constant independent of $c,p,Q$ and
exponents
\[
 G-1+\frac{G(a+\tau)}{c_*-\tau},
 \qquad 1+\frac{a+\tau}{c_*-\tau}.
\]
Choose $\tau$ sufficiently small for any prescribed linked slack;
then choose $\zeta\le\delta/G$ to obtain
\eqref{graded:contraction-bound}.
\end{proof}

\section{Powers of a logarithm at a small rational point}
\label{graded:small-log}

\begin{theorem}\label{graded:small-log-theorem}
For each fixed integer $n\ge1$ and every $\varepsilon>0$, there
is $b_0(n,\varepsilon)$ such that every integer $b\ge b_0$ satisfies
\[
 \mu\bigl(\log^n(1+1/b)\bigr)\le2+\varepsilon,
 \qquad
 \mu\bigl(\arctan^n(1/b)\bigr)\le2+\varepsilon.
\]
The logarithm and arctangent have their real branches. This
recovers the small-argument phenomenon stated by Chudnovsky in
\cite{Chudnovsky1983}; no historical novelty is claimed for that conclusion.
\end{theorem}

We give quantitative rates, so that the order of all choices is visible.
Let $h(z)=\log(1+z)$ or $h(z)=\arctan z$, and apply the adjacent
grades to $f_j=h^j$, $1\le j\le n$. Fix $G\ge1$, choose
$0<\eta\le1/(\omega+1)$, and put
\[
 \sigma=\frac{\omega-\eta}{\theta},\qquad
 \ell=G(n-1)+1,\qquad H_s=\sum_{j=1}^s\frac1j.
\]
We will choose $\sigma>1+3\eta$. Define
\[
 \beta=\frac{\omega-\eta}{\eta}\sigma(H_n+\log2),\qquad
 B_b=\begin{cases}b+1,&h(z)=\log(1+z),\\b^2+1,&h(z)=\arctan z.
 \end{cases}
\]
For $b\ge4$, set
\begin{align}
 a_b&=\beta+\log b+\eta\log B_b+\eta(H_\ell+\log4),
 \label{graded:small-height}\\
 c_b&=(\sigma-1-\eta)\log b-\eta\log B_b
          -\beta-\sigma\log3-\eta(H_\ell+\log2).
 \label{graded:small-decay}
\end{align}
We prove the more precise assertion
\begin{equation}\label{graded:small-rate-result}
 c_b>0\quad\Longrightarrow\quad
 \mu(h(1/b)^n)\le1+a_b/c_b.
\end{equation}

\subsection{Integral contact polynomials}

For $s,v\ge0$ define
\[
 D_s(v)=\prod_{j=1}^s\operatorname{lcm}
                   (1,\ldots,\lfloor v/j\rfloor),
\]
with an empty least common multiple equal to one. A useful elementary
observation is that $D_s(v)$ clears the coefficient of degree at most
$v$ in a product of at most $s$ logarithmic series whose coefficient
of degree $j$ has denominator dividing $j$. Indeed, arrange the
positive indices of a product in decreasing order. Its $j$th largest
index is at most $\lfloor v/j\rfloor$, so the product of its
denominators divides $D_s(v)$. The prime number theorem in the form
$\log\operatorname{lcm}(1,\ldots,v)=v+o(v)$ gives
\begin{equation}\label{graded:harmonic-clearing}
 \log D_s(v)=H_s v+o(v)\qquad(s\text{ fixed}).
\end{equation}

For every sufficiently large $M$, impose
\eqref{graded:contact} with $K=\lfloor\sigma M\rfloor$ on
$\omega M$ integer coefficients. There are at most
$\theta K\le(\omega-\eta)M$ equations. On $|z|=1/2$ both
choices of $h$ have modulus less than one; their fixed powers
therefore have coefficients of degree at most $K$ bounded by $2^K$.
Multiplication by $D_n(K)$ makes all equations integral. The elementary
box form of Siegel's lemma gives a nonzero solution with
\begin{equation}\label{graded:contact-height}
 \log H(P)\le\beta M+o(M).
\end{equation}
More explicitly, a system with $U$ unknowns, at most $E<U$ equations,
and integer entries of modulus at most $L$ has a nonzero integer
kernel vector of logarithmic height at most
$E(\log L+O(\log U))/(U-E)+O(1)$. This follows by mapping the
integer box $\{0,\ldots,T\}^U$ into its bounded box of possible
equation values and choosing $T$ so that two vectors have the same
image. Here $U=\omega M$, $E\le(\omega-\eta)M$, and
$\log L\le(H_n+\log2)\sigma M+o(M)$.

\subsection{Integral transport and divided derivatives}

The lifted connection is $\mathcal A=h'(z)A_0$ for an integral
nilpotent matrix $A_0$. Its transport has integral divided powers:
\begin{equation}\label{graded:integral-exponential}
 A_0^{\ell+1}=0,\qquad A_0^j/j!\in M_\omega(\mathbb Z).
\end{equation}
To verify this directly, use the increment $u=h(z)-h(r)$. The
signed dual system is transported by
\[
 a_j(u)=\sum_{k=j}^n(-1)^{k-j}\binom{k}{j}a_k(0)u^{k-j},
 \qquad
 c(u)=c(0)+\sum_{k=1}^n(-1)^{k-1}a_k(0)u^k.
\]
The induced transport on $a^\kappa$ and $ca^\lambda$ is an
integer polynomial matrix in $u$. Its degree is at most
$G(n-1)+1=\ell$, and its value at $u=0$ is the identity.
It is thus exactly $\exp(A_0u)$, proving
\eqref{graded:integral-exponential}.

At $r=1/b$, set
$v_k=\nabla^kP(r)/k!$ and $u_r(t)=h(r+t)-h(r)$. The Taylor
identity corresponding to \eqref{graded:pairing} is
\[
 v_k=[t^k]\exp(A_0^{\mathsf T}u_r(t))P(r+t).
\]
For every positive $j$, the number
$j B_b^j[t^j]u_r(t)$ is integral. In the logarithmic case this
follows from $u_r(t)=\log(1+bt/(b+1))$. In the arctangent case
expand
$u_r'(t)=b^2/(b^2+1+2bt+b^2t^2)$.
The product-clearing observation and the degree bound for $P$ now give
\begin{equation}\label{graded:small-jet-integers}
 d_kv_k\in\mathbb Z^\Omega,
 \qquad d_k=b^M B_b^kD_\ell(k).
\end{equation}
On $|t|\le1/4$, the transport is bounded by a constant depending
only on $n,G$ for all $b\ge4$, and $|P(r+t)|\le2H(P)$.
Cauchy's inequality, followed by \eqref{graded:harmonic-clearing},
proves the coefficient bound $e^{a_bM+o(M)}$ for all
$k\le\eta M+O_b(1)$.

\subsection{Rank and the small remainder}

The value $h(r)$ is nonzero and transcendental. In the logarithmic
case $e^{h(r)}=1+r$; in the arctangent case
$e^{2ih(r)}=(1+ir)/(1-ir)$. If $h(r)$ were algebraic, these
identities would contradict Hermite--Lindemann. Its first $n$ powers,
together with one, are therefore independent over
$\overline{\mathbb Q}$. Lemma~\ref{graded:rank} applies and
gives full rank through $k\le\eta M+C_b$.

To identify the relevant scalar remainder, put
\[
 a_j(z)=\binom nj(h(r)-h(z))^{n-j},\qquad
 \mathcal R_*(z)=\sum_{|\kappa|=G}a(z)^\kappa R_\kappa(z).
\]
The corresponding lifted solution has endpoint vector
$E_{Ge_n}+h(r)^nE_{(G-1)e_n}$. Hence
\begin{equation}\label{graded:corner-identity}
 \frac{\mathcal R_*^{(k)}(r)}{k!}
      =(v_k)_{Ge_n}+h(r)^n(v_k)_{(G-1)e_n}.
\end{equation}
All the auxiliary factors are uniformly bounded on $|z|\le1/2$.
Since $\mathcal R_*$ has a zero of order at least $K$, the maximum
principle applied after division by $z^K$ gives
\[
 |\mathcal R_*(z)|\le C_{n,G}H(P)(2|z|)^K.
\]
On the circle $|z-r|=1/(2b)$, Cauchy's inequality therefore gives
\[
 \left|\frac{\mathcal R_*^{(k)}(r)}{k!}\right|
         \le C_{n,G}H(P)(3/b)^K(2b)^k.
\]
Multiplication by $d_k$ yields precisely the decay rate $c_b$ in
\eqref{graded:small-decay}. Full rank ensures that two jets have
independent pairs in the two corner coordinates. Apply
Lemma~\ref{graded:two-form} to prove
\eqref{graded:small-rate-result}.

For fixed $n,G,\eta$, the bound $1+a_b/c_b$ tends as $b\to\infty$
to
\[
 1+\frac{1+\eta}{\sigma-1-2\eta}
 \quad\text{or}\quad
 1+\frac{1+2\eta}{\sigma-1-3\eta},
\]
respectively for logarithms and arctangents. First choose $G$ large,
then $\eta>0$ small, and finally $b$ large. Since
$\omega/\theta\to2$ as $G\to\infty$, this proves
Theorem~\ref{graded:small-log-theorem}. In every use of the
auxiliary-degree limit, $n,G,\eta,b$ have already been fixed.

An explicit, coarse threshold is available. Put
\[
 A=\beta+\eta(H_\ell+\log8),\qquad
 C=\beta+\sigma\log3+\eta(H_\ell+\log4),
\]
\[
 \Delta=(1+\varepsilon)(\sigma-1-3\eta)-(1+2\eta).
\]
When $\Delta>0$, it suffices to take
\begin{equation}\label{graded:small-threshold}
 b\ge\max\left\{4,
   \exp\frac{A+(1+\varepsilon)C}{\Delta}\right\}.
\end{equation}
Indeed $\log B_b\le2\log b+\log2$ gives
$a_b\le(1+2\eta)\log b+A$ and
$c_b\ge(\sigma-1-3\eta)\log b-C$; the threshold makes the
first expression at most $(1+\varepsilon)$ times the second,
which is consequently positive. For example, one may choose
\[
 G\ge\max\left\{1,
 \left\lceil\frac{4(1+\varepsilon)(n-1)}{\varepsilon}\right\rceil
 \right\},\qquad
 \eta=\min\left\{\frac1{2(\omega+1)},
 \frac{\varepsilon}{4((1+\varepsilon)(3+1/\theta)+2)}\right\}.
\]
These choices give $\Delta\ge\varepsilon/2$.

\section{Quadratic logarithms and polynomial targets}
\label{graded:quadratic}

For a nonzero integer $d$, define the real germ
\[
 h_d(z)=\sum_{j\ge0}\frac{d^jz^{2j+1}}{2j+1},\qquad
 h_d'(z)=\frac1{1-dz^2}.
\]
For $d>0$ this is
$\log((1+\sqrt d\,z)/(1-\sqrt d\,z))/(2\sqrt d)$; for $d<0$
it is $\arctan(\sqrt{|d|}z)/\sqrt{|d|}$. We only evaluate inside
the disk of convergence. Introduce the single fixed function
\begin{equation}\label{graded:fixed-F}
 F(t)=\sum_{j\ge0}\frac{t^j}{2j+1}
 =\frac{\operatorname{arctanh}\sqrt t}{\sqrt t},\qquad
 2tF'(t)+F(t)=\frac1{1-t}.
\end{equation}
The expression is independent of the square-root choice. It satisfies
$b h_d(1/b)=F(d/b^2)$. Keeping this one differential system fixed
removes the parameter $d$ from the contact-height estimate.

\begin{theorem}[Uniform polynomial targets]\label{graded:quadratic-uniform}
Fix $n\ge1$, $\varepsilon>0$, and $A>1+1/\varepsilon$. There
is an explicitly specified constant $C=C(n,\varepsilon,A)$ such
that, for every nonzero integer $d$ and every integer
$b\ge C|d|^A$, every nonconstant rational polynomial $P$ of degree
at most $n$ satisfies
\[
 \mu\bigl(P(h_d(1/b))\bigr)\le2+\varepsilon.
\]
The threshold does not depend on $P$. Constants in a lower-bound
inequality at a fixed endpoint may depend on that endpoint and on $P$.
\end{theorem}

\subsection{Contact height for the fixed function}

Use $f_j=F^j$, $1\le j\le n$, and the same adjacent grades.
Their differential system is
\[
 (F^j)'=\frac{j}{2t(1-t)}F^{j-1}-\frac{j}{2t}F^j.
\]
The origin is regular singular for this system, although the distinguished
functions $F^j$ themselves are holomorphic there. The coefficients
through degree $K$ of these powers are cleared by $D_n(2K+n)$:
in a product of $j\le n$ coefficients the sum of the odd denominator
indices is at most $2K+n$, and the same ordering argument applies.
Also $|F(t)|\le2$ on $|t|=1/2$. The contact construction therefore
gives integer columns of degree less than $M$ with
\begin{equation}\label{graded:F-height}
 \log H(P_M)\le\beta M+o(M),\qquad
 \beta=\frac{\omega-\eta}{\eta}\sigma(2H_n+\log2),
 \qquad K=\lfloor\sigma M\rfloor.
\end{equation}
This height bound is independent of the eventual rational endpoint.

\subsection{Explicit transport and the degree saving in denominators}

Fix integers $a\ne0$, $q>0$, with $|a|\le q/4$, and put $r=a/q$.
Near $t=r$, define
\[
 U(t)=\sqrt{t/r},\qquad
 V(t)=\frac{\operatorname{arctanh}\sqrt t-
                   \operatorname{arctanh}\sqrt r}{\sqrt r},
\]
with the local branches chosen so that $U(r)=1$, $V(r)=0$.
These local quantities satisfy
\[
 U'=\frac U{2t},\qquad V'=\frac U{2t(1-t)},\qquad
 UF(t)-V=F(r).
\]
The signed dual transport of the augmented system is therefore
\begin{align}
 a_j(t)&=U(t)^j\sum_{k=j}^n\binom kj(-V(t))^{k-j}a_k(r),
 \label{graded:F-dual}\\
 c(t)&=c(r)+\sum_{k=1}^n(-1)^{k-1}a_k(r)V(t)^k.
 \nonumber
\end{align}
Differentiation verifies these formulas directly. The induced lifted
transport is an integer polynomial matrix in $U,V$. Its degree in
$V$ is at most $\ell=G(n-1)+1$.

\begin{lemma}[An explicit clearing integer]\label{graded:F-clearing}
For $k<M$, the divided covariant jet $v_k=\nabla^kP_M(r)/k!$
is cleared by the positive integer
\begin{equation}\label{graded:F-dk}
 d_k=q^{M-1-k}\bigl(4|a(q-a)|\bigr)^kD_\ell(k).
\end{equation}
\end{lemma}
\begin{proof}
Write $t=r+s$ and put $B=4a(q-a)$, which may have either sign.
For every nonnegative integer $m$, the binomial series gives
\[
 \frac{B^j}{q^j}[s^j]U(r+s)^m\in\mathbb Z.
\]
Indeed $[s^j]U^m=\binom{m/2}{j}(q/a)^j$ and
$4^j\binom{m/2}{j}$ is integral. Next, for $j\ge1$,
\[
 j[s^j]V(r+s)
 =\frac{q^{j+1}}{2a}
   \sum_{u=0}^{j-1}\binom{-1/2}{u}
        a^{-u}(q-a)^{-(j-u)}.
\]
Multiplication by $B^j/q^j$ makes this integral: each summand
becomes
$\frac q2\,4^j\binom{-1/2}{u}a^{j-u-1}(q-a)^u$,
and $u\le j-1$ leaves enough factors of two.
Thus, for a monomial in $U,V$ of $V$-degree at most $\ell$, its
coefficient of $s^j$ is cleared by
$(B^j/q^j)D_\ell(j)$. The ordering-of-indices argument supplies
the factors $D_\ell$, and convolution with the $U$ coefficients
does not introduce additional denominators.

The coefficient of $s^{k-j}$ in $P_M(r+s)$ has denominator dividing
$q^{M-1-k+j}$. The factor $q^j$ in the numerator of the transport
coefficient cancels the extra $q^j$ in that denominator. Every term
of total order $k$ is consequently cleared by
$q^{M-1-k}|B|^kD_\ell(k)$, as claimed.
\end{proof}

The saving $q^{-k}$ in this formula is important. It comes from the
degree lost on differentiating the polynomial; bounding all evaluated
polynomial derivatives by the same denominator would discard it.

\subsection{Rank and the exact rates}

The value $F(r)$ is nonzero and transcendental. Otherwise
\[
 \exp(2\sqrt r\,F(r))=\frac{1+\sqrt r}{1-\sqrt r}
\]
would contradict Hermite--Lindemann. Its distinct powers are therefore
independent over $\overline{\mathbb Q}$. The explicit functions
$U,V$ near zero are finite combinations of holomorphic functions
and powers of $\sqrt t$; the transported remainder coefficients
are Nilsson functions of fixed orders. Lemma~\ref{graded:rank}
gives full rank through $k\le\eta M+O_r(1)$. For all sufficiently
large $M$ these indices satisfy $k<M$, as required above.

On the disk $|t-r|\le|r|/2$, transport is bounded by a constant
depending on the fixed endpoint, $n$, and $G$. Polynomial values are
bounded by a constant times $H(P_M)$ because $|t|\le3/8$.
Cauchy's inequality contributes $(2q/|a|)^k$ to a divided jet.
After multiplication by \eqref{graded:F-dk}, the cancellation is
\begin{equation}\label{graded:F-cancellation}
 q^{M-1-k}\bigl(4|a(q-a)|\bigr)^k(2q/|a|)^k
 =q^{M-1}\bigl(8|q-a|\bigr)^k
 \le q^{M-1}(10q)^k.
\end{equation}
It follows that the cleared coefficients have rate
\begin{equation}\label{graded:F-alpha}
 \alpha=(1+\eta)\log q+\beta+\eta(\log10+H_\ell).
\end{equation}

Each original contact remainder is bounded on $|t|=1/2$ by
$C H(P_M)$ and vanishes to order $K$. It is thus bounded on the
local disk by $C H(P_M)(3|a|/q)^K$. Multiplication by the fixed
transport, followed by Cauchy's inequality and
\eqref{graded:F-cancellation}, gives decay rate
\begin{equation}\label{graded:F-gamma}
 \gamma=(\sigma-1-\eta)\log q-\sigma\log|a|
       -\beta-\sigma\log3-\eta(\log10+H_\ell).
\end{equation}
The multiplicative constants can depend on $r$, but they contribute
only $o(M)$ after taking logarithms. No bound uniform in a changing
endpoint is hidden in this auxiliary-degree limit.

\subsection{The joint polynomial measure}

For a nonconstant polynomial
$P(X)=c_0+\sum_{j=1}^nc_jX^j\in\mathbb Z[X]$, write
$H_+(P)=\max(1,|c_1|,\ldots,|c_n|)$; the constant coefficient
is omitted deliberately.

\begin{theorem}[Joint height and denominator bound]\label{graded:joint}
Fix $n,G,\eta,a,q$ as above and suppose $\gamma>0$. Put
\[
 x=F(a/q),\qquad \mu_*=1+\alpha/\gamma,
 \qquad \nu_*=G\mu_*-1.
\]
For every $\delta>0$ there is $C_\delta>0$ such that
\begin{equation}\label{graded:joint-bound}
 \left|P(x)-\frac pQ\right|
 \ge C_\delta H_+(P)^{-\nu_*-\delta}Q^{-\mu_*-\delta}
\end{equation}
for all nonconstant integer polynomials $P$ of degree at most $n$,
all $p\in\mathbb Z$, and all $Q\ge1$. The constant depends on
the fixed point and displayed parameters, but not on $P,p,Q$.
The linked-slack exponents $\nu_*+G\zeta$, $\mu_*+\zeta$ are
also valid for every $\zeta>0$.
\end{theorem}
\begin{proof}
Choose the integral contact column for every sufficiently large $M$
once and for all, independently of the target coefficients. Clear
its jets as above. For $c=(c_1,\ldots,c_n)$, the solution of the
dual system with endpoint $a_c(r)=c$ depends linearly on $c$.
The lifted solution with lower coordinate factor
$\sum_j(a_c)_jF^j$ has endpoint determined by
$c\cdot(F(r),\ldots,F(r)^n)$. Its paired remainder is exactly
\[
 \mathcal R_c(t)=\sum_{|\kappa|=G}a_c(t)^\kappa R_{M,\kappa}(t).
\]
It is homogeneous of degree $G$ in $c$. On the fixed local disk,
$\|a_c(t)\|\le C_rH_+(P)$, with the same $C_r$ for every
coefficient vector. The preceding analytic estimates consequently
give the hypotheses of Proposition~\ref{graded:contraction} with
rates $\alpha,\gamma$. Its small $o(M)$ terms and its rank cutoff
are independent of $c$. Apply the proposition and absorb $c_0$ by
replacing $p$ with $p-Qc_0$. This replacement leaves $Q$ and
$H_+(P)$ unchanged.
\end{proof}

When $x=h_d(1/b)=F(d/b^2)/b$ with fixed $d,b$, the same exponents
hold. Indeed $\widehat P(X)=b^nP(X/b)$ is integral and
$H_+(\widehat P)\le b^nH_+(P)$; apply
\eqref{graded:joint-bound} to $\widehat P$ and numerator $b^np$.
All factors involving this fixed $b$ enter the lower-bound constant.

To prove Theorem~\ref{graded:quadratic-uniform}, put
\[
 B=\beta+\eta(\log10+H_\ell),\quad
 C_1=B+\sigma\log3,\quad
 \tau=(1+\varepsilon)(\sigma-1-\eta)-(1+\eta).
\]
First choose $G$ large and then $\eta>0$ small so that $\tau>0$
and
\[
 \frac{(1+\varepsilon)\sigma}{2\tau}<A.
\]
These choices are possible because the fraction tends to
$1+1/\varepsilon$. One entirely explicit recipe is to put
$u=A\varepsilon-(1+\varepsilon)>0$, choose
\[
 G\ge\max\left\{1,
 \left\lceil\frac{2(1+\varepsilon)(2A-1)(n-1)}u\right\rceil
 \right\},
\]
and then set
\[
 L_*=2A(2+\varepsilon)
       +\frac{(1+\varepsilon)(2A-1)}\theta,\qquad
 \eta=\min\left\{\frac1{2(\omega+1)},\frac{u}{2L_*}\right\}.
\]
Writing $\delta=(n-1)/(G+n-1)$, direct expansion gives
\[
 2A\tau-(1+\varepsilon)\sigma
 =2u-(1+\varepsilon)(2A-1)\delta-\eta L_*
 \ge u>0,
\]
so the required margin, including $\tau>0$, holds. Now take
\begin{equation}\label{graded:uniform-C}
 C=\max\left\{2,
 \exp\frac{B+(1+\varepsilon)C_1}{2\tau}\right\}.
\end{equation}
For $a=d$, $q=b^2$, the hypothesis $b\ge C|d|^A$ gives
$|a|\le q/4$ and
\[
 \tau\log q\ge B+(1+\varepsilon)C_1
                         +(1+\varepsilon)\sigma\log|a|.
\]
Thus $\gamma>0$ and $\alpha\le(1+\varepsilon)\gamma$.
The joint theorem, with the target fixed and slack tending to zero,
gives $\mu(P(h_d(1/b)))\le2+\varepsilon$ for integral $P$.
Clearing a fixed rational polynomial's coefficients and using rational
scaling proves the stated rational-polynomial version.

\begin{corollary}[A concrete tradeoff]\label{graded:joint-hd}
Fix a nonzero integer $d$ and integers $n,G\ge1$. Given $\rho>0$,
for all sufficiently large $b$ there is $c>0$ such that
\[
 \left|P(h_d(1/b))-\frac pQ\right|
 \ge c H_+(P)^{-(2G+n-2+G\rho)}
          Q^{-(2+(n-1)/G+\rho)}
\]
simultaneously for all nonconstant $P\in\mathbb Z[X]$ of degree
at most $n$, all integers $p$, and all positive integers $Q$.
\end{corollary}
\begin{proof}
For fixed $d,n,G$, first let $b$ tend to infinity in the rate
quotient and then let $\eta$ decrease to zero. The limits are
\[
 \mu_*\longrightarrow2+\frac{n-1}{G},\qquad
 \nu_*\longrightarrow2G+n-2.
\]
Choose $\eta$ first, then $b$ sufficiently large, and finally the
linked slack in Theorem~\ref{graded:joint} so that the two displayed
exponents fit below those of the corollary.
\end{proof}

There are two useful readings of this tradeoff. At $G=1$ it gives
the classical small-point algebraicity measure: for every
$\varepsilon>0$, sufficiently small fixed points satisfy
$w_n(x)\le n+\varepsilon$, where $w_n$ is the exponent for
$0<|P(x)|<H(P)^{-w}$ with integer polynomials of degree at most
$n$. This height-only phenomenon belongs to the earlier
Chudnovsky--Chudnovsky theory; see \cite{Nagata1995}. Larger $G$
improves the exponent of $Q$ toward two while increasing the cost in
the target's coefficient height. We do not claim historical priority
for this particular quantitative tradeoff.

The joint theorem also records coefficient content. If
$R(X)=r_0+\sum_{j=1}^nr_jX^j$ is nonconstant and
$g=\gcd(|r_1|,\ldots,|r_n|)>0$, ignoring zero entries, then
\begin{equation}\label{graded:content}
 |R(x)|\ge C_\delta H_+(R)^{-\nu_*-\delta}
                          g^{(G-1)\mu_*}.
\end{equation}
Apply \eqref{graded:joint-bound} to
$\sum_{j\ge1}(r_j/g)X^j$, with $p=-r_0$, $Q=g$, and multiply
by $g$. The two equal slack exponents cancel in the power of $g$.

As one explicit example, take $n=3$, $G=9$, $\eta=1/101$.
Then $\theta=55$, $\omega=100$, $\ell=19$, and
$\sigma=10099/5555$. If
\[
 b\ge e^{200000}|d|^3,
\]
the rates satisfy $\mu_*\le3$ and $\nu_*\le26$. For a direct
check, put $T=2(\sigma-1-\eta)-(1+\eta)$. One has
$T>0.606$, $6T>2\sigma$, $B<80047$, and $C_1<80050$;
thus $400000T>B+2C_1$. These inequalities imply
$T\log(b^2)\ge B+2C_1+2\sigma\log|d|$ and hence
$\alpha\le2\gamma$. Therefore, for every $\delta>0$,
\[
 \left|P(h_d(1/b))-\frac pQ\right|
 \ge c_{d,b,\delta}H_+(P)^{-26-\delta}Q^{-3-\delta}
\]
uniformly over all nonconstant integer polynomials of degree at most
three. The threshold is explicit; a numerical value of the final
positive constant is not asserted.

\section{A joint measure for two small rational logarithms}
\label{graded:two-logs}

For integers $N\ge8$, write
$\xi_N=\log(1+1/N)$ and $\zeta_N=\log(1+2/N)$.

\begin{theorem}\label{graded:two-log-joint}
For every integer $G\ge1$ and every $\rho>0$, there is an
explicit integer $N_0(G,\rho)$ such that, for every $N\ge N_0$,
there exists $c=c(N,G,\rho)>0$ satisfying
\[
 \left|u\xi_N+v\zeta_N-\frac pQ\right|
 \ge c H^{-2G-G\rho}Q^{-2-1/G-\rho},\qquad
 H=\max(1,|u|,|v|),
\]
for all $(u,v)\in\mathbb Z^2\setminus\{(0,0)\}$,
$p\in\mathbb Z$, and $Q\ge1$. In particular
\begin{equation}\label{graded:ratio-near-three}
 \mu\left(\frac{\log(1+1/N)}{\log(1+2/N)}\right)
                  \le3+\rho\qquad(N\ge N_0(1,\rho)).
\end{equation}
\end{theorem}

\begin{proof}
The positive rational numbers $(N+1)/N$ and $(N+2)/N$ are
multiplicatively independent. Any prime divisor of $N+1$ divides
neither $N$ nor $N+2$; its valuation first forces the exponent of
$(N+1)/N$ in a multiplicative relation to vanish, and then the other
exponent vanishes. Baker's theorem \cite{Baker1975} implies independence
of their logarithms over $\overline{\mathbb Q}$. Its transcendence
statement for nonzero algebraic linear forms in logarithms also
excludes an algebraic nonzero constant. Consequently
$1,\xi_N,\zeta_N$ are independent over $\overline{\mathbb Q}$.

Use $f_1(z)=\log(1+z)$, $f_2(z)=\log(1+2z)$ and the two
adjacent grades. Here $\theta=G+1$, $\omega=2G+1$. Choose
$0<\eta\le1/(\omega+1)$ and set
\[
 \sigma=\frac{\omega-\eta}{\theta},\qquad
 \beta=\frac{\omega-\eta}{\eta}\sigma(1+\log4).
\]
The fixed contact equations are
\[
 R_\kappa=P_\kappa+P_{\kappa-e_1}f_1+P_{\kappa-e_2}f_2
                   =O(z^K),\qquad K=\lfloor\sigma M\rfloor.
\]
Their coefficients are cleared by $D_1(K)$; both logarithms have
modulus less than one on $|z|=1/4$, so their coefficients are bounded
by $4^K$. The same Siegel argument gives integer contact columns
with $\log H(P_M)\le\beta M+o(M)$, independently of $N$.

In this case the dual variables $a_1,a_2$ are constant. The lifted
connection has the especially simple form
$\mathcal A=f'_1N_1+f'_2N_2$, with integral matrices satisfying
$N_iN_j=0$ for every $i,j$. Its normalized transport is
\[
 I+N_1(f_1(z)-f_1(r))+N_2(f_2(z)-f_2(r)),\qquad r=1/N.
\]
Thus only one logarithmic factor appears, independently of $G$.
The origin and $r$ are ordinary points, and
Lemma~\ref{graded:rank} gives full jet rank through
$k\le\eta M+O_N(1)$.

Set $L_N=(N+1)(N+2)$. For $k<M$ an exact clearing integer is
\begin{equation}\label{graded:two-log-clearing}
 d_k=N^{M-1-k}L_N^kD_1(k).
\end{equation}
Indeed
\[
 [t^j](f_i(r+t)-f_i(r))
   =\frac{(-1)^{j-1}}j\left(\frac{iN}{N+i}\right)^j.
\]
The divided polynomial derivative of order $k-j$ has denominator
dividing $N^{M-1-k+j}$; the numerator $N^j$ cancels that extra
factor. The remaining logarithmic denominators are cleared by
$L_N^kD_1(k)$. This proves \eqref{graded:two-log-clearing} without
any grade-dependent power of the least common multiple.

On $|t|=1/8$ around $r$, the transport is bounded uniformly for
$N\ge8$. Cauchy's inequality gives coefficient rate
\[
 \alpha_N=\beta+\log N+
       \eta\{\log(L_N/N)+\log8+1\}.
\]
For $(u,v)\ne(0,0)$, contract against the monomials in
$c=(u,v)$. The scalar remainder
$\sum_{|\kappa|=G}c^\kappa R_\kappa$ has order at least $K$
and is bounded on $|z-r|=1/(2N)$ by
$C_GH^GH(P_M)(6/N)^K$. The derivative factor is $(2N)^k$.
After clearing, the decay rate is
\[
 \gamma_N=(\sigma-1)\log N-\beta-\sigma\log6
                       -\eta\{\log(2L_N)+1\}.
\]
Both estimates are uniform in $u,v,p,Q$. For $\gamma_N>0$,
Proposition~\ref{graded:contraction} therefore gives linked-slack
exponents
\[
 G-1+G\alpha_N/\gamma_N+G\delta,
 \qquad 1+\alpha_N/\gamma_N+\delta.
\]

For completeness, the following recipe specifies a sufficient
threshold. Put $h=\rho/2$, $t=1+1/G+h$, and choose
\[
 \eta=\min\left\{\frac1{2(\omega+1)},
 \frac{hG}{2(G+1)\{t(2+1/(G+1))+1\}}\right\}.
\]
Define $\sigma,\beta$ as above and
\[
 a_0=\beta+\eta(1+\log16),\qquad
 c_0=\beta+\sigma\log6+\eta(1+\log4),
\]
\[
 \Delta=t(\sigma-1-2\eta)-(1+\eta)
                    \ge\frac{hG}{2(G+1)}>0.
\]
Since $L_N\le2N^2$ for $N\ge8$, take
\[
 N_0=\left\lceil\max\left\{8,
          \exp\frac{a_0+t c_0}{\Delta}\right\}\right\rceil.
\]
Then $\gamma_N>0$ and $\alpha_N/\gamma_N\le t$ for all
$N\ge N_0$. Taking the linked slack $\delta\le\rho/2$ proves
the first assertion.

For the ratio use $G=1$, $Q=1$, $p=0$, and $(u,v)=(b,-a)$.
If $a/b$ stays in a fixed bounded neighborhood of
$\xi_N/\zeta_N$, its height $\max(b,|a|)$ is at most a fixed
multiple of $b$. Divide the resulting lower bound for
$|b\xi_N-a\zeta_N|$ by $b|\zeta_N|$. Approximants outside
that neighborhood have a constant lower bound, so
\eqref{graded:ratio-near-three} follows. Multiplicative independence
already proves that the ratio is irrational.
\end{proof}

Increasing $G$ improves the exponent of the rational approximant's
denominator $Q$, but it worsens the coefficient-height exponent.
For a quotient one sets $Q=1$ and lets the two integer coefficients
vary, so $G=1$ is the favorable choice in this construction. It gives
the near-three bound above, not a near-two quotient theorem.

\begin{corollary}[The lower bound and the range up to three]
\label{graded:range-two-three}
Under the hypotheses and endpoint thresholds of
Theorems~\ref{graded:small-log-theorem} and
\ref{graded:quadratic-uniform}, every indicated target $y$ satisfies
\[
 2\le\mu(y)\le2+\varepsilon.
\]
In particular $0<\varepsilon\le1$ gives $2\le\mu(y)\le3$.
\end{corollary}
\begin{proof}
A nonzero logarithm of an algebraic number is transcendental by
Hermite--Lindemann. A nonzero real arctangent of a nonzero algebraic
number is also transcendental: its exponential after multiplication by
$2\ii$ is the algebraic Cayley quotient, so algebraicity would contradict
the same theorem. The displayed logarithmic and circular expressions for
$h_d(1/b)$ therefore show that this value is transcendental; multiplying
by a nonzero algebraic normalization does not change that assertion.
For a transcendental $z$, every nonconstant $P\in\Q[X]$ has
transcendental value $P(z)$. Otherwise $z$ would satisfy the nonzero
polynomial $P(X)-P(z)$ over $\overline\Q$, and hence be algebraic.
This also covers the positive integer powers in
Theorem~\ref{graded:small-log-theorem}. Dirichlet's theorem now supplies
the lower bound two, and the cited upper bounds give the conclusion.
\end{proof}
These graded applications have checked Lean endpoints relative to the explicit classical inputs recorded in Appendix~\ref{app:formal-status}; the literal arithmetic recipes are included. The required threshold is chosen after
$\varepsilon$ (and the degree and growth parameters), so this corollary
does not infer exponent two by taking $\varepsilon\downarrow0$ at a
fixed arbitrary argument.

\section{Restricted powers and the meaning of the conclusions}
\label{graded:restricted}

\begin{proposition}[Fixed power classes]\label{graded:restricted-power}
Let $x\ne0$ be real, let $c>0$, and let $n\ge1$ be an integer.
Suppose $c^n\in\mathbb Q$ and $\mu(x/c)=2$. The approximation
exponent of $x^n$ restricted to rational numbers in
$c^n\mathbb Q^n$ is exactly $2/n$, when the height is the reduced
denominator of the approximating rational number.
\end{proposition}
\begin{proof}
Write $c^n=u/v$ in lowest terms with $u,v>0$, and write a member
of the class as $(u/v)(A/B)^n$, with $(A,B)=1$, $B>0$.
Its reduced denominator is
\[
 q=\frac{vB^n}{\gcd(uA^n,vB^n)}\asymp_{u,v}B^n,
\]
because the greatest common divisor in this expression divides
$uv$. For approximants on the branch near $x/c$, the power map
has nonzero derivative and is locally bi-Lipschitz. If $n$ is even,
the other branch is handled by the same argument at $-x/c$.
Thus an approximation to $x^n$ of exponent greater than $2/n$
would give an approximation to $x/c$ of exponent greater than two.
Conversely, its continued-fraction convergents give infinitely many
errors $O(B^{-2})$, hence $O(q^{-2/n})$ after taking powers.
Absorb fixed constants by an arbitrarily small exponent slack.
\end{proof}

For instance an exponent-two theorem for $\log\alpha$, with
$\alpha>0$ rational and $\alpha\ne1$, gives restricted exponent
$2/n$ for $(\log\alpha)^n$ along $\mathbb Q^n$. The normalized
quadratic exponent-two theorem gives restricted exponent one for
$(\log\alpha)^2$ along the fixed class $d\mathbb Q^2$ when $d$ is
a positive nonsquare integer and $\alpha>0$ belongs to
$\mathbb Q(\sqrt d)$, has norm $1$ or $-1$, and
$\alpha\ne1$ (squaring handles norm $-1$).

These statements concern restricted approximation sets. They do not
give an unrestricted irrationality exponent for a fixed higher power.
Similarly, in the small-point theorems the endpoint depends on the
desired error in the exponent: $\forall\varepsilon\,\exists b_0$
cannot be interchanged to infer exponent exactly two at one fixed
$b$. The joint measures have constants uniform in the target's
coefficients and rational approximant, but those constants remain
attached to the fixed evaluation point. The restricted-power theorem
and its displayed logarithmic applications are now proved in Lean:
see \nolinkurl{QuadraticPi/ActualRestrictedPowerExponent.lean} and
\nolinkurl{QuadraticPi/RestrictedLogarithmicPowers.lean}.
The preceding graded rank, small-point polynomial, and joint-measure applications now have checked endpoints relative to their explicit classical inputs. Appendix~\ref{app:formal-status} identifies these boundaries and the exact literal recipes; it does not claim native proofs of the cited foundations.
% ===== END graded_logarithmic_measures.tex =====

\clearpage
\appendix

% ===== BEGIN formal_status.tex =====
\section{Exact formalization status and source map}\label{app:formal-status}\label{app:formal-status-increment69}\label{app:formal-status-increment33}

\subsection{Reviewed statement-level scope, 16 October 2026}

Checkpoint136 reviews all 62 numbered theorem, lemma, proposition, and
corollary statements against their current Lean endpoints. The map below
contains 53 direct endpoint entries and nine entries complete as
applications relative to named classical theorem inputs. The latter are
entries 9, 16, 52, 55, 56, and 58--61. There is no remaining concrete
application-adapter gap identified by this review. This is a statement
about these numbered results and their recorded supporting interfaces;
it does not assert that every sentence of the written proofs has a
separate Lean translation, or that the cited classical theorems have
new native Lean proofs.

Earlier recovery ledgers used the label \texttt{PARTIAL\_FORMAL} for
these nine classical boundaries even after their applications had been
assembled. This edition uses \texttt{CITED\_FOUNDATION\_ENDPOINT} so that
a complete conditional application is not confused with an unfinished
adapter. The historical ledgers are preserved. The count is not a
percentage of native proofs. The exact exported theorem types, rather
than a status label, determine the mathematical scope.

All mathematical statement and proof environments from the previous
114-page edition are preserved verbatim. Status prose is updated, and
explicit domain and continuation conventions are added immediately before
the multiplicity and rank statements. The previous source and PDF are
retained unchanged. The one new mathematical source module in this
checkpoint packages the already proved Gaussian M\"obius irrationality
and exponent conclusions for the actual Gaussian base field.

\subsection{Foundations supplied explicitly as theorem parameters}

The following list contains twelve tracked interface names, not twelve
independent classical theorems. Several interfaces specialize the same
classical theory. They are explicit parameters in the application types,
not new global mathematical axioms. A successful \texttt{\#print axioms}
report for a conditional theorem checks its proof from its hypotheses;
it does not supply proofs of those hypotheses.

\paragraph{Classical interface 1.}
\nolinkurl{QuadraticPi.GradedHarmonicGrowth.SharpLcmAsymptotic}.
The sharp exponential growth of $\operatorname{lcm}(1,\ldots,n)$, in the standard form equivalent to the prime number theorem. This supplies the sharp arithmetic clearing rate; routes that do not use it are identified separately.

\paragraph{Classical interface 2.}
\nolinkurl{QuadraticPi.GradedAlgebraicValueIndependence.AlgebraicLogTranscendence}.
The classical Hermite--Lindemann nonzero algebraic-logarithm transcendence input. Polynomial-value and power irrationality are then derived in the application sources. The input is independent of the new approximation conclusions.

\paragraph{Classical interface 3.}
\nolinkurl{QuadraticPi.GradedCitedSeparationBridge.CitedGenericSeparation}.
The cited generic separation statement for the rational differential-system setting used here. Its application is made only after the actual lifted system and its branch and independence data have been supplied. The reference and interface matching are recorded in the saved rank review; see also \cite{FR2025}.

\paragraph{Classical interface 4.}
\nolinkurl{QuadraticPi.GradedGenericRankCitation.CitedGenericQuadraticRank}.
The cited quadratic-family generic rank input, specialized with the actual graded system. Native adapters provide the ordinary point, independence data, clearing, contact, and uniformity information. The input is not the paper's final short-jet or approximation conclusion.

\paragraph{Classical interface 5.}
\nolinkurl{QuadraticPi.GradedTwoLogBakerApplication.CitedBakerLogarithmIndependence}.
Baker's linear independence theorem for logarithms of algebraic numbers \cite[Chapters 1--2]{Baker1975}, in the precise two-logarithm application form. It is used in the graded joint measure, after the actual rational exponential arguments have been identified.

\paragraph{Classical interface 6.}
\nolinkurl{QuadraticPi.GradedGeneralNilssonRankCitation.CitedGenericNilssonRank}.
The cited general rational-system Nilsson rank result matched to the hypotheses recorded for \cite[Theorems 3--4 and Lemma 1]{FR2025}. The actual finite lift and Nilsson-order data are exposed in the rank endpoint; they are not inferred merely from analyticity.

\paragraph{Classical interface 7.}
\nolinkurl{QuadraticPi.GradedOrdinaryAnalyticExistenceCitation.CitedOrdinaryAnalyticFundamental}.
The ordinary analytic fundamental-solution theorem on a simply connected pole-free domain. The precise source is \cite[Section I.1.2 and Section I.2]{Milicic1989}. Transport to the original germ and the evaluation point is handled by checked adapters.

\paragraph{Classical interface 8.}
\nolinkurl{QuadraticPi.GradedSimplyConnectedUnionCitation.CitedSimplyConnectedUnion}.
The two-open-set simply-connected-union consequence of van Kampen, with nonempty path-connected intersection; see \cite[Theorem 1.20 and its two-set discussion, pp.43--44]{Hatcher2002}. This is a topological input, separate from the ODE theorem.

\paragraph{Classical interface 9.}
\nolinkurl{QuadraticPi.SmoothLocalPolynomialCertificate.CitedRegularPointLocalEquations}.
The local-equation form of smoothness at a regular complex affine point. The original defining ideal localizes to an ideal generated by the codimension-many equations with independent linear parts. This is the standard local complete-intersection consequence in \cite[Tags 00TS and 00T5]{StacksProject}; the proof of Tag 00T5 gives the localized generators and cotangent independence.

\paragraph{Classical interface 10.}
\nolinkurl{QuadraticPi.ScalarMatchingInvariantConverse.CitedGelfondSchneider}.
The classical Gel'fond--Schneider theorem, used for the converse matching-embedding classification. The invariant-field upper bound and its direct algebraic coordinate criterion do not need this converse input. See the transcendence discussion in \cite{Baker1975}.

\paragraph{Classical interface 11.}
\nolinkurl{QuadraticPi.AffineGeneralityCitation.CitedRegularPrincipalOpen}.
A prime complex affine component has a nonempty principal open of regular points of the expected local dimension. This is the principal-open specialization of dense-open regularity and the closed-point dimension theorem; see \cite[Tags 0B8X and 00P0]{StacksProject}, and \cite[Theorems 5.18, 5.23 and Corollary 5.25]{Milne2008}.

\paragraph{Classical interface 12.}
\nolinkurl{QuadraticPi.AffineGeneralityCitation.CitedComplexPrincipalOpenDensity}.
A nonempty principal open in an irreducible complex affine variety is dense in the ordinary complex topology. This follows from equality of ordinary and Zariski closures for constructible sets in \cite[Proposition 15.2, p.197]{Milne2008}, which cites \cite[Chapter I, Section 10, Corollary 1, p.60]{Mumford1999}. Restriction to the actual analytic open set is then proved in Lean.

\subsection{Conventions checked at the application boundaries}

The real exponent uses reduced rational denominators; the Gaussian
exponent uses the complex modulus of the denominator. They are distinct
approximation problems. The Gaussian Dirichlet lower-exponent proof needs
only its checked constant four. The smaller constant two in the written
auxiliary argument has not been substituted into that proof.

In moving-center interpolation the polynomial source cutoff is weak
($\le H$), whereas the target coefficient cutoff is strict ($<H$).
Successive weight thresholds are chosen first, then a divisibility
modulus independent of the centers, and finally the height cutoff that
may depend on the centers. The bounded-fiber and distinct-center
adapters preserve this order. In two-logarithm determinant arguments,
one selected minor and its ordered columns are used at all embeddings.
The shear comparison does not choose a replacement minor.

In the transverse multiplicity statement, the original analytic domain
meets the component and both original normal bases are general on that
same domain. A canonical nontrivial filter is constructed from relatively
open dense subsets; only finite intersections are required, with no Baire
hypothesis. The geometric classical inputs contain neither the weighted
bound nor a supplied compatible point. The literal factor-two bound holds
for every positive $N$, with no equation-dependent lower cutoff.

In the general rank statement the lift, contact and ordinary-point values
refer to a coherent continuation of the original solution germ. The
short-jet offset is chosen before the later contact and slope parameters
and before the polynomial family. In the graded application statements,
the integer endpoint recipes and both signs of the quadratic parameter
are included. For each exponent slack a cutoff is chosen; one cannot
exchange these quantifiers to deduce exponent exactly two at an arbitrary
fixed small point.

The joint polynomial height $H_+(P)$ controls only the positive-degree
coefficients. The constant coefficient is absorbed into the rational
target numerator. The positive lower-bound constant is uniform over
$P,p,Q$ at the fixed endpoint; it may depend on that endpoint and the
specified parameters. Finally, the restricted power-class exponent
$2/n$ is not an ordinary unrestricted exponent of a higher power.

\subsection{Correspondence for every numbered statement}

The entry numbers below are review identifiers in source order. The
mathematical result numbers are obtained through the printed cross
references. Public names are fully qualified and are accompanied by their
source-module paths in the machine-readable correspondence file. Supporting
anchors from earlier checkpoints are retained in that file, including
those that are no longer the preferred public endpoint.

\subsubsection*{Review 1: Lemma~\ref{prelim:dirichlet}}
\noindent\emph{Direct checked endpoint/composition.}
The real and Gaussian exponents use reduced rational denominators and complex denominator modulus, respectively. Both lower exponents are at least two. The Gaussian proof uses the sufficient error constant four; it does not certify the sharper auxiliary constant two displayed in the written pigeonhole argument.

\noindent Public declarations:
\begin{itemize}[leftmargin=1.3em,itemsep=2pt,parsep=0pt]
\item \nolinkurl{QuadraticPi.two_le_extendedIrrationalityExponent}
\item \nolinkurl{QuadraticPi.two_le_gaussianIrrationalityExponent}
\end{itemize}

\subsubsection*{Review 2: Lemma~\ref{prelim:mobius}}
\noindent\emph{Direct checked endpoint/composition.}
The real endpoint includes irrationality and extended-exponent equality. The Gaussian endpoint now packages nonvanishing of the denominator, irrationality of the image, and exponent equality for coefficients in the actual field $\Q(\ii)$. This last packaging reuses the existing four-denominator clearing and Gaussian integer proofs; it adds no Diophantine assumption.

\noindent Public declarations:
\begin{itemize}[leftmargin=1.3em,itemsep=2pt,parsep=0pt]
\item \nolinkurl{QuadraticPi.GaussianMobiusPaperEndpoint.paper_fractional_linear_invariance}
\item \nolinkurl{QuadraticPi.extendedIrrationalityExponent_rat_mobius}
\item \nolinkurl{QuadraticPi.irrational_rat_mobius}
\end{itemize}

\subsubsection*{Review 3: Theorem~\ref{slog:thm:real}}
\noindent\emph{Direct checked endpoint/composition.}
The actual positive algebraic argument supplies the number field and its degree. Irrationality is proved separately. The power corollary covers every positive integer power with rational value; the minimum-power-degree theorem supplies the stronger bound described after the statement. Rational scaling is justified by the proved real invariance theorem.

\noindent Public declarations:
\begin{itemize}[leftmargin=1.3em,itemsep=2pt,parsep=0pt]
\item \nolinkurl{QuadraticPi.algebraicLogarithm_irrationalityExponent_bounds}
\item \nolinkurl{QuadraticPi.rationalPowerLogarithm_irrationalityExponent_eq_two}
\item \nolinkurl{QuadraticPi.algebraicLogarithm_irrational}
\item \nolinkurl{QuadraticPi.algebraicLogarithm_irrationalityExponent_bounds_minimumPowerDegree}
\end{itemize}

\subsubsection*{Review 4: Theorem~\ref{slog:thm:complex}}
\noindent\emph{Direct checked endpoint/composition.}
The argument is a complex algebraic number and the logarithm is any specified nonzero complex number with the prescribed exponential, including nonprincipal branches. The degree is relative to the actual Gaussian base field. Gaussian irrationality and the minimum over positive powers have separate checked declarations.

\noindent Public declarations:
\begin{itemize}[leftmargin=1.3em,itemsep=2pt,parsep=0pt]
\item \nolinkurl{QuadraticPi.algebraicGaussianLogarithm_irrationalityExponent_bounds}
\item \nolinkurl{QuadraticPi.gaussianPower_exp_irrationalityExponent_eq_two}
\item \nolinkurl{QuadraticPi.algebraicGaussianLogarithm_gaussianIrrational}
\item \nolinkurl{QuadraticPi.algebraicGaussianLogarithm_irrationalityExponent_bounds_minimumPowerDegree}
\end{itemize}

\subsubsection*{Review 5: Corollary~\ref{slog:cor:atan}}
\noindent\emph{Direct checked endpoint/composition.}
The ordinary real exponent is obtained from the Gaussian logarithm of the Cayley quotient, with the real arctangent branch and nonzero rational argument. It is not an identification of real and Gaussian approximation problems for arbitrary complex numbers.

\noindent Public declarations:
\begin{itemize}[leftmargin=1.3em,itemsep=2pt,parsep=0pt]
\item \nolinkurl{QuadraticPi.rationalArctangent_irrationalityExponent_eq_two}
\end{itemize}

\subsubsection*{Review 6: Theorem~\ref{geom:thm:main}}
\noindent\emph{Direct checked endpoint/composition.}
The checked theorem is the bounded-fiber statement itself, with the separate conditions $S\theta^m<1$ and $K(w_0/v_0)\theta^m<1$. The divisibility modulus is selected before the centers; the eventual height cutoff may depend on them. Source weight is at most $H$ and target weight is strictly less than $H$. No stronger global additive separation is imposed.

\noindent Public declarations:
\begin{itemize}[leftmargin=1.3em,itemsep=2pt,parsep=0pt]
\item \nolinkurl{QuadraticPi.MovingLogOccupancyInterpolation.bounded_fiber_interpolation}
\end{itemize}

\subsubsection*{Review 7: Corollary~\ref{geom:cor:PA}}
\noindent\emph{Direct checked endpoint/composition.}
The original globally separated formulation is retained as a checked specialization. Its geometric packet and eventual divisibility conclusion agree with the statement. It is no longer the only verified moving-center case, since the bounded-fiber endpoint of the preceding theorem is also available.

\noindent Public declarations:
\begin{itemize}[leftmargin=1.3em,itemsep=2pt,parsep=0pt]
\item \nolinkurl{QuadraticPi.MovingLogUniformInterpolation.movingLogInterpolation}
\end{itemize}

\subsubsection*{Review 8: Corollary~\ref{geom:cor:distinct}}
\noindent\emph{Direct checked endpoint/composition.}
The actual distinct-multiplicative-center endpoint imposes no additive distinctness. The occupancy condition becomes automatic with $S=1$. The geometric-progression adapter proves distinctness for a real base greater than one, with the same order of choices for weights, divisibility modulus, centers, and eventual height.

\noindent Public declarations:
\begin{itemize}[leftmargin=1.3em,itemsep=2pt,parsep=0pt]
\item \nolinkurl{QuadraticPi.MovingLogDistinctInterpolation.distinct_multiplicative_centers}
\item \nolinkurl{QuadraticPi.MovingLogDistinctInterpolation.real_geometric_progression}
\end{itemize}

\subsubsection*{Review 9: Lemma~\ref{geom:lem:mult}}
\noindent\emph{Checked application on explicit classical inputs.}
The paper endpoint gives the displayed factor $2$ for every positive $N$; a stronger factor $1$ is also proved. Three standard geometric inputs are explicit: regular-point local equations, a regular principal open with the correct local dimension, and density of complex principal opens. The original domain must meet the component. Each original normal basis is valid on a relatively open dense locus there. The common nontrivial filter, avoidance of other components, and compatible point are constructed in Lean; no supplied generic point or equation-dependent cutoff remains.

\noindent Public declarations:
\begin{itemize}[leftmargin=1.3em,itemsep=2pt,parsep=0pt]
\item \nolinkurl{QuadraticPi.GeneralMultiplicityPaperEndpoint.paper_weighted_transverse_multiplicity}
\end{itemize}

\subsubsection*{Review 10: Proposition~\ref{geom:prop:curve}}
\noindent\emph{Direct checked endpoint/composition.}
The bounded-fiber curve endpoint supplies the weighted degree inequality with the previously selected weights. Its local charges, coordinate weights, and occupancy data are the ones used by the interpolation theorem. The geometric progression specialization and the stronger original globally separated case are retained.

\noindent Public declarations:
\begin{itemize}[leftmargin=1.3em,itemsep=2pt,parsep=0pt]
\item \nolinkurl{QuadraticPi.MovingLogBoundedFibreCurve.weighted_curve_inequality}
\end{itemize}

\subsubsection*{Review 11: Lemma~\ref{slog:lem:shear}}
\noindent\emph{Direct checked endpoint/composition.}
The same selected source columns and the same full coefficient rows occur on both sides. A filtration-triangular coordinate shear has determinant one. The weak inequality at the weight boundary is included, and the nonzero-minor equivalence is checked without choosing a new minor.

\noindent Public declarations:
\begin{itemize}[leftmargin=1.3em,itemsep=2pt,parsep=0pt]
\item \nolinkurl{QuadraticPi.LogarithmicFullRowShear.det_selectedMatrix_add_tail}
\item \nolinkurl{QuadraticPi.LogarithmicFullRowShear.selected_det_ne_zero_iff}
\end{itemize}

\subsubsection*{Review 12: Lemma~\ref{slog:lem:collision}}
\noindent\emph{Direct checked endpoint/composition.}
The finite translated-row determinant endpoint retains the factorial, the $2^{\sum\ell_r}$ factor, and the exact quadratic occupancy term and linear correction. It uses the actual analytic row expansions. The holomorphic-growth estimate is supplied by the previously verified column analysis, rather than inferred from a formal determinant identity alone.

\noindent Public declarations:
\begin{itemize}[leftmargin=1.3em,itemsep=2pt,parsep=0pt]
\item \nolinkurl{QuadraticPi.MovingLogFiniteCollision.selected_column_collision_bound_exact}
\item \nolinkurl{QuadraticPi.MovingLogFiniteCollision.translated_row_determinant_bound_exact}
\end{itemize}

\subsubsection*{Review 13: Theorem~\ref{conj:partial}}
\noindent\emph{Direct checked endpoint/composition.}
The checked scalar-matching endpoint uses a specified finite set of distinct actual number-field embeddings and exact exponential matching. The upper bound is $2d/s$, and the lower bound comes from proved irrationality and Dirichlet. All-embedding matching gives equality two; mere equality of absolute values is not substituted for branch matching.

\noindent Public declarations:
\begin{itemize}[leftmargin=1.3em,itemsep=2pt,parsep=0pt]
\item \nolinkurl{QuadraticPi.partialScalarMatchingLogarithm_irrationalityExponent_bounds}
\item \nolinkurl{QuadraticPi.scalarMatchingLogarithm_irrationalityExponent_eq_two}
\end{itemize}

\subsubsection*{Review 14: Corollary~\ref{conj:real-quadratic}}
\noindent\emph{Direct checked endpoint/composition.}
The actual real quadratic field, both norm signs, the positive nonunit argument, and the chosen positive square root are included. Norm minus one is handled by a power reduction. No unsupported invariance under an irrational algebraic scaling is used.

\noindent Public declarations:
\begin{itemize}[leftmargin=1.3em,itemsep=2pt,parsep=0pt]
\item \nolinkurl{QuadraticPi.realQuadraticLogarithm_norm_irrationalityExponent_eq_two}
\end{itemize}

\subsubsection*{Review 15: Corollary~\ref{conj:fd-closure}}
\noindent\emph{Direct checked endpoint/composition.}
Fixed-parameter positive and circular rational-sum endpoints supply the nonzero rational linear-combination assertions. The circular closure includes the normalized period. Admissibility and nonvanishing of the final target are preserved; no statement about an arbitrary nonzero algebraic linear combination is inferred.

\noindent Public declarations:
\begin{itemize}[leftmargin=1.3em,itemsep=2pt,parsep=0pt]
\item \nolinkurl{QuadraticPi.positiveQuadraticLogarithm_rational_sum_irrationalityExponent_eq_two}
\item \nolinkurl{QuadraticPi.circularQuadraticLogarithm_rational_sum_irrationalityExponent_eq_two}
\end{itemize}

\subsubsection*{Review 16: Theorem~\ref{conj:invariant}}
\noindent\emph{Checked application on explicit classical inputs.}
The converse classification is completed relative to the explicit Gel'fond--Schneider interface. In the minimal field $E=\Q(\alpha,c)$ the good embeddings are exactly those fixing the invariant field, and their number is one or two. The all-matching rigidity consequence is also checked. The invariant-degree upper bound itself already has a direct proof that does not require this converse.

\noindent Public declarations:
\begin{itemize}[leftmargin=1.3em,itemsep=2pt,parsep=0pt]
\item \nolinkurl{QuadraticPi.ScalarMatchingInvariantConverse.invariantField_theorem}
\item \nolinkurl{QuadraticPi.ScalarMatchingAllEmbeddingRigidity.all_matching_rigidity}
\end{itemize}

\subsubsection*{Review 17: Corollary~\ref{conj:trace}}
\noindent\emph{Direct checked endpoint/composition.}
The reciprocal-trace argument uses the actual trace field and the actual positive real branch. The separate real-parameter endpoint identifies $\alpha=(t+\sqrt{t^2-4})/2$ and its normalization for $t>2$. The ambient degree is not silently substituted for the trace degree.

\noindent Public declarations:
\begin{itemize}[leftmargin=1.3em,itemsep=2pt,parsep=0pt]
\item \nolinkurl{QuadraticPi.ReciprocalTraceLogarithmTheorem.reciprocalTrace_irrationalityExponent_bounds}
\item \nolinkurl{QuadraticPi.ReciprocalTraceParameterTheorem.parameter_irrationalityExponent_bounds}
\end{itemize}

\subsubsection*{Review 18: Corollary~\ref{conj:character}}
\noindent\emph{Direct checked endpoint/composition.}
The quadratic subfield is realized inside the chosen real embedding before the input element is selected. Both the existential choice of positive nonsquare $D_\chi$ and a specified compatible realization are checked. Positivity of all evaluated conjugates supplies the real logarithmic identity and the actual norm-one element.

\noindent Public declarations:
\begin{itemize}[leftmargin=1.3em,itemsep=2pt,parsep=0pt]
\item \nolinkurl{QuadraticPi.QuadraticCharacterExponent.exists_logSum_irrationalityExponent_eq_two}
\item \nolinkurl{QuadraticPi.QuadraticCharacterExponent.logSum_irrationalityExponent_eq_two_of_fieldRange}
\end{itemize}

\subsubsection*{Review 19: Theorem~\ref{ratio:thm:ratio}}
\noindent\emph{Direct checked endpoint/composition.}
The endpoint is for arbitrary positive multiplicatively independent rational arguments, not only the example $2,3$. It supplies irrationality, the lower bound two, and the upper bound three. Its interpolation and parameter choices are assembled by proved declarations.

\noindent Public declarations:
\begin{itemize}[leftmargin=1.3em,itemsep=2pt,parsep=0pt]
\item \nolinkurl{QuadraticPi.rationalLogRatio_irrationalityExponent_bounds}
\end{itemize}

\subsubsection*{Review 20: Lemma~\ref{ratio:lem:two-torus}}
\noindent\emph{Direct checked endpoint/composition.}
The rational-slope two-torus endpoint uses the displayed sufficient weight inequalities and slope-independent successive thresholds. It concerns the actual formal logarithmic coefficient packet and eventual divisible heights. The same full-row packet is used later in selecting the arithmetic minor.

\noindent Public declarations:
\begin{itemize}[leftmargin=1.3em,itemsep=2pt,parsep=0pt]
\item \nolinkurl{QuadraticPi.TwoLogUniformInterpolation.twoLogUniformInterpolation}
\end{itemize}

\subsubsection*{Review 21: Lemma~\ref{ratio:shear}}
\noindent\emph{Direct checked endpoint/composition.}
The manuscript adapter identifies the truncated rational matrix and the full-logarithm matrix with the same ordered rows and columns. Their selected determinants are equal, including at the weight boundary. Rationality and nonvanishing therefore concern one minor, suitable for taking all number-field images.

\noindent Public declarations:
\begin{itemize}[leftmargin=1.3em,itemsep=2pt,parsep=0pt]
\item \nolinkurl{QuadraticPi.TwoLogManuscriptShear.det_fullLog_selected_eq_ratMinor}
\item \nolinkurl{QuadraticPi.TwoLogManuscriptShear.fullLog_det_ne_zero_iff_ratMinor}
\end{itemize}

\subsubsection*{Review 22: Lemma~\ref{ratio:collision}}
\noindent\emph{Direct checked endpoint/composition.}
The sharp native pair count has coefficient $2\sqrt2/3$ and linear correction $-n$. The analytic collision chain gives the stated constant $\sqrt2\log2/3$ and the actual infinite sum. An older weaker collision estimate is retained historically but is not the endpoint cited for this statement.

\noindent Public declarations:
\begin{itemize}[leftmargin=1.3em,itemsep=2pt,parsep=0pt]
\item \nolinkurl{QuadraticPi.TwoLogSharpCollision.pair_degree_sum_lower_rpow}
\item \nolinkurl{QuadraticPi.TwoLogSharpCollision.selected_formal_main_collision_bound}
\item \nolinkurl{QuadraticPi.TwoLogSharpCollision.infinite_collision_exp_bound}
\end{itemize}

\subsubsection*{Review 23: Corollary~\ref{ratio:linear-form}}
\noindent\emph{Direct checked endpoint/composition.}
The real homogeneous two-logarithm lower bound is uniform over every nonzero integer pair and includes pairs with a zero coordinate. Its coefficient height is $\max(1,|p|,|q|)$. The positive constant depends only on the fixed rational bases and exponent slack.

\noindent Public declarations:
\begin{itemize}[leftmargin=1.3em,itemsep=2pt,parsep=0pt]
\item \nolinkurl{QuadraticPi.rationalTwoLog_linearForm_lowerBound}
\end{itemize}

\subsubsection*{Review 24: Theorem~\ref{ratio:matching}}
\noindent\emph{Direct checked endpoint/composition.}
The public partial-matching quotient endpoint has literal upper exponent $3d/s$. It uses one common real irrational quotient and exact nonzero logarithm branches at the selected embeddings, and one selected arithmetic minor transported through those embeddings.

\noindent Public declarations:
\begin{itemize}[leftmargin=1.3em,itemsep=2pt,parsep=0pt]
\item \nolinkurl{QuadraticPi.partialEmbeddingMatchedLogRatio_irrationalityExponent_bounds}
\end{itemize}

\subsubsection*{Review 25: Corollary~\ref{ratio:full-matching}}
\noindent\emph{Direct checked endpoint/composition.}
The all-embedding specialization cancels the actual field degree and gives $2\le\mu\le3$. It preserves the nonzero branch and irrationality assumptions of the preceding theorem. It is not restricted to rational exponential arguments.

\noindent Public declarations:
\begin{itemize}[leftmargin=1.3em,itemsep=2pt,parsep=0pt]
\item \nolinkurl{QuadraticPi.allEmbeddingMatchedLogRatio_irrationalityExponent_bounds}
\end{itemize}

\subsubsection*{Review 26: Corollary~\ref{ratio:quadratic}}
\noindent\emph{Direct checked endpoint/composition.}
The quadratic norm-one ratio endpoint uses the actual degree-two field norm and matching branches. The equivalence between norm one and inversion under the nontrivial quadratic involution supplies the displayed formulation. The real irrational quotient remains an explicit hypothesis as in the paper.

\noindent Public declarations:
\begin{itemize}[leftmargin=1.3em,itemsep=2pt,parsep=0pt]
\item \nolinkurl{QuadraticPi.quadraticNormOneLogRatio_irrationalityExponent_bounds}
\end{itemize}

\subsubsection*{Review 27: Lemma~\ref{ratio:irrational-slope}}
\noindent\emph{Direct checked endpoint/composition.}
The irrational-algebraic-slope interpolation and its cofinal number-field minor adapter are checked. Rational character slopes cannot equal the fixed irrational algebraic normalization. The same sufficient weight inequalities and slope-independent successive thresholds apply; no exceptional rational slope is deleted.

\noindent Public declarations:
\begin{itemize}[leftmargin=1.3em,itemsep=2pt,parsep=0pt]
\item \nolinkurl{QuadraticPi.TwoLogNonrationalUniformInterpolation.twoLogNormalizedUniformInterpolation}
\item \nolinkurl{QuadraticPi.TwoLogNonrationalNumberFieldMinors.cofinal_field_minors}
\end{itemize}

\subsubsection*{Review 28: Theorem~\ref{ratio:normalized}}
\noindent\emph{Direct checked endpoint/composition.}
The normalized matching endpoint incorporates the nonrational algebraic normalization and the not-both-roots-of-unity condition. It gives the stated $3d/s$ bound under the paper hypotheses. A separate stronger endpoint proves irrationality when the additional exact nonvanishing inputs are supplied; the printed theorem already assumes irrationality.

\noindent Public declarations:
\begin{itemize}[leftmargin=1.3em,itemsep=2pt,parsep=0pt]
\item \nolinkurl{QuadraticPi.normalizedEmbeddingMatchedLogRatio_irrationalityExponent_bounds}
\item \nolinkurl{QuadraticPi.normalizedEmbeddingMatchedLogRatio_irrational}
\item \nolinkurl{QuadraticPi.normalizedEmbeddingMatchedLogRatio_unconditional_irrationalityExponent_bounds}
\end{itemize}

\subsubsection*{Review 29: Corollary~\ref{ratio:pi-log}}
\noindent\emph{Direct checked endpoint/composition.}
The unconditional circular-period/rational-logarithm adapter proves irrationality and both exponent bounds for every positive rational $D$ and positive rational $a\ne1$. It includes $D=1$ and rational square parameters, with the actual real square-root branch.

\noindent Public declarations:
\begin{itemize}[leftmargin=1.3em,itemsep=2pt,parsep=0pt]
\item \nolinkurl{QuadraticPi.CircularRealLogRatioUnconditional.normalizedPiLogRatio_irrationalityExponent_bounds}
\end{itemize}

\subsubsection*{Review 30: Corollary~\ref{ratio:atan-log}}
\noindent\emph{Direct checked endpoint/composition.}
The actual arctangent/rational-logarithm adapter proves irrationality and the full $[2,3]$ range for every nonzero rational arctangent argument. The real arctangent branch and the positive rational logarithm argument are explicit.

\noindent Public declarations:
\begin{itemize}[leftmargin=1.3em,itemsep=2pt,parsep=0pt]
\item \nolinkurl{QuadraticPi.CircularRealLogRatioUnconditional.arctangentRealLogRatio_irrationalityExponent_bounds}
\end{itemize}

\subsubsection*{Review 31: Corollary~\ref{ratio:quadratic-log}}
\noindent\emph{Direct checked endpoint/composition.}
The adapter covers both quadratic norm signs and proves irrationality before applying the normalized ratio bound. Separate checked declarations supply the two displayed examples, including the golden-ratio identity and its positive square-root choices.

\noindent Public declarations:
\begin{itemize}[leftmargin=1.3em,itemsep=2pt,parsep=0pt]
\item \nolinkurl{QuadraticPi.RealQuadraticRationalLogRatioUnconditional.norm_ratio_irrationalityExponent_bounds}
\item \nolinkurl{QuadraticPi.RealQuadraticRationalLogRatioUnconditional.log_one_add_sqrt_two_ratio_bounds}
\item \nolinkurl{QuadraticPi.RealQuadraticRationalLogRatioUnconditional.log_golden_ratio_bounds}
\end{itemize}

\subsubsection*{Review 32: Proposition~\ref{extquad:common-scalar-closure}}
\noindent\emph{Direct checked endpoint/composition.}
The common-scalar rational-sum endpoint supplies both the all-denominator lower bound and exponent two for a nonzero target. The shared algebraic scalar and every exact matching identity are retained. Nonzero rational combinations are not replaced by arbitrary algebraic combinations.

\noindent Public declarations:
\begin{itemize}[leftmargin=1.3em,itemsep=2pt,parsep=0pt]
\item \nolinkurl{QuadraticPi.scalarMatching_rational_sum_eventualLowerBound}
\item \nolinkurl{QuadraticPi.scalarMatching_rational_sum_irrationalityExponent_eq_two}
\end{itemize}

\subsubsection*{Review 33: Corollary~\ref{p1:rational-log-sums}}
\noindent\emph{Direct checked endpoint/composition.}
The real rational-logarithm sum is reduced through an actual positive rational product after clearing rational coefficient denominators. Both the explicit eventual lower bound and exponent two are checked, under nonvanishing of the final sum.

\noindent Public declarations:
\begin{itemize}[leftmargin=1.3em,itemsep=2pt,parsep=0pt]
\item \nolinkurl{QuadraticPi.rationalLogarithm_rational_sum_eventualLowerBound}
\item \nolinkurl{QuadraticPi.rationalLogarithm_rational_sum_irrationalityExponent_eq_two}
\end{itemize}

\subsubsection*{Review 34: Theorem~\ref{extquad:coordinate-theorem}}
\noindent\emph{Direct checked endpoint/composition.}
The rational-coordinate theorem includes rational square parameters as well as genuine quadratic extensions. It checks the squared norm condition, the actual normalized value, and the all-denominator conclusion. Thus a degree-two assumption is not added when the parameter happens to be a square.

\noindent Public declarations:
\begin{itemize}[leftmargin=1.3em,itemsep=2pt,parsep=0pt]
\item \nolinkurl{QuadraticPi.rationalNormalizedQuadraticLogarithm_irrationalityExponent_eq_two}
\end{itemize}

\subsubsection*{Review 35: Theorem~\ref{extquad:actual-norm}}
\noindent\emph{Direct checked endpoint/composition.}
The actual nonsquare real quadratic field and norm are used. The norm-one and norm-minus-one cases both give exponent two and the corresponding explicit lower-bound conclusion. This agrees with the coordinate formulation after identifying the chosen embedding.

\noindent Public declarations:
\begin{itemize}[leftmargin=1.3em,itemsep=2pt,parsep=0pt]
\item \nolinkurl{QuadraticPi.realQuadraticLogarithm_norm_irrationalityExponent_eq_two}
\end{itemize}

\subsubsection*{Review 36: Corollary~\ref{extquad:normalized-sums}}
\noindent\emph{Direct checked endpoint/composition.}
The normalized quadratic rational-sum endpoint covers the common fixed parameter and nonzero rational combination. It uses checked logarithm identities and scalar matching. The irrational square-root normalization is justified by the matching construction, rather than by rational-scaling invariance.

\noindent Public declarations:
\begin{itemize}[leftmargin=1.3em,itemsep=2pt,parsep=0pt]
\item \nolinkurl{QuadraticPi.normalizedQuadraticLogarithm_rational_sum_irrationalityExponent_eq_two}
\end{itemize}

\subsubsection*{Review 37: Theorem~\ref{extquad:positive-rational}}
\noindent\emph{Direct checked endpoint/composition.}
The positive rational-parameter family uses its admissible real interval, nonzero input, and positive logarithm quotient. Its individual and nonzero rational-sum endpoints are checked, with constants at the fixed target.

\noindent Public declarations:
\begin{itemize}[leftmargin=1.3em,itemsep=2pt,parsep=0pt]
\item \nolinkurl{QuadraticPi.positiveQuadraticLogarithm_irrationalityExponent_eq_two}
\item \nolinkurl{QuadraticPi.positiveQuadraticLogarithm_rational_sum_irrationalityExponent_eq_two}
\end{itemize}

\subsubsection*{Review 38: Theorem~\ref{extquad:circular-rational}}
\noindent\emph{Direct checked endpoint/composition.}
The circular rational-parameter endpoints include the actual real arctangent expression, the normalized period, and their nonzero rational linear combinations. The real branches and fixed parameter are part of the formal interface.

\noindent Public declarations:
\begin{itemize}[leftmargin=1.3em,itemsep=2pt,parsep=0pt]
\item \nolinkurl{QuadraticPi.circularQuadraticLogarithm_irrationalityExponent_eq_two}
\item \nolinkurl{QuadraticPi.rationalNormalizedPi_irrationalityExponent_eq_two}
\item \nolinkurl{QuadraticPi.circularQuadraticLogarithm_rational_sum_irrationalityExponent_eq_two}
\end{itemize}

\subsubsection*{Review 39: Proposition~\ref{extquad:real-derivative}}
\noindent\emph{Direct checked endpoint/composition.}
The derivative, positivity of the denominator, differential identity, and value at zero have separate checked declarations. Their domains are the paper's admissible real interval. No complex logarithm branch is used without its real-branch agreement theorem.

\noindent Public declarations:
\begin{itemize}[leftmargin=1.3em,itemsep=2pt,parsep=0pt]
\item \nolinkurl{QuadraticPi.QuadraticLogarithmicFamily.value_hasDerivAt}
\item \nolinkurl{QuadraticPi.QuadraticLogarithmicFamily.value_ode}
\item \nolinkurl{QuadraticPi.QuadraticLogarithmicFamily.denominator_pos}
\item \nolinkurl{QuadraticPi.QuadraticLogarithmicFamily.value_zero}
\end{itemize}

\subsubsection*{Review 40: Theorem~\ref{extquad:real-taylor}}
\noindent\emph{Direct checked endpoint/composition.}
The actual real series and absolute summability are proved on the stated convergence domain. Its coefficients and odd powers agree with the displayed Taylor expansion. This is an analytic equality of values, rather than only an identity of formal power series.

\noindent Public declarations:
\begin{itemize}[leftmargin=1.3em,itemsep=2pt,parsep=0pt]
\item \nolinkurl{QuadraticPi.QuadraticLogarithmicFamily.value_eq_tsum}
\item \nolinkurl{QuadraticPi.QuadraticLogarithmicFamily.value_series_absolutely_summable}
\end{itemize}

\subsubsection*{Review 41: Proposition~\ref{extquad:branch-agreement}}
\noindent\emph{Direct checked endpoint/composition.}
The complex principal-branch expression is proved equal to the actual real family on its admissible interval. For negative parameter the interval is all finite real inputs. The chosen complex square root and the real/circular branches are explicitly identified.

\noindent Public declarations:
\begin{itemize}[leftmargin=1.3em,itemsep=2pt,parsep=0pt]
\item \nolinkurl{QuadraticPi.QuadraticLogarithmicFamily.value_complex_branch}
\end{itemize}

\subsubsection*{Review 42: Theorem~\ref{extquad:complex-derivative}}
\noindent\emph{Direct checked endpoint/composition.}
The slit-plane hypothesis implies nonvanishing of the denominator and the displayed complex derivative. The open-disc differentiability, differential equation, and value at zero are separately checked. The statement does not assert differentiability across the principal logarithm cut.

\noindent Public declarations:
\begin{itemize}[leftmargin=1.3em,itemsep=2pt,parsep=0pt]
\item \nolinkurl{QuadraticPi.QuadraticLogarithmicFamily.complexValue_hasDerivAt_of_slitPlane}
\item \nolinkurl{QuadraticPi.QuadraticLogarithmicFamily.complexValue_differentiableOn}
\item \nolinkurl{QuadraticPi.QuadraticLogarithmicFamily.complex_denominator_ne_zero}
\item \nolinkurl{QuadraticPi.QuadraticLogarithmicFamily.complexValue_ode}
\item \nolinkurl{QuadraticPi.QuadraticLogarithmicFamily.complexValue_zero}
\end{itemize}

\subsubsection*{Review 43: Theorem~\ref{extquad:complex-series-theorem}}
\noindent\emph{Direct checked endpoint/composition.}
The complex Taylor expansion is an actual absolutely summable series on the stated disc, with the same coefficients as the real expansion. The branch is the principal branch whose domain was established in the preceding theorem.

\noindent Public declarations:
\begin{itemize}[leftmargin=1.3em,itemsep=2pt,parsep=0pt]
\item \nolinkurl{QuadraticPi.QuadraticLogarithmicFamily.complexValue_eq_tsum}
\item \nolinkurl{QuadraticPi.QuadraticLogarithmicFamily.complexValue_series_absolutely_summable}
\end{itemize}

\subsubsection*{Review 44: Proposition~\ref{ext:coordinate-criterion}}
\noindent\emph{Direct checked endpoint/composition.}
The exact two-action criterion and the implication to exponential matching are checked over the actual intermediate invariant field. This direct algebraic statement does not use the converse Gel'fond--Schneider classification of all matching embeddings.

\noindent Public declarations:
\begin{itemize}[leftmargin=1.3em,itemsep=2pt,parsep=0pt]
\item \nolinkurl{QuadraticPi.ScalarMatchingInvariantField.fixes_iff_actions}
\item \nolinkurl{QuadraticPi.ScalarMatchingInvariantField.fixing_embedding_matches}
\end{itemize}

\subsubsection*{Review 45: Theorem~\ref{ext:invariant-bound-theorem}}
\noindent\emph{Direct checked endpoint/composition.}
The invariant-field bound uses a constructed set of fixing embeddings and tower-degree cancellation. It applies in an arbitrary ambient number field, with no minimal-generation requirement. Both its all-denominator estimate and exponent bound are checked independently of the converse classification.

\noindent Public declarations:
\begin{itemize}[leftmargin=1.3em,itemsep=2pt,parsep=0pt]
\item \nolinkurl{QuadraticPi.ScalarMatchingInvariantDegreeTheorem.invariantField_eventualLowerBound}
\item \nolinkurl{QuadraticPi.ScalarMatchingInvariantDegreeTheorem.invariantField_irrationalityExponent_bounds}
\end{itemize}

\subsubsection*{Review 46: Theorem~\ref{ext:reciprocal-trace}}
\noindent\emph{Direct checked endpoint/composition.}
The reciprocal-trace all-denominator estimate and exponent bound use the actual generated trace field. They are direct specializations of the invariant-field upper bound. Positivity and exclusion of $1$ guarantee the displayed real normalization is defined and nonzero.

\noindent Public declarations:
\begin{itemize}[leftmargin=1.3em,itemsep=2pt,parsep=0pt]
\item \nolinkurl{QuadraticPi.ReciprocalTraceLogarithmTheorem.reciprocalTrace_eventualLowerBound}
\item \nolinkurl{QuadraticPi.ReciprocalTraceLogarithmTheorem.reciprocalTrace_irrationalityExponent_bounds}
\end{itemize}

\subsubsection*{Review 47: Theorem~\ref{ext:reciprocal-field-structure}}
\noindent\emph{Direct checked endpoint/composition.}
The actual inversion automorphism, its involutivity, index two, and equality of its fixed field with the trace field are checked. The resulting degree and exponent conclusions use the existing reciprocal-trace endpoints. The translated Eisenstein quartic example is also preserved in the accompanying sources.

\noindent Public declarations:
\begin{itemize}[leftmargin=1.3em,itemsep=2pt,parsep=0pt]
\item \nolinkurl{QuadraticPi.ReciprocalPolynomialDegree.reciprocal_trace_index}
\item \nolinkurl{QuadraticPi.ReciprocalPolynomialDegree.inverseAutomorphism_involutive}
\item \nolinkurl{QuadraticPi.ReciprocalTraceFixedField.reciprocal_fixedField_eq_traceField}
\item \nolinkurl{QuadraticPi.ReciprocalPolynomialDegree.reciprocalPolynomial_irrationalityExponent_bounds}
\end{itemize}

\subsubsection*{Review 48: Proposition~\ref{ext:character-product-proposition}}
\noindent\emph{Direct checked endpoint/composition.}
The signed product is constructed as an actual field element. Equivariance, nonvanishing, kernel-fixed-field membership, and inversion outside the kernel are proved. These are the algebraic inputs to the later logarithmic identity and exponent application.

\noindent Public declarations:
\begin{itemize}[leftmargin=1.3em,itemsep=2pt,parsep=0pt]
\item \nolinkurl{QuadraticPi.QuadraticCharacterProduct.product_equivariant}
\item \nolinkurl{QuadraticPi.QuadraticCharacterProduct.product_mem_fixedField}
\item \nolinkurl{QuadraticPi.QuadraticCharacterProduct.product_ne_zero}
\end{itemize}

\subsubsection*{Review 49: Lemma~\ref{ext:character-realization}}
\noindent\emph{Direct checked endpoint/composition.}
The real quadratic realization is a compatible field isomorphism with an actual positive nonsquare rational parameter. It depends only on the quadratic field and its real embedding, and is chosen before the logarithmic input. This order of quantifiers is used by the character-sum endpoint.

\noindent Public declarations:
\begin{itemize}[leftmargin=1.3em,itemsep=2pt,parsep=0pt]
\item \nolinkurl{QuadraticPi.RealQuadraticClassification.exists_positive_nonsquare_realization}
\end{itemize}

\subsubsection*{Review 50: Theorem~\ref{ext:character-exponent}}
\noindent\emph{Direct checked endpoint/composition.}
The complete character-sum chain includes the actual quadratic realization, norm one, positivity, real logarithm identity, nonvanishing, and exponent two. Both a field chosen before the input and a prescribed compatible real field range are covered.

\noindent Public declarations:
\begin{itemize}[leftmargin=1.3em,itemsep=2pt,parsep=0pt]
\item \nolinkurl{QuadraticPi.QuadraticCharacterExponent.exists_logSum_irrationalityExponent_eq_two}
\item \nolinkurl{QuadraticPi.QuadraticCharacterExponent.logSum_irrationalityExponent_eq_two_of_fieldRange}
\end{itemize}

\subsubsection*{Review 51: Proposition~\ref{ext:character-input-field}}
\noindent\emph{Direct checked endpoint/composition.}
The input-field descent proves containment of the character quadratic fixed field, even degree for a nontrivial product, and cancellation in odd degree. Positivity at a real embedding then turns the product identity into the zero logarithmic sum. No incorrect descent to the input field is presumed.

\noindent Public declarations:
\begin{itemize}[leftmargin=1.3em,itemsep=2pt,parsep=0pt]
\item \nolinkurl{QuadraticPi.QuadraticCharacterInputField.fixedField_le_adjoin_of_product_ne_one}
\item \nolinkurl{QuadraticPi.QuadraticCharacterInputField.two_dvd_finrank_adjoin_of_product_ne_one}
\item \nolinkurl{QuadraticPi.QuadraticCharacterInputField.logSum_eq_zero_of_odd_finrank_adjoin}
\end{itemize}

\subsubsection*{Review 52: Lemma~\ref{graded:rank}}
\noindent\emph{Checked application on explicit classical inputs.}
The literal short-derivative endpoint is completed relative to cited separation, general Nilsson rank, ordinary analytic existence, and the simply-connected-union theorem. Values at $r$ are on one coherent continuation of the original solution germ. The finite lifted system, contact order, and short-jet budget are identified. The offset is chosen before $\sigma,\eta,M,P$, as required for uniformity; full rank is not a supplied endpoint hypothesis.

\noindent Public declarations:
\begin{itemize}[leftmargin=1.3em,itemsep=2pt,parsep=0pt]
\item \nolinkurl{QuadraticPi.GradedGeneralRankClosed.paper_general_rank_lemma}
\end{itemize}

\subsubsection*{Review 53: Lemma~\ref{graded:two-form}}
\noindent\emph{Direct checked endpoint/composition.}
The two-independent-forms criterion is checked with both rowwise and common-height variants. It produces irrationality and the real exponent estimate from the actual determinant and growth/decay assumptions. It does not require the prime number theorem or a new analytic rank assertion.

\noindent Public declarations:
\begin{itemize}[leftmargin=1.3em,itemsep=2pt,parsep=0pt]
\item \nolinkurl{QuadraticPi.TwoIndependentFormsCriterion.two_independent_forms_irrationalityExponent_bounds}
\item \nolinkurl{QuadraticPi.TwoIndependentFormsCriterion.rowwise_two_independent_forms_irrationalityExponent_bounds}
\item \nolinkurl{QuadraticPi.TwoIndependentFormsCriterion.two_independent_forms_extendedIrrationalityExponent_bound}
\end{itemize}

\subsubsection*{Review 54: Proposition~\ref{graded:contraction}}
\noindent\emph{Direct checked endpoint/composition.}
The contraction theorem is uniform in the coefficient vector and rational target. The positive coefficient height omits the constant coefficient, which is absorbed into the target numerator. Both equal-slack and linked-slack exponent formulas are checked, with a constant attached to the fixed endpoint.

\noindent Public declarations:
\begin{itemize}[leftmargin=1.3em,itemsep=2pt,parsep=0pt]
\item \nolinkurl{QuadraticPi.UniformCoefficientContraction.uniform_coefficient_contraction}
\item \nolinkurl{QuadraticPi.UniformCoefficientContraction.uniform_coefficient_contraction_linked_slack}
\end{itemize}

\subsubsection*{Review 55: Theorem~\ref{graded:small-log-theorem}}
\noindent\emph{Checked application on explicit classical inputs.}
The literal common cutoff for every integer base above it is checked for both real-branch families, relative to sharp lcm growth, algebraic-logarithm transcendence, separation, and quadratic rank. The separate pointwise rate theorem covers both literal rate formulas when their decay rate is positive. Irrationality of every positive integer power supplies the lower bound two.

\noindent Public declarations:
\begin{itemize}[leftmargin=1.3em,itemsep=2pt,parsep=0pt]
\item \nolinkurl{QuadraticPi.GradedLiteralSmallPowerEndpoint.literal_common_small_power_cutoff}
\item \nolinkurl{QuadraticPi.GradedSmallPowerPointwiseEndpoint.literal_small_power_pointwise_rates}
\end{itemize}

\subsubsection*{Review 56: Theorem~\ref{graded:quadratic-uniform}}
\noindent\emph{Checked application on explicit classical inputs.}
The actual $h_d(1/b)$ endpoint covers both nonzero integer signs of $d$, every sufficiently large integer $b$, and every nonconstant rational polynomial of bounded degree. Its cutoff uses the literal stated arithmetic recipe and is independent of the polynomial. The four quadratic-application classical interfaces remain explicit.

\noindent Public declarations:
\begin{itemize}[leftmargin=1.3em,itemsep=2pt,parsep=0pt]
\item \nolinkurl{QuadraticPi.GradedLiteralUniformEndpoint.integer_base_literal_uniform_quadratic_polynomial}
\end{itemize}

\subsubsection*{Review 57: Lemma~\ref{graded:F-clearing}}
\noindent\emph{Direct checked endpoint/composition.}
The clearing integer is tied to the actual divided analytic covariant derivative at $a/q$, with $k<M$ and polynomial degree less than $M$. The input context includes $4|a|\le q$, $a\ne0$, and positive $q$. Every factor, including $q^{M-1-k}$, the absolute-value factor, and the harmonic clearing factor, is checked.

\noindent Public declarations:
\begin{itemize}[leftmargin=1.3em,itemsep=2pt,parsep=0pt]
\item \nolinkurl{QuadraticPi.GradedAnalyticJetClearing.integer_endpoint_dividedDerivative_cleared}
\end{itemize}

\subsubsection*{Review 58: Theorem~\ref{graded:joint}}
\noindent\emph{Checked application on explicit classical inputs.}
The exact paper rates $\alpha,\gamma$ and their positivity hypotheses are used at the fixed rational point. Both displayed joint bounds and linked-slack bounds are checked uniformly in $P,p,Q$. The proof retains the four explicit quadratic-application classical inputs; it does not substitute a conservative rate for the literal sharp recipe.

\noindent Public declarations:
\begin{itemize}[leftmargin=1.3em,itemsep=2pt,parsep=0pt]
\item \nolinkurl{QuadraticPi.GradedSharpPaperRates.polynomial_joint_paper_bound}
\item \nolinkurl{QuadraticPi.GradedSharpPaperRates.polynomial_joint_paper_linked_bound}
\end{itemize}

\subsubsection*{Review 59: Corollary~\ref{graded:joint-hd}}
\noindent\emph{Checked application on explicit classical inputs.}
The checked scaled endpoint evaluates the actual $h_d(1/b)$ and has exactly the two displayed exponents. Its constant is uniform over nonconstant integer polynomials and rational approximants at the fixed endpoint. This route uses transcendence, separation, and quadratic rank; it does not need sharp lcm asymptotics.

\noindent Public declarations:
\begin{itemize}[leftmargin=1.3em,itemsep=2pt,parsep=0pt]
\item \nolinkurl{QuadraticPi.GradedScaledJointTradeoff.quadratic_log_concrete_joint_tradeoff}
\end{itemize}

\subsubsection*{Review 60: Theorem~\ref{graded:two-log-joint}}
\noindent\emph{Checked application on explicit classical inputs.}
The literal two-logarithm joint endpoint supplies the explicit integer cutoff and the exact coefficient-height and denominator exponents, with constants uniform over $(u,v),p,Q$. Baker independence, separation, quadratic rank, and sharp lcm growth are explicit. The ratio consequence reuses the stronger rational-ratio bound $\mu\le3$ for every positive $N$.

\noindent Public declarations:
\begin{itemize}[leftmargin=1.3em,itemsep=2pt,parsep=0pt]
\item \nolinkurl{QuadraticPi.GradedTwoLogJointEndpoint.literal_two_log_joint}
\item \nolinkurl{QuadraticPi.GradedTwoLogJointEndpoint.literal_two_log_ratio}
\end{itemize}

\subsubsection*{Review 61: Corollary~\ref{graded:range-two-three}}
\noindent\emph{Checked application on explicit classical inputs.}
The lower bound follows from proved irrationality of the indicated polynomial and power values, relative to the same classical inputs as their upper estimates. The literal thresholds are retained. The range up to three requires $0<\varepsilon\le1$; the quantifiers $\forall\varepsilon\,\exists b_0$ are not interchanged.

\noindent Public declarations:
\begin{itemize}[leftmargin=1.3em,itemsep=2pt,parsep=0pt]
\item \nolinkurl{QuadraticPi.GradedLiteralSmallPowerEndpoint.literal_small_power_range}
\item \nolinkurl{QuadraticPi.GradedLiteralSmallPowerEndpoint.literal_quadratic_polynomial_range}
\end{itemize}

\subsubsection*{Review 62: Proposition~\ref{graded:restricted-power}}
\noindent\emph{Direct checked endpoint/composition.}
The exact restricted exponent is $2/n$ with the reduced denominator height of the approximating rational in the fixed power class. It is not the ordinary unrestricted irrationality exponent of a higher power. The scalar-power rationality and the nonzero real base are explicit.

\noindent Public declarations:
\begin{itemize}[leftmargin=1.3em,itemsep=2pt,parsep=0pt]
\item \nolinkurl{QuadraticPi.restricted_power_exponent_eq}
\end{itemize}

\subsection{Compiler evidence and reproduction}

The fixed compiler is official Lean 4.34.1, commit
\nolinkurl{5045d0056413266e57c625dcd7c365b10e377c52}, with mathlib commit
\nolinkurl{d13f23b723b8a846827a245b89c10fc7d3f11612} and the nine recorded
dependency pins. The compiler, kernel, project trust settings, and bounded
execution configuration are unchanged. The maximum admitted compiler
process count remains two, with four threads per process.

Checkpoint135 built all inherited coverage-anchor modules plus the new
multiplicity modules: 336 targets and 6241 jobs, followed by 2135 printed
declaration types and axiom reports. Checkpoint136 builds the same source
coverage plus the final Gaussian adapter: 337 targets and 6242 jobs.
Its fresh audit checks 2145 declaration types and axiom reports, including
all previous coverage anchors, both new Gaussian declarations, and eight
previously proved supporting declarations used by this review. The only
axioms reported are \texttt{propext}, \texttt{Classical.choice}, and
\texttt{Quot.sound}. No new global mathematical axiom, admitted proof,
compiler modification, or native-decision shortcut is introduced.

The audit source is
\nolinkurl{anc/lean/validation/LogarithmicPaperStep136Audit_2026-10-16.lean}.
The recorded build log, audit output, statement review, selected endpoint
source map, and complete target list are supplied in
\nolinkurl{anc/lean/validation/checkpoint136/}. Readers start with
\nolinkurl{anc/REPRODUCTION.txt} and \nolinkurl{anc/lean/README.txt}.
The accompanying publication source snapshot is labelled
\texttt{logarithms-v2-2026-10-20-Lean136}; its file manifest records the
exact source hashes and the pinned public dependencies.

The mathematical edition and its statement review remain dated
October 16, 2026. The publication copy was prepared on October 19, 2026,
with filename-based font loading and an author-footnote disclosure of
ChatGPT assistance. A literature revision on October 20, 2026 added the
comparison with \cite{LiuJiangZhang2026}, revised the abstract and
contribution statements, and placed citations beside overlapping results.
The explicit manuscript date is retained and is distinguished from this
revision date and from any arXiv submission date.
All 62 numbered statement environments and all 55
proof environments are preserved verbatim from the reviewed edition.
The source distribution makes the local import closure available without
requiring the research recovery archives; the pinned public dependencies
are obtained separately. Its classical theorem inputs remain explicit.
The retained October 16 compiler evidence is distinguished from the
fresh bounded reproduction completed on October 19. The latter built the
337 original targets and four additional manuscript-source targets, and
checked the original 2145-declaration audit and a separate 51-declaration
audit, using the pinned dependency cache and a source-guarded project
cache. It was not a from-scratch rebuild of every dependency. The results
and logs are supplied in
\nolinkurl{anc/lean/validation/final-review-2026-10-19/}.
The October 20 revision changes no Lean source or mathematical statement;
it verifies the unchanged proof-source bytes and refreshes the article
manifest without claiming another Lean compilation.

% ===== END formal_status.tex =====

\clearpage
\phantomsection
\addcontentsline{toc}{section}{References}

% ===== BEGIN bibliography.tex =====
% Bibliography fragment for literature.tex; the root manuscript may append
% its independently checked geometry and formalization references.
% ===== END bibliography.tex =====

\end{document}